\documentclass[11pt,a4paper]{amsart}
\usepackage[margin=26mm]{geometry}
\usepackage{amsmath,amssymb,mathtools,enumitem}
\usepackage[colorlinks=true,linkcolor=blue,citecolor=blue,urlcolor=blue]{hyperref}
\newtheorem{theorem}{Theorem}
\newtheorem{lemma}[theorem]{Lemma}
\newtheorem{proposition}[theorem]{Proposition}
\newtheorem{corollary}[theorem]{Corollary}
\theoremstyle{definition}
\newtheorem{definition}[theorem]{Definition}
\newtheorem{remark}[theorem]{Remark}
\newcommand{\R}{\mathbb R}
\newcommand{\C}{\mathbb C}
\newcommand{\N}{\mathbb N}
\newcommand{\AP}{\operatorname{AP}}
\newcommand{\Sch}{\mathcal S}
\newcommand{\T}{\mathcal T}
\newcommand{\norm}[1]{\left\|#1\right\|}
\newcommand{\abs}[1]{\left|#1\right|}
\newcommand{\br}[1]{\langle #1\rangle}
\newcommand{\dd}{\,d}
\newcommand{\diag}{\operatorname{diag}}
\DeclareMathOperator{\supp}{supp}
\title{Global flows for the cubic nonlinear Schrödinger equation without decay assumptions}
\author{Shinichi Kotani$^{1}$}
\author{Yiqian Wang$^{2}$}
\author{Jiahao Xu$^{3}$}
\author{Shuo Zhang$^{4}$}
\date{}
\begin{document}
\begin{abstract}
For each positive integer $N$, we construct global spectral group actions
for the defocusing and focusing nonlinear Schr\"odinger hierarchies on
initial-data classes containing $W^{2N+1,\infty}(\mathbb R)$.
For $N\ge2$, the quadratic flow gives global
classical cubic NLS solutions. The construction extends the
Sato--Segal--Wilson and Kotani frameworks to matrix Dirac systems; a
homogeneous tau-function identity supplies the focusing positivity.
No decay, smallness, periodicity, arithmetic condition, or prescribed
spectral background is required for the initial-data inclusion.
As a principal application, we prove global well-posedness on
$M_{\infty,1}^{5}(\mathbb R)$ for both signs, with finite-time bounds and
locally Lipschitz dependence on initial-data norm balls. The passage to this
Banach-space flow uses near-real-axis Weyl asymptotics and continuity of
normalized Toeplitz operators on compact initial-data sets to obtain uniform
spatial bounds from five bounded derivatives. In the defocusing case, the flow extends to
$M_{\infty,1}^{5}(\mathbb R)+H^1(\mathbb R)$.
Further consequences include quantitative approximation, preservation of
spatial Bohr almost periodicity and its frequency module, conjugacy of
spatial translation hulls, and preservation of the full complex Dirac
spectrum and spatial transfer growth rates.
\end{abstract}
\maketitle

\section{Introduction}
\footnotetext[1]{Dept. of Math. \ \ Osaka Univ.
	\ Toyonaka Japan \ skotani@outlook.com.}  \footnotetext[2]{School of
	Math. Nanjing Univ. Nanjing China yiqianw@nju.edu.cn.} \footnotetext[3]{School of
	Math. Nanjing Univ. Nanjing China xjhao512@outlook.com.} \footnotetext[4]{School of Math. Nanjing Univ. Nanjing China
	shuozhang@smail.nju.edu.cn.}
The nonlinear Schr\"odinger equation (NLS) is one of the basic models for the
evolution of slowly modulated dispersive wave packets. It appears in nonlinear
optics, water-wave theory, plasma physics, and Bose--Einstein condensation, and
it exhibits a characteristic competition between dispersion and nonlinearity.
Among its most prominent coherent structures are bright solitons in the
focusing regime and dark solitons on a nonzero background in the defocusing
regime.

We consider the initial value problem for the one-dimensional cubic NLS
equation on the real line,
\begin{equation}
	\begin{cases}
		i\partial_{t}q(t,x) = -\dfrac{1}{2}\partial_{x}^{2}q(t,x)
		+\kappa |q(t,x)|^{2}q(t,x),\\
		q(0,x) = q_{0}(x),
	\end{cases}
	\qquad (t,x)\in\mathbb{R}^{2}, \label{EE2.1}
\end{equation}
where $q$ is complex-valued and $\kappa\in\{1,-1\}$. With this convention,
$\kappa=1$ is the defocusing case and $\kappa=-1$ is the focusing case. These
two reductions have markedly different spectral features. The defocusing
Zakharov--Shabat operator has a self-adjoint structure, whereas its focusing
counterpart is non-self-adjoint and may possess nonreal discrete spectrum and
spectral singularities.

\subsection{Global spectral flows}
Our main construction gives global spectral group actions for both the
defocusing and focusing NLS reductions. For each positive integer $N$, let
$\mathcal Q_{N,d}$ and $\mathcal Q_{N,f}$ denote the spectral initial-data
classes defined in Section~\ref{sec:spectral-setup}, and let
$\Gamma_N^{\mathrm{real}}$ be the real multiplier group defined there.
Both reductions admit the polynomial subgroup
\begin{equation}
 \Gamma_N^{\operatorname{sub}}
 =\{e^{ih};\ h\in\mathbb R[z],\ \deg h\leq N\}
 \subset\Gamma_N^{\mathrm{real}}.
 \label{E1.2}
\end{equation}
The corresponding transformations $\operatorname{dNLS}(g)$ and
$\operatorname{fNLS}(g)$ are obtained by multiplying spectral symbols and
reconstructing the potential; their precise definitions are given in
Section~\ref{sec:spectral-setup}. The following two theorems state the global
construction and its concrete range of initial data.

\begin{theorem}
	\label{t1}
	The transformations are independent of the representing symbol and
	contour. The following families define group actions on the indicated spaces:
	\[
\begin{aligned}
	\{\operatorname{dNLS}(g)\}_{g\in\Gamma_N^{\operatorname{real}}}
	&\quad\text{on }\mathcal Q_{N,d},\\
	\{\operatorname{fNLS}(g)\}_{g\in\Gamma_N^{\operatorname{sub}}}
	&\quad\text{on }\mathcal Q_{N,f}.
	\end{aligned}
\]
	For $\alpha\in\{d,f\}$ and $N\geq2$, set
	\[
q(t,x)=\bigl(\mathop{\alpha\mathrm{NLS}}(e^{itz^{2}})q\bigr)(x),
	\qquad q\in\mathcal Q_{N,\alpha}.
\]
	Then $q(t,x)$ solves
	\[
\begin{cases}
	i\partial_tq(t,x)
	=-\dfrac12\partial_x^2q(t,x)\pm|q(t,x)|^2q(t,x),\\
	q(0,x)=q(x),
	\end{cases}
\]
	where the upper sign corresponds to $\alpha=d$ and the lower sign to
	$\alpha=f$. Moreover, $q(t,\cdot)\in C^N(\mathbb R)$ for each fixed
	$t\in\mathbb R$, and
	$q(\cdot,x)\in C^{\lfloor N/2\rfloor}(\mathbb R)$ for each fixed
	$x\in\mathbb R$. More precisely, all mixed derivatives
	$\partial_x^j\partial_t^kq$ with $j+2k\leq N$ exist and are jointly
	continuous on $\mathbb R^2$. Thus, for $N\geq2$, the solution is classical.
\end{theorem}

The Weyl asymptotics of S. Zhang \cite{SZ}, combined with the spectral
construction below, give the following sufficient condition for membership
in $\mathcal Q_{N,d}$ and $\mathcal Q_{N,f}$.

\begin{theorem}
	\label{t2}
	If $q\in W^{2N+1,\infty}(\mathbb R)$, then
	$q\in\mathcal Q_{N,d}\cap\mathcal Q_{N,f}$.
\end{theorem}

In particular, taking $N=2$ gives a global classical solution for every
$q_0\in W^{5,\infty}(\mathbb R)$, for either sign of the cubic NLS.
The inclusion requires no spatial decay, smallness, periodicity, arithmetic
condition, or prescribed spectral-gap structure. It covers all sufficiently
smooth periodic, quasi-periodic, and Bohr almost-periodic potentials, as well
as general initial data with five bounded derivatives. The next result equips a concrete
subclass with an invariant Banach-space topology and quantitative stability.

\subsection{Global well-posedness in a non-decaying phase space}
The spectral construction has a concrete realization as a global flow on a
Banach space. We choose a modulation space that supports a strongly continuous
Schr\"odinger group and the cubic product estimates needed for local
well-posedness. Uniform bounds for spectral reconstruction then give global
continuation and stability in this space.

Choose a real, even $\chi\in C_c^\infty((-1,1))$ such that
$\sum_{k\in\mathbb Z}\chi(\xi-k)=1$, and put
\[
\Box_k f=\mathcal F^{-1}\bigl(\chi(\xi-k)\widehat f(\xi)\bigr),
 \qquad \br{k}=(1+k^2)^{1/2}.
\]
The Fourier convention is $\widehat f(\xi)=\int e^{-ix\xi}f(x)\dd x$.
\begin{definition}
For $s\ge0$, define the modulation space
\[
X_s=M_{\infty,1}^{s}(\R)
 =\left\{f\in\Sch'(\R):
 \norm f_s:=\sum_{k\in\mathbb Z}\br{k}^s\norm{\Box_k f}_\infty<\infty
 \right\}.
\]
We use the Banach phase space $X=X_5$.
\end{definition}
The block series converges uniformly; in particular its sum is a bounded,
uniformly continuous representative of $f$. Different such partitions
give equivalent norms. The elementary estimates needed below are proved
in Section~\ref{sec:modulation}. The following theorem is a principal
application of the spectral construction.

\begin{theorem}\label{t3}
For either sign in (\ref{EE2.1}) and every $q_0\in X$, there is a global
solution
\[
q\in C(\R;X)\cap C^1(\R;X_3).
\]
It solves (\ref{EE2.1}) classically. Write $\Phi_t(q_0)=q(t)$.
For every $T,R<\infty$ there are finite
constants $F_T(R)$ and $L_T(R)$ such that
\begin{align}
 \sup_{|t|\le T}\norm{q(t)}_5&\le F_T(R)
 &&(\norm{q_0}_5\le R),\label{E1.3}\\
 \sup_{|t|\le T}\norm{\Phi_t(q_0)-\Phi_t(p_0)}_5
 &\le L_T(R)\norm{q_0-p_0}_5
 &&(\norm{q_0}_5,\norm{p_0}_5\le R).\label{EE2.2}
\end{align}
The maps $\Phi_t$ form a continuous group on $X$ and commute with spatial
translations. The solution is unique among classical solutions $v$ such
that $v,v_x$ are bounded on every finite time strip.

If $q_0\in X_s$ for some $s\ge5$, then $q\in C(\R;X_s)$, with bounds
and local Lipschitz dependence in $X_s$ on bounded initial-data sets.
In particular,
\[
q_0\in C_b^\infty(\R)
 \quad\Longrightarrow\quad q\in C^\infty(\R;C_b^\infty(\R)),
\]
with continuous dependence for the Fr\'echet topology of uniform
derivative seminorms.
\end{theorem}

\begin{remark}
The estimates here are finite-time estimates. The compactness argument
gives a finite function of the initial norm, but does not give an explicit
numerical growth law for that function. The five-derivative input
is used to control the reconstruction in the weighted symbol topology
of Proposition~\ref{p37}. The inclusions
\[
W^{7,\infty}(\R)\subset X_5\subset W^{5,\infty}(\R)
\]
will be proved below. The theorem does not assert invariance of
$W^{5,\infty}$.
\end{remark}

The additional estimate connecting Theorems~\ref{t1}--\ref{t2} to
Theorem~\ref{t3} is Proposition~\ref{p37}. For the canonical
spectral solution $Q_\kappa[v]$ with $v\in W^{5,\infty}$, it gives
\[
\sup_{p_5(v)\le R}\sup_{|t|\le T}
 p_2\bigl(Q_\kappa[v](t,\cdot)\bigr)\le B_T(R)<\infty.
\]
where $p_j(f)=\max_{0\le a\le j}\norm{\partial_x^a f}_{\infty}$. Together with $W^{2,\infty}\hookrightarrow X_0$, uniqueness and the local
$X_5$ theory, this bound yields global continuation and the stability
estimates in Theorem~\ref{t3}.

\subsection{Further consequences}
The global Banach-space flow also gives dynamics for finite-energy
perturbations and preserves spatial almost periodicity. We first state the
defocusing extension to a sum space.

\begin{definition}
For $s\ge5$, define the Banach sum
\begin{equation}\label{EE2.3}
 Z_s=X_s+H^1(\R),\qquad
 \norm f_{Z_s}
 =\inf_{f=a+b}\bigl(\norm a_s+\norm b_{H^1}\bigr),
 \quad a\in X_s,\quad b\in H^1(\R).
\end{equation}
\end{definition}

\begin{theorem}
\label{t7}
Let $\kappa=1$ and $s\ge5$. The flow $\Phi_t$ extends to a
continuous group on $Z_s$ and commutes with spatial translations. For every $u_0\in Z_s$, its trajectory
is a global mild solution of \eqref{EE2.1} belonging to
$C(\R;Z_s)$, and is unique among mild solutions in $C(\R;X_0)$.
For every $T,R<\infty$, there are finite constants
$G_{s,T}(R)$ and $K_{s,T}(R)$ such that
\begin{align}
 \sup_{|t|\le T}\norm{\Phi_t(u_0)}_{Z_s}
 &\le G_{s,T}(R),
 &&\norm{u_0}_{Z_s}\le R,\label{E1.6}\\
 \sup_{|t|\le T}
 \norm{\Phi_t(u_0)-\Phi_t(v_0)}_{Z_s}
 &\le K_{s,T}(R)\norm{u_0-v_0}_{Z_s},
 &&\norm{u_0}_{Z_s},\norm{v_0}_{Z_s}\le R.
 \label{E1.7}
\end{align}
For every decomposition $u_0=a+b$ with $a\in X_s$ and
$b\in H^1(\R)$,
\begin{equation}\label{EE2.4}
 t\longmapsto\Phi_t(a+b)-\Phi_t(a)
 \quad\hbox{belongs to } C(\R;H^1(\R)).
\end{equation}
On each finite time interval this map depends locally Lipschitz
continuously on $(a,b)\in X_s\times H^1(\R)$, with bounds on
every initial-data norm ball. In particular,
\[
\Phi_t\bigl(a+H^1(\R)\bigr)=\Phi_t(a)+H^1(\R).
\]
\end{theorem}

\begin{theorem}\label{t8}
For either sign in \eqref{EE2.1}, let $q_0\in X\cap\AP(\R)$,
where $\AP$ denotes Bohr almost periodicity.
Then $q(t)\in X\cap\AP(\R)$ for every $t\in\R$. If
$\Gamma(f)$ is the additive subgroup of $\R$ generated by the frequencies
at which $f$ has a nonzero Bohr Fourier coefficient, then
\[
\Gamma(q(t))=\Gamma(q_0)\qquad(t\in\R).
\]
The map $\Phi_t$ restricts to a homeomorphism between the spatial
translation hulls of $q_0$ and $q(t)$ in $X$, and intertwines translations.
\end{theorem}

Corollary~\ref{c44} gives approximation in $X$ by smooth data
and, for almost-periodic initial values, by trigonometric polynomials,
with solution convergence uniform on compact time intervals.

Cutoff and periodic approximations of non-decaying data need not
converge in the global phase-space norm, motivating a local stability
estimate. Proposition~\ref{p45} and
Corollary~\ref{c46} give convergence of cutoff and
long-period solutions in $H^1$ and the uniform norm on fixed spatial
windows, with error $O(L^{-1/2})$ on finite time intervals under a
uniform $W^{5,\infty}$ bound on the initial data. This identifies the
general non-decaying evolution as a local limit of both finite-mass
and periodic dynamics.

For arbitrary $X$ data, spatial translation orbits need not be
precompact in the global norm, so we consider their closures in
$C^4_{\mathrm{loc}}$. Propositions~\ref{p47}
and~\ref{p48} establish continuity of the flow in this
local topology on bounded $X$ sets and show that the resulting compact
translation hulls are mapped homeomorphically onto their evolved
hulls, with translations intertwined. This extends the hull dynamics
of Theorem~\ref{t8} to general non-decaying profiles and allows
their spatially invariant distributions to be transported by the flow.

The finite-time spatial bounds also allow the Lax compatibility
relations to act through bounded invertible transformations on the
full line. Theorem~\ref{t49} consequently gives
preservation of the full complex Dirac spectrum and the upper and
lower spatial transfer growth rates for both reductions, while
Corollary~\ref{c50} proves existence and
preservation of averaged transfer growth under the corresponding
translation-invariant hull measures. These results describe the
spectral and spatial statistical invariants of the constructed global
flow without requiring almost periodicity or ergodicity.

\subsection{Background and comparison}
Equation (\ref{EE2.1}) is completely integrable. More precisely, it is
equivalent to the Lax equation
\[
\partial_{t}D=[A,D],
\]
where
\[
D=i\sigma_3\partial_{x}+Q(x,t)
\]
and
\[
A=i\sigma_3\partial_{x}^{2}+Q\partial_{x}
  +\frac12\partial_{x}Q-\frac{i}{2}\sigma_3Q^2,
\]
with
\[
\sigma_3=\begin{pmatrix}1&0\\0&-1\end{pmatrix}, \qquad
	Q(x,t)=\begin{pmatrix}0&q(t,x)\\
	\kappa\overline{q(t,x)}&0\end{pmatrix}.
\]
The spectral equation $D\boldsymbol{f}=z\boldsymbol{f}$ is the
Zakharov--Shabat system. The inverse scattering solution of the focusing
equation was introduced by Zakharov and Shabat \cite{ZS}. Rigorous direct and
inverse scattering theories for first-order systems were subsequently
developed by Beals and Coifman \cite{BC}; Zhou's formulation also permits
arbitrary spectral singularities \cite{Zhou}.

For rapidly decreasing initial data, the direct scattering map replaces
$q_0$ by a reflection coefficient together with discrete eigenvalues and
norming constants when discrete spectrum is present. The scattering data have
an elementary time evolution, and the solution is reconstructed from a
Marchenko equation or a Riemann--Hilbert problem. This framework gives both
pure soliton solutions and the interaction of solitons with radiation. In the
defocusing case, Deift and Zhou obtained long-time asymptotics for data in a
weighted Sobolev space \cite{DZ}. For focusing NLS, Borghese, Jenkins, and
McLaughlin derived cone-wise asymptotics whose leading term is a modulated
multi-soliton \cite{BJM}. Thus the rapidly decreasing theory describes in
considerable detail the soliton and dispersive components of the flow, but its
scattering coordinates are not directly available for general bounded,
non-decaying functions.

Inverse scattering has also been adapted to nonzero backgrounds. For focusing
NLS with constant boundary values, see Biondini and Kova\v{c}i\v{c}
\cite{BK}. For defocusing finite-density data, Cuccagna and Jenkins proved
long-time asymptotics and asymptotic stability of $N$-dark-soliton solutions
\cite{CJ}. These theories
are particularly relevant to coherent structures, but the prescribed
asymptotic background remains an essential part of their spectral
formulation.

There is, in parallel, a substantial PDE theory which does not use the
scattering transform. In one dimension the cubic equation is mass
subcritical, and the classical $L^2(\mathbb R)$ theory gives global solutions
for both signs; see, for example, Tsutsumi \cite{Tsutsumi}. Bourgain's Fourier
restriction method established a basic low-regularity theory for periodic NLS
\cite{Bourgain93}, while his invariant-measure construction provided global
dynamics on sets of full Gibbs measure \cite{Bourgain94}. At negative
regularity, Koch and Tataru obtained a priori estimates in $H^s(\mathbb R)$
\cite{KT}, and Killip, Vi{\c{s}}an, and Zhang constructed low-regularity
conservation laws simultaneously on the line and the circle \cite{KVZ}.
Building on this integrable-PDE viewpoint, Harrop-Griffiths, Killip, and
Vi{\c{s}}an proved sharp global well-posedness for both focusing and
defocusing cubic NLS in
\[
H^s(\mathbb R),\qquad s>-\frac12
\]
\cite{HG}. Related global theories in modulation and Fourier--Lebesgue spaces
were obtained by Oh and Wang \cite{OW}. Randomization gives another extension
of deterministic theory: Colliander and Oh proved almost-sure well-posedness
below $L^2(\mathbb T)$ for the Wick-ordered periodic cubic NLS equation \cite{CO}. These results
are sharp or nearly sharp in their respective finite-mass and periodic
settings, but a general bounded function on $\mathbb R$ need not belong to any
of the underlying Sobolev or modulation spaces.

Periodic initial data form a second classical integrable regime. The
finite-gap method produces explicit theta-function solutions and a
constructive solution of the periodic Cauchy problem for finite-gap data; see
Kotlyarov and Its \cite{KI}. McKean and Vaninsky developed action--angle
variables for the periodic cubic Schr\"odinger equation \cite{MV}. For the
defocusing equation this leads to a global integrable normal-form picture,
whereas in the focusing case the non-self-adjoint spectrum makes the geometry
more delicate; local Birkhoff coordinates near the zero solution were
constructed by Kappeler, Lohrmann, Topalov, and Zung \cite{KLTZ}. Periodic
inverse spectral methods and the PDE methods mentioned above are
complementary, but both make essential use of the compact spatial geometry or
of periodic spectral data.

Quasi-periodicity occurs in two distinct senses that should not be
conflated. KAM theory for infinite-dimensional systems constructs specially selected
solutions which are quasi-periodic in time on a compact spatial domain; among
the foundational NLS results are those of Kuksin and P\"oschel \cite{KP} and
Eliasson and Kuksin \cite{EK}. This is different from prescribing a spatially
quasi-periodic function on the real line and solving its Cauchy problem.
For the latter problem, Damanik, Li, and Xu proved local existence and
uniqueness for the standard NLS under a power-law upper bound on the Fourier
coefficients \cite{DD1}; their companion work treats derivative NLS
\cite{DD2}. More recently, Fillman, Li, Luki\'c, and Zhou used direct and
inverse Dirac spectral theory to obtain defocusing cubic NLS solutions that
are almost periodic in both space and time for small analytic quasi-periodic
data satisfying Diophantine conditions \cite{FLLZ}.

The almost-periodic Cauchy problem is still less amenable to the usual
compactness and dispersive tools. Oh proved local well-posedness in the
Wiener-type algebra of almost-periodic functions with absolutely convergent
Fourier series \cite{OhAP}, and established global existence for a regular
class of limit-periodic data in the defocusing case \cite{OhLP}. From the
inverse spectral side, Boutet de Monvel and Egorova constructed NLS solutions
with Cantor-type spectrum \cite{BME}, while Boutet de Monvel and Marchenko
developed a spectral construction for a class of bounded initial data obtained
from limits of reflectionless Dirac operators \cite{BMM}. These results show
that non-decaying and almost-periodic dynamics can be treated beyond the
finite-gap setting, but the available constructions typically impose
Fourier-summability, approximation, smallness, arithmetic, or spectral
conditions.

For general bounded smooth data, global existence on the line remains
a separate problem from the finite-mass theory; see Dodson--Soffer--Spencer
\cite{DSS}. Klaus \cite{Klaus} singles out a global theory in a modulation
space with spatial exponent $\infty$ as an open endpoint problem.
Zhao \cite{ZhaoBounded} establishes local well-posedness for bounded data,
including local spatial almost-periodicity and frequency-module control.
Theorems~\ref{t1} and~\ref{t2}, with $N=2$, give global classical
solutions for arbitrary $W^{5,\infty}$ initial data.
Theorem~\ref{t3} realizes this construction as a globally well-posed
flow on $M_{\infty,1}^{5}$ and, in particular, on $C_b^\infty$.

Finite-energy perturbations of non-decaying backgrounds provide another
natural class of initial data. Klaus and Kunstmann \cite{KlausKunstmann}
proved global well-posedness for the defocusing cubic NLS in
$H^1(\mathbb R)+H^r(\mathbb T)$, $r>3/2$, using a time-dependent
relative Hamiltonian. Theorem~\ref{t7} treats arbitrary
backgrounds in $M_{\infty,1}^{5}(\mathbb R)$ and their
$H^1(\mathbb R)$ perturbations, with a flow defined independently of the
decomposition and locally Lipschitz in the sum norm. This removes the
periodic background assumption while requiring higher regularity of
the background.

\subsection{Proof strategy and organization}
Our construction of global NLS evolutions is based on the
Sato--Segal--Wilson theory.
Sato realized integrable hierarchies as flows on an infinite-dimensional
Grassmann manifold \cite{MS}, and Segal and Wilson gave an analytic
Hardy-space formulation in which the hierarchy is encoded by multiplication
operators and tau-functions \cite{GS}. Kotani recently adapted this framework
to construct KdV flows for general non-decaying potentials, including a large
class of almost-periodic potentials \cite{SK} \cite{Kotani2024}. We extend this construction
from the scalar Schr\"odinger operator underlying KdV to the matrix Dirac
operator underlying NLS.

Two issues are specific to the present matrix problem. First, the same
potential may be represented by different Grassmannian symbols, so the group
action must be shown to be independent of this choice. We compute explicitly
the one-factor transformation of the two projective $m$-coordinates, iterate
it for rational group elements, and pass to arbitrary group elements by
rational approximation. Second, the defocusing argument has access to
self-adjoint Weyl--Titchmarsh theory, whereas focusing NLS does not. We
therefore compare the projective $m$-coordinates furnished by the Sato
factorization with the finite-interval Weyl disks of the skew-self-adjoint
Dirac system.  Direct uniqueness follows from the exterior Hardy
normalization, while inverse uniqueness is reduced on every compact interval
to the local Borg--Marchenko theorem.

The non-vanishing of the tau-function makes these spectral transformations
global. In the focusing case, a homogeneous determinant identity yields the
positivity needed for Toeplitz invertibility. Near-real-axis Weyl asymptotics
then place every $W^{2N+1,\infty}$ potential in the spectral classes of
Theorem~\ref{t2}.

To obtain Theorem~\ref{t3}, we establish a further uniformity property
of the reconstruction. On a $W^{5,\infty}$ initial-data ball, a common
spectral contour and a weighted symbol norm make the normalized Toeplitz
operators continuous in the local topology of the potentials. Compactness
and invertibility give uniform bounds for their inverses. Translation
covariance converts the resulting bounds at one spatial point into the
$W^{2,\infty}$ estimate of Proposition~\ref{p37}. We identify the
local $X_5$ solution with the spectral solution by uniqueness, and use the
low-order bound in a tame product estimate to continue it globally.

Section~\ref{sec:group-action} develops the spectral data, tau-functions and
group transformations. Sections~\ref{sec:defocusing-construction}
and~\ref{sec:focusing-construction} prove Theorems~\ref{t1} and~\ref{t2}
for the two reductions. Section~\ref{sec:phase} establishes the uniform
spectral estimates, proves global well-posedness, and derives the further
dynamical consequences. The defocusing sum-space theorem combines the
resulting background bounds with a coercive relative energy for
$H^1(\mathbb R)$ perturbations.

\section{Spectral construction and group action}\label{sec:group-action}

\subsection{Admissible contours and spectral data}\label{sec:spectral-setup}
In the spectral construction, fix a positive integer $N$. Let
\[
C=\{x+i\omega(x):x\in\mathbb{R}\}\cup\{x-i\omega(x):x\in\mathbb{R}\},
\]
where $\omega$ is a smooth positive function on $\mathbb{R}$ satisfying
\begin{equation}
	\begin{array}
		[c]{l}%
		\omega'\in L^{\infty},\qquad
		c_0(1+|x|)^{-(N-1)}\leq \omega(x)\leq C_0(1+|x|)^{-(N-1)}
		\quad (x\in\mathbb R)%
	\end{array}
	\text{.} \label{E2.1}
\end{equation}
Here and below, the positive constants $c_0$ and $C_0$ may depend on the
contour. Compact modifications of such a contour are allowed. We call every
contour satisfying (\ref{E2.1}) \textbf{admissible}. This contour will be used
throughout. We denote
\[
D_{+}=\{z\in\mathbb{C}:|\operatorname{Im}z|<\omega(\operatorname{Re}z)\},
\qquad
D_{-}=\{z\in\mathbb{C}:|\operatorname{Im}z|>\omega(\operatorname{Re}z)\}.
\]
In the rational closures below, only rational functions with boundary
restrictions in $L^2(C)$ are included. The basic Hilbert space $L^{2}(C)$
is the direct sum of the closed subspaces
$H(D_{\pm})$:
\[
\begin{array}
	[c]{ll}%
	H(D_{+})=L^{2}(C)\text{-closure of }\left\{  \text{all rational functions
	with no poles in }D_{+} \right\} \\
	H(D_{-})=L^{2}(C)\text{-closure of }\left\{  \begin{array}
		[c]{l}%
		\text{all rational functions $r$ with no poles} \\
		\text{in $D_{-}$ and } r(z) = o(1) \text{ as } z \to \infty
	\end{array} \right\}%
\end{array} \text{.}
\]
The projections $\mathfrak{p}_{\pm}$ from $L^{2}(C)$ onto $H(D_{\pm})$ are
obtained by
\[
\begin{array}
	[c]{ll}%
	\left(\mathfrak{p}_{+}f\right)(z)
	=\dfrac{1}{2\pi i}{\displaystyle\int_{C}}
	\dfrac{f(\lambda)}{\lambda-z}\,d\lambda
	& \text{for }z\in D_{+},\\
	\left(\mathfrak{p}_{-}f\right)(z)
	=\dfrac{1}{2\pi i}{\displaystyle\int_{C}}
	\dfrac{f(\lambda)}{z-\lambda}\,d\lambda
	& \text{for }z\in D_{-},
\end{array}
\]
The lower graph is oriented from left to right and the upper graph
from right to left, as the positively oriented boundary of the strip. The operators
$\mathfrak{p}_{\pm}$ are bounded on $L^{2}(C)$ (see~\cite{JV}) and satisfy
\[
\mathfrak{p}_{\pm}^{2}=\mathfrak{p}_{\pm},\qquad
\mathfrak{p}_{+}+\mathfrak{p}_{-}=I_{L^{2}(C)}.
\]
We use boldface for the vector-valued spaces:
\[
\boldsymbol{L}^{2}(C)=L^{2}(C)\times L^{2}(C),\qquad
\boldsymbol{H}(D_{\pm})=H(D_{\pm})\times H(D_{\pm}),
\]
and use the same notation $\mathfrak{p}_{\pm}$ for the componentwise
projections from $\boldsymbol{L}^{2}(C)$ onto
$\boldsymbol{H}(D_{\pm})$. They are bounded and satisfy
\[
\mathfrak{p}_{\pm}^{2}=\mathfrak{p}_{\pm},\qquad
\mathfrak{p}_{+}+\mathfrak{p}_{-}=I_{\boldsymbol L^{2}(C)}.
\]

We enlarge $H(D_{+})$ so as to include polynomials. For
$n\in\mathbb{Z}_{\geq1}$, set
\[
H_{n}(D_{+})=(z-b)^{n}H(D_{+}),
\]
where $b\in D_{-}$, and equip $H_{n}(D_{+})$ with the norm
\[
\left\Vert u \right\Vert_{n}
=\left(\displaystyle\int_{C}|u(\lambda)|^{2}|\lambda|^{-2n}
|d\lambda|\right)^{\frac12}.
\]
The space $H_{n}(D_{+})$ is independent of the choice of $b$, with
equivalent norms, and $z^{m}\in H_{n}(D_{+})$ whenever
$0\leq m\leq n-1$. Its vector-valued analogue is denoted by
$\boldsymbol{H}_{n}(D_{+})$.

Multiplication by $z$ acts between successive weighted spaces:
\[
z:\boldsymbol H_n(D_+)\longrightarrow\boldsymbol H_{n+1}(D_+).
\]
We use this scale and the following weighted Toeplitz operators to realize
the Sato--Segal--Wilson construction.

For $L\in\mathbb Z_{+}$, define the symbol class on the fixed contour $C$ by
\[
\mathcal{A}_{L}(C)=\left\{
\begin{array}
	[c]{c}%
	A(\lambda)\in GL(2,\mathbb C)\text{; }A(\lambda)
	\text{ is bounded on }C\text{ and there exists}\\
	\text{an analytic $2\times2$ matrix function $F$ on $D_{+}$ such that}\\
	\sup\limits_{z\in\overline D_{+}}\|F(z)\|<\infty,\qquad
	\sup\limits_{z\in C}\|z^{L}(A(z)-F(z))\|<\infty
\end{array}
\right\}.
\]
For $A\in\mathcal{A}_{L}(C)$, define the Toeplitz operator
\begin{equation}
	\left(T(A)\boldsymbol u\right)(z)
	=F(z)\boldsymbol u(z)
	+\mathfrak p_{+}\left((A-F)\boldsymbol u\right)(z),
	\qquad \boldsymbol u\in\boldsymbol H_{n}(D_{+}), \label{E2.2}
\end{equation}
for $n\leq L$. The definition is independent of the choice of $F$, and
$T(A)$ is bounded on $\boldsymbol H_{n}(D_{+})$.

All subsequent constructions use these operators on the indicated
weighted spaces.

Sato introduced the tau-function to represent flows on the Grassmannian
\cite{MS}. In the present setting, define
\begin{equation}
\mathcal A_N^{inv}(C)=
\left\{A\in\mathcal A_{N+2}(C);\
\begin{array}{c}
T(A)\text{ is invertible on }\boldsymbol H_n(D_+)\\
\text{for }1\leq n\leq N+1
\end{array}\right\}, \label{E2.3}
\end{equation}
The two extra orders of symbol decay ensure that the first asymptotic
coefficient can be differentiated through total spectral degree $N$;
see Lemma~\ref{l15}. Only the weighted spaces up to
$\boldsymbol H_{N+1}(D_+)$ are required for these derivatives.
The contour in (\ref{E2.1}) is unchanged.
We introduce the multiplier group $\Gamma_m(C)$ by
\[
\Gamma_m(C)=\left\{
\begin{array}
	[c]{c}%
	g=re^{ih}\text{; }h\in\mathbb R[z],\ \deg h\le m,\\
	r\text{ is rational, has no zeros or poles in }\overline D_+,\\
	\text{and both }r\text{ and }r^{-1}\text{ are bounded on }C
\end{array}
\right\}.
\]
Every $g\in\Gamma_N$ is bounded on $C$. For $g\in\Gamma_m(C)$, set
\[
G(g)=\begin{pmatrix}g&0\\0&g^{-1}\end{pmatrix}.
\]
Let $m\leq N$. For $A\in\mathcal A_N^{inv}(C)$ and
$g\in\Gamma_m$, multiplication by $G(g)$ defines an action on the symbols. As shown below,
\[
G(g)^{-1}T(G(g)A)T(A)^{-1}-I
\]
is Hilbert--Schmidt on $\boldsymbol H_n(D_+)$ for $1\leq n\leq N+1$. We may
therefore define the \textbf{tau-function} by
\[
\tau_A^{(2)}(g)=\operatorname{det}_{2}
\left(G(g)^{-1}T(G(g)A)T(A)^{-1}\right),
\]
where $\operatorname{det}_{2}$ denotes the modified Fredholm determinant
\cite{BS}. Lemma~\ref{l10} shows that $\tau_A^{(2)}(g)$ is independent
of the choice of $\boldsymbol H_n(D_+)$.

Let $r$ be a rational function with distinct poles
$\{\zeta_j\}_{j=1}^{k}$ and distinct zeros $\{\eta_j\}_{j=1}^{k}$ in
$D_-$. Then $r\in\Gamma_0(C)$ and
\[
G(r)^{-1}T(G(r)A)T(A)^{-1}-I
\]
is finite rank and hence trace class. Thus
\[
\tau_A(r)=\det\left(G(r)^{-1}T(G(r)A)T(A)^{-1}\right)
\]
is well defined. The ordinary and modified determinants are related by
\[
\operatorname{det}_{2}(I+A)=\det(I+A)e^{-\operatorname{tr}A}
\]
for a trace-class operator $A$.

To construct an NLS flow on a subclass of $\mathcal A_N^{inv}(C)$, it is
essential to ensure that $G(g)A\in\mathcal A_N^{inv}(C)$; this is equivalent
to $\tau_A^{(2)}(g)\neq0$. The main issue is therefore to identify classes of
symbols for which the tau-function does not vanish along the relevant group
action. The defocusing reality reduction is expressed by
$\sigma_2^{-1}\overline A\sigma_2=A$, where
$\overline f(z)=\overline{f(\overline z)}$ and
\[
\sigma_2=\begin{pmatrix}0&1\\1&0\end{pmatrix},\qquad
\sigma_4=\begin{pmatrix}0&1\\-1&0\end{pmatrix}.
\]
Then $\tau_A^{(2)}(g)\in\mathbb R$ whenever
$\sigma_2^{-1}\overline A\sigma_2=A$ and
$g\in\Gamma_m^{\mathrm{real}}(C)$, where
\[
\Gamma_m^{\mathrm{real}}(C)
=\{g\in\Gamma_m(C);\ g(z)\overline g(z)=1\}.
\]

A positivity condition on the symbols will be used to prove non-vanishing.
For $\zeta\in\mathbb C\setminus\mathbb R$, set
\[
r_\zeta(z)=\frac{\zeta}{\overline\zeta}
\frac{z-\overline\zeta}{z-\zeta}.
\]
Thus $r_\zeta$ has a pole at $\zeta$, a zero at $\overline\zeta$, and
$r_{\overline\zeta}=r_\zeta^{-1}$. For $N\geq1$, define
\begin{equation}
\left\{
\begin{array}
	[c]{ll}%
	\Gamma_m^{\mathrm{real}}
	=\{g;\ g\in\Gamma_m^{\mathrm{real}}(C)
	\text{ for some curve $C$ satisfying (\ref{E2.1})}\},\\
	\mathcal A_{N,d}^{inv}
	=\left\{A;\begin{array}{ll}
	\text{(i) }\sigma_2^{-1}\overline A\sigma_2=A,\\
	\text{(ii) for some admissible $C$, $A$ satisfies the}\\
	\quad\text{analytic-contour condition below, and}\\
	\quad A|_{\widetilde C}\in\mathcal A_N^{inv}(\widetilde C)
	\text{ for every admissible}\\
	\quad \widetilde C\subset\overline D_+,\\
	\text{(iii) }\tau_A(r_\zeta)\geq0
	\text{ for every }\zeta\in\mathbb C\setminus\mathbb R
	\end{array}\right\},\\
	\mathcal A_{N,f}^{inv}
	=\left\{A;\begin{array}{ll}
	\text{(i) }\sigma_4^{-1}\overline A\sigma_4=A,\\
	\text{(ii) }A\in\mathcal A_N^{inv}(C)\text{ for some $C$}
	\end{array}\right\}\text{.}
\end{array}
\right. \label{E2.4}
\end{equation}
The analytic-contour condition in the defocusing definition means that
$A$ is a single holomorphic $GL(2,\mathbb C)$-valued function on
$\overline D_+\setminus\mathbb R$ (holomorphic in a neighborhood of each
point), and that there is one bounded analytic matrix function $F$ on
$D_+$ such that, for every admissible $\widetilde C\subset\overline D_+$,
\begin{equation}
 \sup_{z\in\overline D_+}\|F(z)\|<\infty,\qquad
 \sup_{z\in\overline D_+\setminus\widetilde D_+}
       \|z^{N+2}(A(z)-F(z))\|<\infty.
 \label{E2.5}
\end{equation}

Recall the polynomial subgroup $\Gamma_N^{\operatorname{sub}}$ from
\eqref{E1.2}.
Every element of $\Gamma_N^{\operatorname{sub}}$ belongs to
$\Gamma_N^{\mathrm{real}}(C)$ for every admissible $C$: indeed,
$|\operatorname{Im}h(x+iy)|
 \leq C_h|y|(1+|x|+|y|)^{N-1}$ is bounded on $\overline D_+$.
We shall prove that $\tau_A^{(2)}(g)>0$ for
$A\in\mathcal A_{N,d}^{inv}$ and $g\in\Gamma_N^{\mathrm{real}}$,
and, on a fixed focusing contour, for
$A\in\mathcal A_{N,f}^{inv}$ and $g\in\Gamma_N^{\mathrm{real}}(C)$.
In particular, all $g\in\Gamma_N^{\operatorname{sub}}$ are allowed in the focusing case.
This guarantees the required Toeplitz invertibility. Set
$\Phi_A=AT(A)^{-1}I-I$. This is a $2\times2$ matrix-valued holomorphic
function on $D_-$, normalized at high imaginary infinity, with the
weighted Hardy expansion
\[
\Phi_A(z)=\sum_{k=1}^{N+1}A_kz^{-k}+\Psi_A^{(0)}(z)z^{-N-1},
\qquad
A_1=\begin{pmatrix}
\alpha_{11}(A)&\alpha_{12}(A)\\
\alpha_{21}(A)&\alpha_{22}(A)
\end{pmatrix}.
\]
Put $e_x(z)=e^{ixz}$ for $x\in\mathbb R$ and define the spaces of initial data by
\[
\left\{
\begin{array}{ll}
\mathcal Q_{N,d}
=\{q(x);\ q(x)=2\alpha_{21}(G(e_x)A)
\text{ for some }A\in\mathcal A_{N,d}^{inv}\},\\
\mathcal Q_{N,f}
=\{q(x);\ q(x)=2\alpha_{21}(G(e_x)A)
\text{ for some }A\in\mathcal A_{N,f}^{inv}\},
\end{array}
\right.
\]
and, with $g\in\Gamma_N^{\mathrm{real}}$ in the defocusing case and
$g\in\Gamma_N^{\operatorname{sub}}$ in the focusing case, define
\[
\left\{
\begin{array}{ll}
(\operatorname{dNLS}(g)q)(x)
=2\alpha_{21}(G(ge_x)A)
\quad\text{for }A\in\mathcal A_{N,d}^{inv},\\
(\operatorname{fNLS}(g)q)(x)
=2\alpha_{21}(G(ge_x)A)
\quad\text{for }A\in\mathcal A_{N,f}^{inv}.
\end{array}
\right.
\]

Theorem~\ref{t1} asserts that these transformations descend to the
potential classes independently of their spectral representation. We now
develop the identities used to prove this assertion and the inclusion in
Theorem~\ref{t2}.

\subsection{Tau-function}\label{sec:tau-function}
The invertibility of $T(G(g)A)$ will be established through the tau-function associated with
\[
G(g)^{-1} T(G(g)A) T(A)^{-1} \text{.}
\]
We first verify that the relevant perturbation of the identity is Hilbert--Schmidt.

\begin{lemma}
	\label{l9} For $A \in \mathcal{A}_{N}^{inv}\left(  C\right)$ and $g \in \Gamma_N(C)$, $G(g)^{-1}T(G(g)A) - T(A)$ is of Hilbert--Schmidt class on $\boldsymbol{H}_{n}(D_{+})$ for $1 \leq n \leq N+1$.
\end{lemma}

\begin{proof}
    Put $E=A-F$. We use the decay $E(\lambda)=O(|\lambda|^{-N-2})$
    supplied by (\ref{E2.3}), with the same contour parameter $N$.
	(\ref{E2.2}) implies
	\begin{align*}
		&T(G(g)A) \boldsymbol{u}(z) = \left( G(g)F\boldsymbol{u} \right)(z) + \dfrac{1}{2\pi i} \displaystyle\int_{C} \dfrac{G(g)(\lambda) \left( A(\lambda)-F(\lambda) \right) \boldsymbol{u}(\lambda)}{\lambda-z} d\lambda \\
		& = \left( G(g)T(A)\boldsymbol{u} \right)(z) + \dfrac{1}{2\pi i} \displaystyle\int_{C} \dfrac{ \left( G(g)(\lambda)-G(g)(z) \right) \left( A(\lambda)-F(\lambda) \right) \boldsymbol{u}(\lambda)}{\lambda-z} d\lambda %
	\end{align*}
	for $\boldsymbol{u} \in \boldsymbol{H}_{n}(D_{+})$. Then $G(g)^{-1}T(G(g)A) - T(A)$ is of Hilbert--Schmidt class on $\boldsymbol{H}_{n}(D_{+})$ if
	\[
\Delta \equiv \displaystyle\int_{C^{2}} \left\Vert G(g)(z)^{-1} \dfrac{G(g)(\lambda) - G(g)(z)}{\lambda-z} \right\Vert^{2} |\lambda|^{-2N+2n-4} |z|^{-2n} |d\lambda||dz| < \infty.
\]
	Choose a neighborhood $U$ as in \eqref{E2.11} on which
$g,g^{-1}$ are bounded and analytic.
Such a choice is possible: outside a fixed compact set, all chords
joining relatively close points of $C$ lie in a strip
$|\operatorname{Im}z|\le C\langle\operatorname{Re}z\rangle^{-(N-1)}$.
The real polynomial in $g=re^{ih}$ has bounded imaginary part
there. On the compact portion, decrease the relative-distance
constant in \eqref{E2.11} and choose the neighborhood to avoid the
finitely many zeros and poles of $r$.
The same choice gives
$|g'|+|(g^{-1})'|\le C\langle z\rangle^{N-1}$ on $U$.
Since
	\[
\dfrac{g(\lambda)-g(z)}{\lambda-z} = \int_{0}^{1} g'((\lambda-z)t+z) dt
\]
	for $\lambda,z\in C$ with $|z-\lambda|<\varepsilon|\lambda|$, and since
	\[
g'(z)g(z)^{-1},\; (g^{-1})'(z)g(z)=O(|z|^{N-1})
\]
	on $U$, the difference quotient in (\ref{E2.2}) is $O(|\lambda|^{N-1})$ in this region. Moreover, $|z|\asymp|\lambda|$. Hence
	\begin{align*}
		\Delta_{1} & \equiv \displaystyle\int_{|z-\lambda|<\varepsilon |\lambda|} \left\Vert G(g)(z)^{-1} \dfrac{G(g)(\lambda) - G(g)(z)}{\lambda-z} \right\Vert^{2} |\lambda|^{-2N+2n-4} |z|^{-2n} |d\lambda||dz| \\
		& \leq C \displaystyle\int_{\substack{(z,\lambda)\in C^{2}\\|z-\lambda|<\varepsilon|\lambda|}} |\lambda|^{2n-6}|z|^{-2n}|d\lambda||dz| < \infty. %
	\end{align*}
	Indeed, the integrand in the last expression is $O(|\lambda|^{-6})$, while the length of $\{z\in C:|z-\lambda|<\varepsilon|\lambda|\}$ is $O(|\lambda|)$; the bounded part of $C$ causes no difficulty. On the complementary region, boundedness of $G(g)$ and $G(g)^{-1}$ on $C$ gives
	\begin{align*}
		\Delta_{2} & \equiv \displaystyle\int_{|z-\lambda| \geq \varepsilon |\lambda|} \left\Vert G(g)(z)^{-1} \dfrac{G(g)(\lambda) - G(g)(z)}{\lambda-z} \right\Vert^{2} |\lambda|^{-2N+2n-4} |z|^{-2n} |d\lambda||dz| \\
		& \leq C \displaystyle\int_{C^{2}} |\lambda|^{-2N+2n-6} |z|^{-2n} |d\lambda||dz| < \infty. %
	\end{align*}
	Because $n\leq N+1$, both factors in the last integral are integrable at infinity. Consequently,
	\[
\Delta = \Delta_{1} + \Delta_{2} < \infty,
\]
	which proves the assertion. \bigskip
\end{proof}

We shall also use the following invariance result.

\begin{lemma}\label{l10}
Let $A\in\mathcal A_N^{inv}(C)$ and
$g\in\Gamma_N(C)$. For $1\le n\le N+1$, set
\[
K_n=\bigl(G(g)^{-1}T(G(g)A)-T(A)\bigr)T(A)^{-1}
 \quad\text{on }\boldsymbol H_n(D_+).
\]
Then $\det_2(I+K_n)$ is independent of $n$.
If $g$ is rational, $\det(I+K_n)$ and $\operatorname{tr}K_n$
are also independent of $n$.
\end{lemma}

\begin{proof}
Write $G=G(g)$, $E=A-F$, and
$\boldsymbol H_0=\boldsymbol H(D_+)$. Since $E\boldsymbol u\in
\boldsymbol L^2(C)$ for $\boldsymbol u\in\boldsymbol H_n(D_+)$,
\[
\bigl(G^{-1}T(GA)-T(A)\bigr)\boldsymbol u
 =G^{-1}\mathfrak p_+(GE\boldsymbol u)
   -\mathfrak p_+(E\boldsymbol u)
 \in\boldsymbol H_0.
\]
Here $G^{\pm1}$ are bounded analytic multipliers on $D_+$.
Both terms define bounded maps $\boldsymbol H_n\to\boldsymbol H_0$.
The Toeplitz operators agree on the nested weighted spaces, and their
inverses agree there by uniqueness. Thus the operators $K_n$ agree on
intersections and map every $\boldsymbol H_n$ into $\boldsymbol H_0$.

For $\lambda\ne0$, every generalized $\lambda$-eigenvector lies in
$\boldsymbol H_0$. Indeed, the assertion for an eigenvector follows
from $\boldsymbol u=\lambda^{-1}K_n\boldsymbol u$; induction follows
from
\[
\boldsymbol u=\lambda^{-1}
 \bigl(K_n\boldsymbol u-(K_n-\lambda)\boldsymbol u\bigr).
\]
Consequently all $K_n$ have the same nonzero eigenvalues with the same
algebraic multiplicities. Lemma~\ref{l9} and the canonical product
formula for $\det_2$ prove the first assertion. For rational $g$ the
corrections are finite rank, and the ordinary determinant and trace
are respectively the product of $1+\lambda$ and the sum of $\lambda$
over this same nonzero eigenvalue list, counted with algebraic
multiplicity. This proves the remaining assertions.
\end{proof}

It follows that the tau-function
\[
\tau^{(2)}_{A}(g) = \text{det}_{2} \left(  G(g)^{-1}T\left(
G(g)A \right)  T\left(  A \right)^{-1} \right)
\]
is well defined independently of $n$. We next record its basic properties. Since $g$ is analytic and bounded on $D_{+}$,
\[
A \in \mathcal{A}_{N}^{inv}\left(  C\right) \text{, } g \in \Gamma_{N}^{\text{real}}(C) \Rightarrow G(g)A \in \mathcal{A}_{N+2}\left(  C\right) \text{,}
\]
In what follows, assume that
\[
A \in \mathcal{A}_{N}^{inv}\left(  C\right) \text{, } g_{1} \text{, }g_{2} \in \Gamma_{N}^{\text{real}}(C) \text{.}
\]
If $G(g_{1})A\in\mathcal{A}_{N}^{inv}(C)$, all three tau-functions appearing in the cocycle identity below are defined.

For notational convenience, set
\begin{equation}
	E_{A}(g_{1},g_{2}) = \text{tr} \left( (G(g_{1})G(g_{2}))^{-1} R_{A}(g_{1},g_{2}) T(G(g_{1})A)^{-1} R_{A}(1,g_{1}) T(A)^{-1} \right) \text{,} \label{E2.6} %
\end{equation}
where%
\begin{equation}
	R_{A}(g_{1},g_{2})  = T\left( G(g_{2}) G(g_{1}) A \right) - G(g_{2}) T\left( G(g_{1}) A \right) \text{.} \label{E2.7}
\end{equation}

\begin{lemma}
	\label{l11} Assume $1\leq n\leq N+1$. Then the following statements hold.
	\begin{enumerate}
	\item[(i)] The map $T(G(g_{1})A)$ is bijective on $\boldsymbol{H}_{n}(D_{+})$ if and only if $\tau^{(2)}_{A}(g_{1}) \neq 0$.
	\item[(ii)] If $\tau_{A}^{(2)}(g_{1}) \neq 0$, then%
	\begin{equation}
		\tau^{(2)}_{A}(g_{1}g_{2}) = \tau^{(2)}_{A}(g_{1}) \tau^{(2)}_{G(g_{1})A}(g_{2}) \exp \left( -E_{A}(g_{1},g_{2}) \right) \text{.} \label{E2.8}
	\end{equation}
	\end{enumerate}
\end{lemma}

\begin{proof}
	For notational convenience, set
	\[
G_{1} = G(g_{1}) \text{, } G_{2} = G(g_{2}) \text{.}
\]
	The identity (\ref{E2.7}) implies %
	\[
G_{2}^{-1} T(G_{1}G_{2}A) T(G_{1}A)^{-1} = I + G_{2}^{-1} R_{A}(g_{1},g_{2}) T(G_{1}A)^{-1}
\]
	where $R_{A}(g_{1},g_{2})$ is Hilbert--Schmidt. The general theory of modified Fredholm determinants shows that $G_{1}^{-1}T(G_{1}A)T(A)^{-1}$ is invertible if and only if its modified determinant is nonzero. Since $G_{1}$ and $T(A)$ are invertible, this is equivalent to the invertibility of $T(G_{1}A)$, proving (i).

	The definition of the tau-function says%
	\[
\tau^{(2)}_{A}(g_{1}g_{2}) = \text{det}_{2} \left( (G_{1}G_{2})^{-1} T(G_{1}G_{2}A) T(A)^{-1} \right) \text{.}
\]
	Moreover,
	\begin{equation}
		\begin{array}
			[c]{l}%
			(G_{1}G_{2})^{-1} T(G_{1}G_{2}A) T(A)^{-1} = \\
			\left( (G_{1}G_{2})^{-1} T(G_{1}G_{2}A) T(G_{1}A)^{-1} G_{1} \right) \left( G_{1}^{-1} T(G_{1}A) T(A)^{-1} \right) \text{.} %
		\end{array} \label{E2.9}
	\end{equation}
	We use the identities
	\[
\left\{
	\begin{array}
		[c]{ll}%
		\text{det}_{2} \left( G^{-1} (I+A) G \right) = \text{det}_{2}(I+A) \\
		\text{det}_{2} \left( (I+A)(I+B) \right) = \text{det}_{2}(I+A) \text{det}_{2}(I+B) e^{-\text{tr}(AB)} %
	\end{array}
	\right.
\]
	of Hilbert--Schmidt operators $A$, $B$ and a bounded operator $G$ having a bounded inverse. Set
	\[
\begin{aligned}
	K_{1}&=(G_{1}G_{2})^{-1}T(G_{1}G_{2}A)T(G_{1}A)^{-1}G_{1}-I,\\
	K_{2}&=G_{1}^{-1}T(G_{1}A)T(A)^{-1}-I.
	\end{aligned}
\]
	Taking modified determinants in (\ref{E2.9}) therefore yields%
	\begin{align*}
		\tau^{(2)}_{A}(g_{1}g_{2}) & = \operatorname{det}_{2}(I+K_{1})\operatorname{det}_{2}(I+K_{2})
		\exp\!\left(-E_A(g_1,g_2)\right) \\
		& = \tau^{(2)}_{A}(g_{1}) \tau^{(2)}_{G_{1}A}(g_{2}) \text{exp} \left( -E_{A}(g_{1},g_{2}) \right) \text{.}%
	\end{align*} \bigskip
\end{proof}

Continuity of $\operatorname{det}_{2}$ in the Hilbert--Schmidt norm reduces continuity of $\tau^{(2)}_{A}$ to that of the normalized perturbation
\[
G(g)^{-1}T(G(g)A)-T(A).
\]
For $g\in\Gamma_N^{\operatorname{real}}(C)$, set
\[
\mathcal K_g(z,\lambda)=\frac{G(g)(z)^{-1}G(g)(\lambda)-I}{\lambda-z}.
\]
For $g_{1},g_{2}\in\Gamma_{N}^{\operatorname{real}}(C)$, define
\[
d_{n}(g_{1},g_{2})^{2}
=\int_{C^{2}}\|\mathcal K_{g_1}(z,\lambda)-\mathcal K_{g_2}(z,\lambda)\|^{2}
|\lambda|^{-2N+2n-4}|z|^{-2n}|d\lambda|\,|dz|.
\]
We then have the following local Lipschitz estimate.

\begin{lemma}
	\label{l12} Let $A\in\mathcal{A}_{N}^{inv}(C)$, $g_{1},g_{2}\in\Gamma_{N}^{\operatorname{real}}(C)$, and $1\leq n\leq N+1$. Assume $d_{n}(g_{j},1)\leq c_{1}$ for $j=1,2$. Then there exists a constant $c_{A}$, depending only on $c_{1}$, $A$, and $n$, such that
	\begin{equation}
		|\tau^{(2)}_{A}(g_{1}) - \tau^{(2)}_{A}(g_{2})| \leq c_{A} d_{n}(g_{1},g_{2}) \text{.} \label{E2.10}
	\end{equation}
\end{lemma}

\begin{proof}
	Recall the definition $\tau^{(2)}_{A}(g) = \text{det}_{2} \left(  G(g)^{-1}T\left(
	G(g)A \right)  T\left(  A \right)^{-1} \right)$ for $g \in \Gamma_{N}^{\text{real}}(C)$, and %
	\[
G(g)^{-1} T(G(g)A) T(A)^{-1} - I = \left( G(g)^{-1} T(G(g)A) - T(A) \right) T(A)^{-1} \text{.}
\]
	If $\|K_{1}\|_{HS},\|K_{2}\|_{HS}\leq c$, then there exists a constant $C_{c}$ such that
	\[
|\operatorname{det}_{2}(I+K_{1})-\operatorname{det}_{2}(I+K_{2})|\leq C_{c}\|K_{1}-K_{2}\|_{HS}.
\]
	Apply this estimate to the two perturbations in (\ref{E2.10}). Their Hilbert--Schmidt norms are bounded in terms of $c_{1}$ and $\|T(A)^{-1}\|$, while their difference is bounded by $d_{n}(g_{1},g_{2})\|T(A)^{-1}\|$. This proves (\ref{E2.10}). \bigskip
\end{proof}

For later use, we give a sufficient condition for convergence of $\tau_{A}^{(2)}(g_{k})$. Let $U$ be a neighborhood of $\overline{D}_{+}$ such that
\begin{equation}
	z,\lambda \in C \text{, } |z-\lambda| \leq \varepsilon |\lambda| \Rightarrow (\lambda-z) t+z \in U \text{ for } t \in [0,1] \label{E2.11}
\end{equation}
with some $0<\varepsilon < 1$.

\begin{lemma}
	\label{l13} Assume the following properties for $g_{k}$, $g \in \Gamma_{N}^{\operatorname{real}}(C) $:%
	\begin{equation}
		\left\{
		\begin{array}
			[c]{ll}%
			\text{(i) there exists $c_{1}>0$ such that, for every $z\in U$ and $k\geq1$,} \\
			\quad \quad |g_{k}(z)|,|g_{k}^{-1}(z)| \leq c_{1} \text{, } |g_{k}'(z)|,|(g_{k}^{-1})'(z)| \leq c_{1}(1+|z|)^{N-1} \\
			\text{(ii) } g_{k}(z) \to g(z) \text{ as } k \to \infty \text{ for any } z\in C %
		\end{array}
		\right. \text{.} \label{E2.12}
	\end{equation}
	Then, for $A\in\mathcal{A}_{N}^{inv}(C)$ and $1\leq n\leq N+1$,
	\[
\tau_{A}^{(2)}(g_{k}) \to \tau_{A}^{(2)}(g).
\]
\end{lemma}

\begin{proof}
	Set%
	\[
\Delta_{k} = \dfrac{G(g)(z)^{-1}G(g)(\lambda) - G(g_{k})(z)^{-1}G(g_{k})(\lambda)}{\lambda-z}.
\]
	Since
	\[
d_{n}(g_{k},g)^{2} = \displaystyle\int_{C^{2}} \left\Vert \Delta_{k}(z,\lambda) \right\Vert^{2} |\lambda|^{-2N+2n-4} |z|^{-2n} |d\lambda||dz|,
\]
	it suffices to find a majorant for $\|\Delta_{k}(z,\lambda)\|^{2}$ that is integrable with respect to the measure
	\[
|\lambda|^{-2N+2n-4}|z|^{-2n}|d\lambda|\,|dz|.
\]
	Note%
	\[
\Delta_{k}(z,\lambda) = G(g)(z)^{-1} \dfrac{G(g)(\lambda) - G(g)(z)}{\lambda-z} - G(g_{k})(z)^{-1} \dfrac{G(g_{k})(\lambda) - G(g_{k})(z)}{\lambda-z}.
\]
	The bounds in (\ref{E2.12}), together with the same bounds for the limit $g$, imply that there exists $c>0$ such that
	\[
\left\Vert \dfrac{G(g_{k})(z) - G(g_{k})(\lambda)}{z-\lambda} \right\Vert \leq c \left\{
	\begin{array}
		[c]{ll}%
		|\lambda|^{N-1} \text{ if } |z-\lambda| \leq \varepsilon |\lambda|, \\
		|\lambda|^{-1} \text{ if } |z-\lambda| > \varepsilon |\lambda|. %
	\end{array}
	\right.
\]
	holds. Then there exists $C>0$ such that
	\[
f(z,\lambda)=C|\lambda|^{2N-2}\mathbf{1}_{\{|z-\lambda|\leq\varepsilon|\lambda|\}}+C|\lambda|^{-2}\mathbf{1}_{\{|z-\lambda|>\varepsilon|\lambda|\}}
\]
	is integrable with respect to $|\lambda|^{-2N+2n-4} |z|^{-2n} |d\lambda||dz|$ and satisfies
	\[
\left\Vert \Delta_{k}(z,\lambda) \right\Vert^{2} \leq |f(z,\lambda)|.
\]
	Dominated convergence now gives $d_{n}(g_{k},g)\to0$, and Lemma \ref{l12} yields the desired convergence. \bigskip
\end{proof}

\subsection{Derivation of equations}\label{sec:equation-derivation}
We now derive the equations generated by the group action on $\mathcal{A}_{N}^{inv}(C)$. For $A\in\mathcal{A}_{N}^{inv}(C)$, define the $2\times2$ matrix-valued function $U_{A}^{(n)}$, whose two columns belong to $\boldsymbol{H}_{n+1}(D_{+})$, by
\[
U_{A}^{(n)}(z) = T(A)^{-1}z^{n}I \in \boldsymbol{H}_{n+1}(D_{+}) \times \boldsymbol{H}_{n+1}(D_{+}) \text{ for } 0 \leq n \leq N \text{,}
\]
and set %
\[
\Phi_{A}^{(n)} = AU_{A}^{(n)} - z^{n}I \text{.}
\]
Every entry of $\Phi_{A}^{(n)}$ belongs to $H(D_{-})$, and
\[
\Phi_{A}^{(n)}(z) = \dfrac{1}{2\pi i} \displaystyle\int_{C} \dfrac{(A(\lambda)-F(\lambda)) U_{A}^{(n)}(\lambda)}{z-\lambda} d\lambda \text{.}
\]
The decay $\|A(\lambda)-F(\lambda)\|=O(|\lambda|^{-N-2})$, together with $U_{A}^{(n)}\in\boldsymbol{H}_{n+1}(D_{+})^{2}$, implies
\[
\lambda^{k-1}(A(\lambda)-F(\lambda))U_{A}^{(n)}(\lambda)\in L^{1}(C)^{2\times2}\qquad(1\leq k\leq N+1-n).
\]
Expanding the Cauchy kernel therefore gives
\[
\Phi_{A}^{(n)}(z) = \sum\limits_{k = 1}^{N+1-n} A_{k}^{(n)} z^{-k} + \Psi_{A}^{(n)}(z)z^{n-N-1}
\]
where
\[
\left\{
\begin{array}
	[c]{ll}%
	A_{k}^{(n)} = \dfrac{1}{2\pi i} \displaystyle\int_{C} \lambda^{k-1} (A(\lambda)-F(\lambda)) U_{A}^{(n)}(\lambda) d\lambda \text{ for } 1 \leq k \leq N+1-n \\
	\Psi_{A}^{(n)}(z) = \dfrac{1}{2\pi i} \displaystyle\int_{C} \dfrac{\lambda^{N+1-n} (A(\lambda)-F(\lambda)) U_{A}^{(n)}(\lambda)}{z-\lambda} d\lambda \in H(D_{-})^{2\times2} %
\end{array}
\right.
\]
For brevity, write
\[
A_{k} = A_{k}^{(0)} \text{, } \Phi_{A}(z) = \Phi_{A}^{(0)}(z) \text{.}
\]

\begin{lemma}
	\label{l14}
	 For $A\in\mathcal{A}_{N}^{inv}(C)$ and $0\leq n\leq N$, one has
	\begin{equation}
		\left\{
		\begin{array}
			[c]{ll}%
			z^{n}U_{A}^{(0)}(z) = \sum_{0 \leq k \leq n} U_{A}^{(k)}(z)A_{n-k} \\
			z^{n}\Phi_{A}^{(0)}(z) = \sum_{0 \leq k \leq n-1} z^{k}A_{n-k} + \sum_{0 \leq k \leq n} \Phi_{A}^{(k)}(z)A_{n-k} %
		\end{array}
		\right. , \label{E2.13}%
	\end{equation}
	and for $1 \leq j \leq N+1-n$
	\begin{equation}
		A_{n+j} = \sum_{0 \leq k \leq n} A_{j}^{(k)}A_{n-k}, \qquad A_{0}=I. \label{E2.14} %
	\end{equation}
\end{lemma}

\begin{proof}
	For $\boldsymbol{u}\in\boldsymbol{H}_{1}(D_{+})$, observe that
	\begin{align*}
		T(A)(z^{n}\boldsymbol{u}) & = F(z)z^{n}\boldsymbol{u}(z) + \dfrac{1}{2\pi i} \displaystyle\int_{C} \dfrac{\lambda^{n}(A(\lambda) - F(\lambda)) \boldsymbol{u}(\lambda)}{\lambda-z} d\lambda \\
		& = z^{n}F(z)\boldsymbol{u}(z) + \dfrac{z^{n}}{2\pi i} \displaystyle\int_{C} \dfrac{(A(\lambda) - F(\lambda)) \boldsymbol{u}(\lambda)}{\lambda-z} d\lambda \\
		& + \sum_{0 \leq k \leq n-1} \dfrac{z^{k}}{2\pi i} \displaystyle\int_{C} \lambda^{n-1-k} (A(\lambda) - F(\lambda)) \boldsymbol{u}(\lambda) d\lambda \\
		& = z^{n}T(A)\boldsymbol{u} + \sum_{0 \leq k \leq n-1} \dfrac{z^{k}}{2\pi i} \displaystyle\int_{C} \lambda^{n-1-k} (A(\lambda) - F(\lambda)) \boldsymbol{u}(\lambda) d\lambda \text{.}%
	\end{align*}
	Consequently,
	\begin{align*}
		T(A)z^{n}U_{A}^{(0)}(z)
		&=z^{n}T(A)U_{A}^{(0)}(z)\\
		&\quad+\sum_{0 \leq k \leq n-1} \dfrac{z^{k}}{2\pi i} \displaystyle\int_{C} \lambda^{n-1-k} (A(\lambda) - F(\lambda)) U_{A}^{(0)}(\lambda) d\lambda \\
		& = z^{n}I + \sum_{0 \leq k \leq n-1} z^{k} A_{n-k} \text{.}%
	\end{align*}
	Since %
	\[
T(A)^{-1} z^{n} B = U_{A}^{(n)}B
\]
	holds for every constant $2\times2$ matrix $B$, the first identity in (\ref{E2.13}) follows by applying $T(A)^{-1}$. Applying $A$ gives the second identity. Comparing the coefficients in the expansion at infinity then yields (\ref{E2.14}). \bigskip
\end{proof}

For the one-parameter family $g_{t}=e^{ith}\in\Gamma_{N}^{\operatorname{real}}(C)$,
\[
\partial_{t} G(g_{t}) = ih \sigma_{3} G(g_{t}) = ihG(g_{t}) \sigma_{3}
\]
with%
\[
\sigma_{3} = \left( \begin{array}{cc}
	1 & 0 \\
	0 & -1
\end{array} \right)\text{.}%
\]
We have the following differentiability statement.

\begin{lemma}
    \label{l15}
    Let $A\in\mathcal A_N^{inv}(C)$, let $h\in\mathbb R[z]$ have degree
    at most $m\leq N$, and set $g_t=e^{ith}$.
    Assume $G(g_t)A\in\mathcal A_N^{inv}(C)$ for every $t\in\mathbb R$.
    For $1\leq n\leq N+1-m$ and
    $\boldsymbol u\in\boldsymbol H_n(D_+)$, one has
    \[
\begin{aligned}
    \partial_tT(G(g_t)A)\boldsymbol u
      &=i\sigma_3T(hG(g_t)A)\boldsymbol u,\\
    \partial_tT(G(g_t)A)^{-1}\boldsymbol u
      &=-iT(G(g_t)A)^{-1}\sigma_3T(hG(g_t)A)
           T(G(g_t)A)^{-1}\boldsymbol u.
    \end{aligned}
\]
    Both derivatives are continuous with values in
    $\boldsymbol H_{n+m}(D_+)$. More generally,
    $\partial_t^kT(G(g_t)A)^{-1}\boldsymbol u$ exists and is continuous
    in $\boldsymbol H_{n+km}(D_+)$ whenever $n+km\leq N+1$.

    The coefficient $A_\ell^{(n)}(G(g_t)A)$ is $k$ times continuously
    differentiable whenever
    \[
0\leq n\leq N,\qquad \ell\geq1,\qquad n+\ell+km\leq N+1.
\]
    The same conclusions hold for
    \[
g_{\boldsymbol t}=\exp\!\left(i\sum_{\nu=1}^r t_\nu h_\nu\right),
      \qquad \boldsymbol t\in\mathbb R^r,
\]
    where $h_\nu\in\mathbb R[z]$, $\deg h_\nu\le m_\nu$ for integers
    $0\le m_\nu\le N$, and
    $G(g_{\boldsymbol t})A\in\mathcal A_N^{inv}(C)$ for every
    $\boldsymbol t$. For $\beta\in\mathbb Z_{\ge0}^r$, put
    $d(\beta)=\sum_\nu\beta_\nu m_\nu$. If
    $\boldsymbol u\in\boldsymbol H_n(D_+)$ and
    $1\le n$, $n+d(\beta)\le N+1$, then
    $\partial_{\boldsymbol t}^{\beta}T(G(g_{\boldsymbol t})A)^{-1}
    \boldsymbol u$ exists and is jointly continuous in
    $\boldsymbol H_{n+d(\beta)}(D_+)$. The derivative
    $\partial_{\boldsymbol t}^{\beta}A_\ell^{(n)}(G(g_{\boldsymbol t})A)$
    exists and is jointly continuous when
    $n\ge0$, $\ell\ge1$, and $n+\ell+d(\beta)\le N+1$.
\end{lemma}

\begin{proof}
    Put $G_t=G(g_t)$, $T_t=T(G_tA)$, and $E=A-F$.
    First, the Hilbert--Schmidt estimates in Section~\ref{sec:tau-function}, including their
    extension to $\boldsymbol H_{N+1}(D_+)$, show that
    \[
S_t=G_t^{-1}T_t=T(A)+\bigl(G_t^{-1}T_t-T(A)\bigr)
\]
    is norm-continuous on every $\boldsymbol H_n(D_+)$,
    $1\leq n\leq N+1$. Indeed, on a compact set of parameters the kernels
    of the parenthesized perturbation have the common integrable
    majorants used in Lemma~\ref{l13}, and converge pointwise.
    Since $S_t$ is invertible, $S_t^{-1}$ is locally norm-continuous.
    Multiplication by $G_t^{-1}$ is strongly continuous and locally
    uniformly bounded, so $T_t^{-1}=S_t^{-1}G_t^{-1}$ has these properties
    on each of the same weighted spaces.

    The difference quotient of $G_t(\lambda)$ is bounded by
    $C|\lambda|^m$ locally uniformly in $t$. Formula (\ref{E2.2}),
    the bound $E(\lambda)=O(|\lambda|^{-N-2})$, and dominated convergence
    therefore give
    \[
\partial_tT_t\boldsymbol u
        =i\sigma_3T(hG_tA)\boldsymbol u
        \quad\hbox{in }\boldsymbol H_{n+m}(D_+).
\]
    Here $T(hG_tA)$ is interpreted by (\ref{E2.2}) as a map from
    $\boldsymbol H_n(D_+)$ to $\boldsymbol H_{n+m}(D_+)$.
    The inverse identity
    \[
\frac{T_{t+\varepsilon}^{-1}-T_t^{-1}}{\varepsilon}
        =-T_{t+\varepsilon}^{-1}
          \frac{T_{t+\varepsilon}-T_t}{\varepsilon}T_t^{-1}
\]
    now proves the asserted formula, using strong continuity and the
    local uniform bound for the inverse on the target space.
    Iteration proves the higher derivative statement.

    It remains to justify differentiation of the asymptotic coefficients.
    Their integral representation is
    \[
A_\ell^{(n)}(G_tA)
        =\frac{1}{2\pi i}\int_C
          \lambda^{\ell-1}G_t(\lambda)E(\lambda)
          U_{G_tA}^{(n)}(\lambda)\,d\lambda.
\]
    In the $k$th derivative, a term with $a$ derivatives on $G_t$ is,
    up to a constant matrix, of the form
    \[
\lambda^{\ell-1}h(\lambda)^aG_t(\lambda)E(\lambda)
        \partial_t^{k-a}U_{G_tA}^{(n)}(\lambda),
        \qquad 0\leq a\leq k.
\]
    On any compact parameter set $K$, its $L^1(C)$ norm is bounded by
    \[
C_K\|\lambda^{n+\ell+km}E\|_{L^2(C)}
        \bigl\|\partial_t^{k-a}U_{G_tA}^{(n)}
        \bigr\|_{n+1+(k-a)m}.
\]
    If $n+\ell+km\leq N+1$, the first factor is finite, since
    \[
\lambda^{n+\ell+km}E(\lambda)=O(|\lambda|^{-1})
        \quad\hbox{or decays faster}.
\]
    These estimates and the weighted-space continuity just established
    justify differentiation in $L^1(C)$ and continuity of the resulting
    integrals. This proves the one-parameter coefficient assertions,
    including their endpoints.

    For the finite-parameter family, the same kernel majorants are
    uniform on compact parameter sets. The normalized Toeplitz operators
    are therefore jointly norm-continuous, and their unnormalized inverses
    are jointly strongly continuous and locally uniformly bounded on each
    weighted space. Apply the inverse difference-quotient identity
    successively in the coordinate directions. Differentiation in
    $t_\nu$ costs $m_\nu$ weighted orders, with locally uniform
    difference-quotient bounds between the corresponding spaces.
    Repeated differentiation gives mixed derivatives with total loss
    $d(\beta)$ and joint continuity in the asserted target space.
    In the coefficient integral, the term with derivatives
    $\gamma\le\beta$ on the multiplier is bounded in $L^1(C)$ by
    \[
C_K\|\lambda^{n+\ell+d(\beta)}E\|_{L^2(C)}
      \bigl\|\partial_{\boldsymbol t}^{\beta-\gamma}
      U_{G(g_{\boldsymbol t})A}^{(n)}
      \bigr\|_{n+1+d(\beta-\gamma)}.
\]
    The first factor is finite under the stated condition. The same
    weighted-space argument gives continuity in $L^1(C)$, proving the
    joint coefficient assertions.
    \bigskip
\end{proof}

We next derive identities for the coefficients $\{A_{j}\}$ of $\Phi_{G(g_{t,x})A}^{(0)}$ when $g_{t,x}(z)=e^{ixz+itz^{2}}$. We suppress the dependence on $(t,x)$ and write
\[
\left\{
\begin{array}
	[c]{ll}%
	G = G(g_{t,x}) \text{, } U^{(n)} = U_{GA}^{(n)} \\
	\Phi = GAU_{GA}^{(0)} - I %
\end{array}
\right.
\]

\begin{lemma}
	\label{l16} Let $N\geq2$ and $A\in\mathcal A_N^{inv}(C)$.
    Assume $G(g_{t,x})A\in\mathcal A_N^{inv}(C)$ for all
    $t,x\in\mathbb R$. Then%
	\[
\left\{
	\begin{array}
		[c]{ll}%
		\partial_{x} A_{1} = i[\sigma_{3},A_{2}] + iA_{1} [A_{1},\sigma_{3}] \\
		\partial_{t} A_{1} = \partial_{x}A_{2} + iA_{1}^{2} [\sigma_{3},A_{1}] - iA_{1} [\sigma_{3},A_{2}] %
	\end{array}
	\right.
\]
\end{lemma}

\begin{proof}
	Lemma \ref{l15} and (\ref{E2.13}) show that
	\begin{align}
		\partial_{x} U^{(0)} & = -iT(GA)^{-1} \sigma_{3} T(GA) zU^{(0)} \nonumber \\
		& = -iT(GA)^{-1} \sigma_{3} T(GA) \left( U^{(1)} + U^{(0)} A_{1} \right) \nonumber \\
		& = -iT(GA)^{-1} \sigma_{3} \left( zI+A_{1} \right) \nonumber \\
		& = -iU^{(1)}\sigma_{3} - iU^{(0)}\sigma_{3}A_{1} \text{.} \label{E2.15} %
	\end{align}
	Similarly,
	\begin{align*}
		\partial_{t} U^{(0)} & = -iT(GA)^{-1} \sigma_{3} T(GA)\sum_{0 \leq k \leq 2} U^{(k)} A_{2-k} \\
		& = -iT(GA)^{-1} \sigma_{3} \left( z^{2}I + zA_{1} + A_{2} \right) \\
		& = -i\left( U^{(2)}\sigma_{3} + U^{(1)} \sigma_{3} A_{1} + U^{(0)}\sigma_{3}A_{2} \right) \text{.}%
	\end{align*}
	Now we can compute derivatives of $\Phi$. (\ref{E2.15}) shows
	\begin{align}
		\partial_{x}(I+\Phi) & = \partial_{x}(GAU^{(0)}) = iz\sigma_{3}GAU^{(0)} + GA\partial_{x} U^{(0)} \nonumber \\
		& = iz\sigma_{3}\left( I+\Phi \right) - iGA \left( U^{(1)}\sigma_{3} + U^{(0)}\sigma_{3}A_{1} \right) \nonumber \\
		& = iz\sigma_{3}\left( I+\Phi \right) - i\left( \left( zI+\Phi^{(1)} \right) \sigma_{3} + \left( I+\Phi \right) \sigma_{3}A_{1} \right) \nonumber \\
		& = iz\sigma_{3}\Phi - i\Phi^{(1)}\sigma_{3} - i(I+\Phi) \sigma_{3}A_{1} \text{,} \label{E2.16}%
	\end{align}
	which yields %
	\[
\partial_{x} A_{j} = i\sigma_{3}A_{j+1} - iA_{j}^{(1)} \sigma_{3} - iA_{j} \sigma_{3}A_{1}
\]
	for $1\leq j\leq N$ by comparing the coefficients of $z^{-j}$. Applying (\ref{E2.14}) with $n=1$ gives
	\begin{align}
		\partial_{x}A_{j} & = i\sigma_{3}A_{j+1} - i(A_{j+1}-A_{j}A_{1})\sigma_{3} - iA_{j}\sigma_{3}A_{1} \nonumber \\
		& = i[\sigma_{3},A_{j+1}] + iA_{j}[A_{1},\sigma_{3}] \text{.} \label{E2.17}%
	\end{align}
	A similar calculation to (\ref{E2.16}) yields
	\begin{align*}
		\partial_{t}(I+\Phi) & = \partial_{t}\left(GAU^{(0)}\right) = iz^{2}\sigma_{3}GAU^{(0)} + GA\partial_{t}U^{(0)} \\
		& = iz^{2}\sigma_{3}(I+\Phi) - i\left(z^{2}I+\Phi^{(2)}\right)\sigma_{3} \\
		&\qquad - i\left(zI+\Phi^{(1)}\right)\sigma_{3}A_{1} - i(I+\Phi)\sigma_{3}A_{2} \\
		& = iz^{2}\sigma_{3} \Phi - i\Phi^{(2)}\sigma_{3} \\
		&\qquad - i\left( zI+\Phi^{(1)} \right) \sigma_{3}A_{1} - i\left( I+\Phi \right) \sigma_{3}A_{2} \text{,}%
	\end{align*}
	hence (\ref{E2.14}) implies
	\begin{align*}
		\partial_{t}A_{1} & = i\sigma_{3}A_{3} - iA_{1}^{(2)}\sigma_{3} - iA_{1}^{(1)}\sigma_{3}A_{1} - iA_{1}\sigma_{3}A_{2} \\
		& = i\sigma_{3}A_{3} - i\left( A_{3} -A_{1}A_{2} - (A_{2} -A_{1}^{2})A_{1} \right)\sigma_{3} \\
		&\qquad - i(A_{2} - A_{1}^{2})\sigma_{3}A_{1} - iA_{1}\sigma_{3}A_{2} \\
		& = i[\sigma_{3},A_{3}] - iA_{2} [\sigma_{3},A_{1}] + iA_{1}^{2} [\sigma_{3},A_{1}] - iA_{1} [\sigma_{3},A_{2}] \text{.}%
	\end{align*}
	Together with (\ref{E2.17}) for $j=2$, this proves the lemma. \bigskip
\end{proof}

Set %
\[
A_{1} = \left( \begin{array}{cc}
	\alpha_{11} & \alpha_{12} \\
	\alpha_{21} & \alpha_{22}
\end{array} \right) \text{, }
A_{2} = \left( \begin{array}{cc}
	\beta_{11} & \beta_{12} \\
	\beta_{21} & \beta_{22}
\end{array} \right) \text{.}
\]
The first identity in Lemma \ref{l16} gives
\begin{equation}
	\left\{
	\begin{array}
		[c]{ll}%
		\partial_{x}\alpha_{11} = 2i\alpha_{12}\alpha_{21} \\
		\partial_{x}\alpha_{12} = 2i\beta_{12} - 2i\alpha_{11}\alpha_{12} \\
		\partial_{x}\alpha_{21} = -2i\beta_{21} + 2i\alpha_{22}\alpha_{21} \\
		\partial_{x}\alpha_{22} = -2i\alpha_{12}\alpha_{21} %
	\end{array}
	\right. \text{,} \label{E2.18}
\end{equation}
and the off-diagonal entries of the second give
\begin{equation}
	\left\{
	\begin{array}
		[c]{ll}%
		\partial_{t}\alpha_{12} = \partial_{x}\beta_{12} + 2i(\alpha_{11}^{2} + \alpha_{12}\alpha_{21}) \alpha_{12} - 2i\alpha_{11}\beta_{12} \\
		\partial_{t} \alpha_{21} = \partial_{x}\beta_{21} - 2i(\alpha_{12}\alpha_{21} + \alpha_{22}^{2}) \alpha_{21} + 2i\alpha_{22}\beta_{21}
	\end{array}
	\right. \label{E2.19}
\end{equation}
Eliminating $\beta_{12}$ and $\beta_{21}$ from (\ref{E2.18})--(\ref{E2.19}) yields the following closed system.

\begin{lemma}
	\label{l17} The pair $(\alpha_{12},\alpha_{21})$ satisfies
	\begin{equation}
		\left\{
		\begin{array}
			[c]{ll}%
			i\partial_{t}\alpha_{12} = \dfrac{1}{2} \partial_{x}^{2} \alpha_{12} - 4\alpha_{12}^{2} \alpha_{21} \\
			i\partial_{t}\alpha_{21} = -\dfrac{1}{2} \partial_{x}^{2} \alpha_{21} + 4\alpha_{21}^{2} \alpha_{12} %
		\end{array}
		\right. \label{E2.20}
	\end{equation}
\end{lemma}

Set
\[
\begin{aligned}
F(x,z)&=G(e_{x})AU_{G(e_{x})A}^{(0)}
=I+\Phi_{G(e_x)A}(z),\\
\Psi(x,z)&=G(e_x)^{-1}F(x,z)
=AU_{G(e_x)A}^{(0)}.
\end{aligned}
\]
Using (\ref{E2.13}) with $n=1$ in the calculation leading to (\ref{E2.16}), one obtains
\[
\partial_xF=iz[\sigma_3,F]+iF[A_1,\sigma_3].
\]
Consequently, the gauge-normalized matrix $\Psi$ satisfies
\[
i\partial_x\Psi-\Psi[\sigma_3,A_1]=z\Psi\sigma_3.
\]
It follows that
\begin{equation}
	i\sigma_{3}\partial_{x} \boldsymbol{f} + \left( \begin{array}{cc}
		0 & 2\alpha_{21} \\
		2\alpha_{12} & 0
	\end{array} \right) \boldsymbol{f} = z\boldsymbol{f} \label{E2.21}
\end{equation}
where $\boldsymbol{f}(x)=\Psi(x)^{T}\boldsymbol{e}_{1}$.
This solution is called the \textbf{Baker--Akhiezer function}.

\begin{lemma}\label{l18}
If $\widetilde D_+\subset D_+$ are admissible strips, restriction is a
bounded map $H_n(D_+)\to H_n(\widetilde D_+)$ for every integer $n\ge0$.
Here $H_0=H$.
\end{lemma}
\begin{proof}
For $f\in H(D_+)$, split its Cauchy representation into the contributions
of the upper and lower boundary graphs. Each contribution is analytic
in the graph half-plane containing $D_+$. A point on either smaller graph
lies on a vertical interior ray from the source graph at the same real
coordinate. This ray lies in a nontangential cone with aperture depending
only on the Lipschitz constant. The $L^2$ nontangential maximal estimate
for the Cauchy integral on a Lipschitz graph, as in \cite{JV}, yields
\[
\|f|_{\widetilde C}\|_{L^2(\widetilde C)}
 \le C_{C,\widetilde C}\|f|_C\|_{L^2(C)}.
\]
Both target arc measures are comparable with the real coordinate measure.
First apply this estimate to the defining rational functions and then
pass to their closure; the restrictions belong to $H(\widetilde D_+)$.
For $n\ge1$, apply the result to $(z-b)^{-n}u$, with $b\in D_-$.
The weights $|z-b|^{-n}$ and $|z|^{-n}$ are equivalent on each fixed
contour. Multiplying back proves the assertion.
\end{proof}

\begin{lemma}
	\label{l19} Suppose that the contours $C$ and $\widetilde{C}$ satisfy
	$\overline{\widetilde D_+}\subset D_+$. Let
	$A\in\mathcal{A}_{N}^{inv}(C)$ and
	$\widetilde{A}\in\mathcal{A}_{N}^{inv}(\widetilde{C})$.
	Assume that $A$ has the holomorphic continuation and the common analytic
	part $F$ in the analytic-contour condition
	(\ref{E2.5}), and that
	$\widetilde A=A|_{\widetilde C}$.
	Then, for $0\leq n\leq N$,
	$U_{A}^{(n)}=U_{\widetilde{A}}^{(n)}$ on $\widetilde{D}_{+}$ and
	$\Phi_{A}^{(n)}=\Phi_{\widetilde{A}}^{(n)}$ on the common exterior domain.
\end{lemma}

\begin{proof}
Write $E=A-F$.
Take $u=(z-b)^jv$, where $v$ is rational with no poles in $D_+$ and
with $L^2(C)$ boundary values. Then $u(z)=O(|z|^{j-1})$.
For fixed $z\in\widetilde D_+$, the kernel
$E(\lambda)u(\lambda)/(\lambda-z)$ is analytic between the two contours.
On the closing pieces at $\operatorname{Re}\lambda=\pm L$ it is
$O(L^{j-N-4})$, and those pieces have length $O(L^{-(N-1)})$.
Their integrals vanish as $L\to\infty$ for $j\le N+1$. Hence
\[
(T_C(A)u)|_{\widetilde D_+}
 =T_{\widetilde C}(\widetilde A)(u|_{\widetilde D_+}).
\]
Lemma~\ref{l18}, Toeplitz boundedness, and rational density
extend this identity to every $u\in\boldsymbol H_j(D_+)$.
For $u=U_A^{(n)}$, take $j=n+1$ and invert the smaller-contour operator:
\[
U_A^{(n)}=U_{\widetilde A}^{(n)}
 \quad\hbox{on }\widetilde D_+.
\]
Between the contours, $AU_A^{(n)}-z^nI$ is analytic and has the same
Hardy boundary trace as $\Phi_A^{(n)}$ on the old contour. Analytic gluing
and the same Cauchy deformation identify it with
$\Phi_{\widetilde A}^{(n)}$; hence the exterior factors agree.
\end{proof}

\subsection{\texorpdfstring{$m$-function and calculation of $\tau_A(r)$ for rational $r$}{Weyl coordinates and rational tau-functions}}
In this subsection, we use the notation
\[
\left\{
\begin{array}
	[c]{ll}%
	U(z) = \left( T(A)^{-1}I \right)(z) \in \boldsymbol{H}_{1}(D_{+}) \times \boldsymbol{H}_{1}(D_{+}) \\
	F(z) = \left( AU \right) (z) \\
	\Phi(z) = F(z) - I \in \boldsymbol{H}(D_{-}) \times \boldsymbol{H}(D_{-})  %
\end{array}
\right. \text{,}
\]
	thus suppressing the superscript $(0)$ and the subscript $A$. %

Let %
\[
F(z) = \left( \begin{array}{cc}
	f_{11}(z) & f_{12}(z) \\
	f_{21}(z) & f_{22}(z)
\end{array} \right) \text{, }
\Phi(z) = \left( \begin{array}{cc}
	\varphi_{11}(z) & \varphi_{12}(z) \\
	\varphi_{21}(z) & \varphi_{22}(z)
\end{array} \right) \text{,}
\]
and define
\begin{equation}
	\left\{
	\begin{array}
		[c]{ll}%
		m_{1}(z) = \dfrac{f_{12}(z)}{f_{11}(z)} = \dfrac{\varphi_{12}(z)}{1+\varphi_{11}(z)} \\
		m_{2}(z) = \dfrac{f_{21}(z)}{f_{22}(z)} = \dfrac{\varphi_{21}(z)}{1+\varphi_{22}(z)} %
	\end{array}
	\right. \text{,} \label{E2.22}
\end{equation}
and call $(m_{1},m_{2})$ the \textbf{$m$-function} of $A$. Its components are meromorphic on $D_{-}$ and satisfy
\[
m_{j}(z) = O(z^{-1}) \text{,}
\]
uniformly in sufficiently high horizontal half-planes, by the
Cauchy point-evaluation estimate for the Hardy remainder in Section~\ref{sec:equation-derivation}.
In particular the denominators are not identically zero on either
exterior component. We now express the tau-function in terms of the $m$-function.

For $\zeta \in D_{-}$ set
\[
q_{\zeta}(z) = (1-\zeta^{-1}z)^{-1} = \zeta(\zeta-z)^{-1} \in H(D_{+}) \text{.}
\]
Then%
\[
Aq_{\zeta}U = q_{\zeta}I + q_{\zeta} \Phi = q_{\zeta} (I+\Phi(\zeta)) + q_{\zeta} (\Phi - \Phi(\zeta))
\]
is the decomposition into its $\boldsymbol{H}_{1}(D_{+})$ and $\boldsymbol{H}(D_{-})$ matrix columns. Hence
\[
T(A) (q_{\zeta}U) = \mathfrak{p}_{+}(q_{\zeta}AU) = q_{\zeta} (I+\Phi(\zeta)) \text{,}
\]
which implies%
\begin{equation}
	q_{\zeta}U = T(A)^{-1} (q_{\zeta} (I+\Phi(\zeta))) = \left(T(A)^{-1}(q_{\zeta}I)\right)(I+\Phi(\zeta)) \text{.} \label{E2.23}
\end{equation}
Set %
\[
\Delta(\zeta) = (1+\varphi_{11}(\zeta))(1+\varphi_{22}(\zeta)) - \varphi_{12}(\zeta)\varphi_{21}(\zeta) \text{.}
\]

\begin{lemma}
	\label{l20} $\Delta(\zeta) \neq 0$ on $D_{-}$, and%
	\begin{equation}
		T(A)^{-1}(q_{\zeta}I) = q_{\zeta}U(I+\Phi(\zeta))^{-1} \text{.} \label{E2.24}
	\end{equation}
\end{lemma}

\begin{proof}
	Suppose there exists $\boldsymbol{a} \neq \boldsymbol{0} \in \mathbb{C}^{2}$ such that%
	\[
(I+\Phi(\zeta)) \boldsymbol{a} = \boldsymbol{0} \text{.}
\]
	Then (\ref{E2.23}) implies %
	\[
q_{\zeta}U(z) \boldsymbol{a} = \left( T(A)^{-1}(q_{\zeta}I) \right) (I+\Phi(\zeta)) \boldsymbol{a} = \boldsymbol{0} \text{,}
\]
	hence %
	\[
U(z) \boldsymbol{a} = \boldsymbol{0} \text{ identically on } D_{+} \text{,}
\]
	which implies $\boldsymbol{a}=T(A)U\boldsymbol{a}=\boldsymbol{0}$, a contradiction. Thus $\Delta(\zeta)\neq0$, and (\ref{E2.24}) follows from (\ref{E2.23}). \bigskip
\end{proof}

Let $r$ be a rational function with distinct poles $\{\zeta_{j}\}_{j=1}^{m}$ and distinct zeros $\{\eta_{j}\}_{j=1}^{m}$ in $D_{-}$. Then $r$ and $r^{-1}$ have the partial-fraction expansions
\[
\left\{
\begin{array}
	[c]{ll}%
	r(z) = a + \sum\limits_{j=1}^{m} \dfrac{\zeta_{j}a_{j}}{\zeta_{j}-z} = a + \sum\limits_{j=1}^{m} a_{j}q_{\zeta_{j}}(z) \\
	r(z)^{-1} = b+\sum\limits_{j=1}^{m} \dfrac{\eta_{j}b_{j}}{\eta_{j}-z} = b+\sum\limits_{j=1}^{m} b_{j}q_{\eta_{j}}(z) %
\end{array}
\right.
\]
Note%
\[
\left\{
\begin{array}
	[c]{ll}%
	\zeta_{j}a_{j} = \lim\limits_{z \rightarrow \zeta_{j}} (\zeta_{j}-z) r(z) \\
	\eta_{j}b_{j} = \lim\limits_{z \rightarrow \eta_{j}} (\eta_{j}-z) r(z)^{-1} %
\end{array}
\right.
\]
Let%
\[
R(z) = G(r)(z) = \left( \begin{array}{cc}
	r(z) & 0 \\
	0 & r(z)^{-1}
\end{array} \right) \text{.}
\]
We first show that
\[
R^{-1} T(RA)T(A)^{-1} - I
\]
has finite rank. For $\boldsymbol u\in\boldsymbol H_n(D_+)$, put
$\boldsymbol v=T(A)^{-1}\boldsymbol u$.
Choose a bounded analytic part $F_0$ of the symbol $A$ and set
$E=A-F_0$. Since $E\boldsymbol v\in\boldsymbol L^2(C)$,
the exterior remainder is well defined by
\[
\boldsymbol\varphi
 =\mathfrak p_-(E\boldsymbol v)
 =A\boldsymbol v-\boldsymbol u
 \in\boldsymbol H(D_-).
\]
The bounded analytic multiplier $R$ preserves $\boldsymbol H(D_+)$.
Using only Cauchy projections of $L^2$ functions, we obtain
\[
\begin{aligned}
 T(RA)T(A)^{-1}\boldsymbol u
 &=RF_0\boldsymbol v+\mathfrak p_+(RE\boldsymbol v)\\
 &=R\boldsymbol u+\mathfrak p_+(R\boldsymbol\varphi)\\
 &=R\boldsymbol u+
   \frac1{2\pi i}\int_C
   \frac{R(\lambda)\boldsymbol\varphi(\lambda)}{\lambda-z}
   \,d\lambda .
 \end{aligned}
\]

	Writing $\boldsymbol{\varphi} = (\varphi_{1} , \varphi_{2} )$, we obtain, for $z \in D_{+}$,
\begin{align*}
	\dfrac{1}{2\pi i} \displaystyle\int_{C} \dfrac{r(\lambda)\varphi_{1} (\lambda)}{\lambda-z} d\lambda
	&= \dfrac{1}{2\pi i} \displaystyle\int_{C} \dfrac{a\varphi_{1} (\lambda)}{\lambda-z} d\lambda \\
	&\quad+ \sum\limits_{j=1}^{m} \dfrac{1}{2\pi i} \displaystyle\int_{C} \dfrac{a_{j}q_{\zeta_{j}}(\lambda)\varphi_{1} (\lambda)}{\lambda-z} d\lambda \\
	& = \sum\limits_{j=1}^{m} a_{j}q_{\zeta_{j}}(z) \varphi_{1} (\zeta_{j}) \text{,}%
\end{align*}
and, similarly,
\begin{align*}
	\dfrac{1}{2\pi i} \displaystyle\int_{C} \dfrac{r(\lambda)^{-1}\varphi_{2} (\lambda)}{\lambda-z} d\lambda
	&= \dfrac{1}{2\pi i} \displaystyle\int_{C} \dfrac{b\varphi_{2} (\lambda)}{\lambda-z} d\lambda \\
	&\quad+ \sum\limits_{j=1}^{m} \dfrac{1}{2\pi i} \displaystyle\int_{C} \dfrac{b_{j}q_{\eta_{j}}(\lambda)\varphi_{2} (\lambda)}{\lambda-z} d\lambda \\
	& = \sum\limits_{j=1}^{m} b_{j}q_{\eta_{j}}(z) \varphi_{2} (\eta_{j}) \text{,}%
\end{align*}
hence
\[
\begin{aligned}
	(R^{-1} T(RA)T(A)^{-1} - I)\boldsymbol{u}
	&= \sum\limits_{j=1}^{m} a_{j}\varphi_{1} (\zeta_{j})r(z)^{-1}q_{\zeta_{j}}(z)\boldsymbol{e}_{1}\\
	&\quad + \sum\limits_{j=1}^{m} b_{j}\varphi_{2} (\eta_{j})r(z)q_{\eta_{j}}(z)\boldsymbol{e}_{2}.
	\end{aligned}
\]
Therefore, its image is contained in the $2m$-dimensional space spanned by $\{r^{-1}q_{\zeta_{j}}\boldsymbol{e}_{1},rq_{\eta_{j}}\boldsymbol{e}_{2}\}_{j=1}^{m}$. To obtain its matrix in this ordered basis, we compute the functionals $\varphi_{1} (\zeta_{j})$ and $\varphi_{2} (\eta_{j})$ on the basis vectors. First, $r^{-1}q_{\zeta_{k}}$ has poles at $\{\eta_{h}\}$ and
\[
\lim\limits_{z \rightarrow \eta_{h}} (\eta_{h}-z) r(z)^{-1} q_{\zeta_{k}}(z) = \eta_{h}b_{h}q_{\zeta_{k}}(\eta_{h}) \text{,}
\]
hence%
\[
r(z)^{-1} q_{\zeta_{k}}(z) = \sum\limits_{h=1}^{m} \dfrac{\eta_{h}b_{h}q_{\zeta_{k}}(\eta_{h})}{\eta_{h}-z} = \sum\limits_{h=1}^{m} b_{h}q_{\zeta_{k}}(\eta_{h})q_{\eta_{h}}(z) \text{.}
\]
Therefore,
\[
AT(A)^{-1} r^{-1} q_{\zeta_{k}} \boldsymbol{e}_{1} = \sum\limits_{h=1}^{m} b_{h}q_{\zeta_{k}}(\eta_{h}) AT(A)^{-1}q_{\eta_{h}} \boldsymbol{e}_{1} \text{.}
\]
Since (\ref{E2.24}) implies%
\[
\left\{
\begin{array}
	[c]{ll}%
	AT(A)^{-1} q_{\zeta} \boldsymbol{e}_{1} = \dfrac{(1+\varphi_{22}(\zeta))q_{\zeta}\boldsymbol{f}_{1} - \varphi_{21}(\zeta)q_{\zeta}\boldsymbol{f}_{2}}{\Delta(\zeta)} \\
	AT(A)^{-1} q_{\zeta} \boldsymbol{e}_{2} = \dfrac{(1+\varphi_{11}(\zeta))q_{\zeta}\boldsymbol{f}_{2} - \varphi_{12}(\zeta)q_{\zeta}\boldsymbol{f}_{1}}{\Delta(\zeta)} %
\end{array}
\right. \text{,}
\]
and for $\zeta \in D_{-}$ %
\[
\mathfrak{p}_{-} q_{\zeta} \boldsymbol{f}_{\nu} = \mathfrak{p}_{-} q_{\zeta} (\boldsymbol{e}_{\nu} + \boldsymbol{\varphi}_{\nu}) = q_{\zeta} (\boldsymbol{\varphi}_{\nu} - \boldsymbol{\varphi}_{\nu}(\zeta)),
\]
we obtain
\[
\left\{
	\begin{array}
		[c]{ll}%
		\mathfrak{p}_{-} AT(A)^{-1} q_{\zeta} \boldsymbol{e}_{1} = \dfrac{(1+\varphi_{22}(\zeta))(\boldsymbol{\varphi}_{1} - \boldsymbol{\varphi}_{1}(\zeta)) - \varphi_{21}(\zeta)(\boldsymbol{\varphi}_{2} - \boldsymbol{\varphi}_{2}(\zeta))}{\Delta(\zeta)}q_{\zeta} \\
		\mathfrak{p}_{-} AT(A)^{-1} q_{\zeta} \boldsymbol{e}_{2} = \dfrac{(1+\varphi_{11}(\zeta))(\boldsymbol{\varphi}_{2} - \boldsymbol{\varphi}_{2}(\zeta)) - \varphi_{12}(\zeta)(\boldsymbol{\varphi}_{1} - \boldsymbol{\varphi}_{1}(\zeta))}{\Delta(\zeta)}q_{\zeta} %
	\end{array}
	\right.
\]
Evaluating these functionals on $r^{-1}q_{\zeta_{k}}\boldsymbol e_{1}$ gives
\begin{align*}
	a_{j} \varphi_{1}[r^{-1}q_{\zeta_{k}}\boldsymbol e_{1}](\zeta_{j})
	&=\sum\limits_{h=1}^{m} b_{h} q_{\zeta_{k}}(\eta_{h})a_{j}q_{\eta_h}(\zeta_j)\\
	&\quad\times\left[
	\dfrac{(1+\varphi_{22}(\eta_{h}))(1+\varphi_{11}(\zeta_{j})) - \varphi_{21}(\eta_{h}) \varphi_{12}(\zeta_{j})}{\Delta(\eta_{h})}
	-1
	\right]\\
	&=\sum\limits_{h=1}^{m} b_{h} q_{\zeta_{k}}(\eta_{h}) \left( b_{hj}^{(1)} - a_{j}q_{\eta_{h}}(\zeta_{j}) \right),
\end{align*}
where
\begin{align*}
	b_{hj}^{(1)} & = a_{j} \dfrac{(1+\varphi_{22}(\eta_{h}))(1+\varphi_{11}(\zeta_{j})) - \varphi_{21}(\eta_{h}) \varphi_{12}(\zeta_{j})}{\Delta(\eta_{h})} q_{\eta_{h}}(\zeta_{j}) \\
	& = a_{j} \dfrac{(1+\varphi_{22}(\eta_{h}))(1+\varphi_{11}(\zeta_{j})) (1-m_{1}(\zeta_{j})m_{2}(\eta_{h}))}{\Delta(\eta_{h})} q_{\eta_{h}}(\zeta_{j}) \text{.}%
\end{align*}
Note
\[
\sum_{h=1}^{m} a_{j}q_{\eta_{h}}(\zeta_{j})b_{h}q_{\zeta_{k}}(\eta_{h}) = a_{j}r(\zeta_{j})^{-1}q_{\zeta_{k}}(\zeta_{j}) = \delta_{jk},
\]
hence
\[
a_{j}\varphi_{1}[r^{-1}q_{\zeta_{k}}\boldsymbol e_{1}](\zeta_{j}) = \sum_{h=1}^{m} b_{hj}^{(1)}b_{h}q_{\zeta_{k}}(\eta_{h}) - \delta_{jk}.
\]
The computation for $\varphi_{2}(\eta_{j})$ is similar and gives
\[
b_{j}\varphi_{2}[r^{-1}q_{\zeta_{k}}\boldsymbol e_{1}](\eta_{j}) = \sum_{h=1}^{m} b_{h}q_{\zeta_{k}}(\eta_{h})b_{hj}^{(2)} \text{,}
\]
with
\[
b_{hj}^{(2)} = b_{j}\dfrac{(1 + \varphi_{22}(\eta_{h}))(1 + \varphi_{22}(\eta_{j}))(m_{2}(\eta_{j}) - m_{2}(\eta_{h}))}{\Delta(\eta_{h})}q_{\eta_{h}}(\eta_{j})
\]
with the diagonal convention
\[
(m_{2}(\eta_{j}) - m_{2}(\eta_{h}))q_{\eta_{h}}(\eta_{j}) = -\eta_{j}m_{2}^{\prime}(\eta_{j}) \text{.}
\]
We next compute the same functionals on $rq_{\eta_{k}}\boldsymbol{e}_{2}$. The function $rq_{\eta_{k}}$ has poles at $\{\zeta_{h}\}$, and
\[
\lim_{z \to \zeta_{h}} (\zeta_{h} - z) r(z)q_{\eta_{k}}(z) = \zeta_{h}a_{h}q_{\eta_{k}}(\zeta_{h}),
\]
hence
\[
r(z)q_{\eta_{k}}(z) = \sum_{h=1}^{m} \dfrac{\zeta_{h}a_{h}q_{\eta_{k}}(\zeta_{h})}{\zeta_{h} - z} = \sum_{h=1}^{m} a_{h}q_{\eta_{k}}(\zeta_{h})q_{\zeta_{h}}(z) \text{.}
\]
Therefore,
\[
AT(A)^{-1}rq_{\eta_{k}}\boldsymbol{e}_{2} = \sum_{h=1}^{m} a_{h}q_{\eta_{k}}(\zeta_{h})AT(A)^{-1}q_{\zeta_{h}}\boldsymbol{e}_{2} \text{.}
\]
Then
\[
a_{j}\varphi_{1}[rq_{\eta_{k}}\boldsymbol e_{2}](\zeta_{j}) = \sum_{h=1}^{m} b_{hj}^{(3)}a_{h}q_{\eta_{k}}(\zeta_{h})
\]
where
\[
b_{hj}^{(3)} = a_{j}\dfrac{(1 + \varphi_{11}(\zeta_{h}))(1 + \varphi_{11}(\zeta_{j}))(m_{1}(\zeta_{j}) - m_{1}(\zeta_{h}))}{\Delta(\zeta_{h})}q_{\zeta_{h}}(\zeta_{j})
\]
with the diagonal convention
\[
(m_{1}(\zeta_{j}) - m_{1}(\zeta_{h}))q_{\zeta_{h}}(\zeta_{j}) = -\zeta_{j}m_{1}^{\prime}(\zeta_{j}).
\]
For the second component, one has
\[
b_{j}\varphi_{2}[rq_{\eta_{k}}\boldsymbol e_{2}](\eta_{j}) = \sum_{h=1}^{m} a_{h}q_{\eta_{k}}(\zeta_{h})(b_{hj}^{(4)} - b_{j}q_{\zeta_{h}}(\eta_{j}))
\]
with
\[
b_{hj}^{(4)} = b_{j}\dfrac{(1 + \varphi_{11}(\zeta_{h}))(1 + \varphi_{22}(\eta_{j}))(1 - m_{1}(\zeta_{h})m_{2}(\eta_{j}))}{\Delta(\zeta_{h})}q_{\zeta_{h}}(\eta_{j}).
\]
Here again
\[
\sum_{h=1}^{m} b_{j}a_{h}q_{\eta_{k}}(\zeta_{h})q_{\zeta_{h}}(\eta_{j}) = b_{j}r(\eta_{j})q_{\eta_{k}}(\eta_{j}) = \delta_{jk},
\]
hence
\[
b_{j}\varphi_{2}[rq_{\eta_{k}}\boldsymbol e_{2}](\eta_{j}) = \sum_{h=1}^{m} a_{h}q_{\eta_{k}}(\zeta_{h})b_{hj}^{(4)} - \delta_{jk}.
\]
Taking the transpose of the matrix of $R^{-1}T(RA)T(A)^{-1}$ on this subspace, which does not change its determinant, gives
\[
\begin{pmatrix}
	B_{11} & B_{12} \\
	B_{21} & B_{22}
\end{pmatrix}
\]
with
\[
B_{11} = \left(\sum_{h=1}^{m} b_{h}q_{\zeta_{k}}(\eta_{h})b_{hj}^{(1)}\right) \text{, } B_{12} = \left(\sum_{h=1}^{m} b_{h}q_{\zeta_{k}}(\eta_{h})b_{hj}^{(2)}\right),
\]
\[
B_{21} = \left(\sum_{h=1}^{m} a_{h}q_{\eta_{k}}(\zeta_{h})b_{hj}^{(3)}\right) \text{, } B_{22} = \left(\sum_{h=1}^{m} a_{h}q_{\eta_{k}}(\zeta_{h})b_{hj}^{(4)}\right),
\]
hence
\[
\tau_{A}(r) = \det\left(R^{-1}T(RA)T(A)^{-1}\right) = \det\begin{pmatrix}
	B_{11} & B_{12} \\
	B_{21} & B_{22}
\end{pmatrix}.
\]
Set
\[
Q_{1} = (b_{h}q_{\zeta_{k}}(\eta_{h}))_{k,h=1}^{m}, \quad
Q_{2} = (a_{h}q_{\eta_{k}}(\zeta_{h}))_{k,h=1}^{m}.
\]
Hence%
\[
\tau_{A}(r) = \det\begin{pmatrix}
	Q_{1}\left(b_{hj}^{(1)}\right) & Q_{1}\left(b_{hj}^{(2)}\right) \\
	Q_{2}\left(b_{hj}^{(3)}\right) & Q_{2}\left(b_{hj}^{(4)}\right)
\end{pmatrix} = \det(Q_{1})\det(Q_{2})\det\begin{pmatrix}
	b_{hj}^{(1)} & b_{hj}^{(2)} \\
	b_{hj}^{(3)} & b_{hj}^{(4)}
\end{pmatrix} \text{.}
\]
Substituting the formulas for $b_{hj}^{(1)},\ldots,b_{hj}^{(4)}$ above,
using $q_{\eta_h}(z)=\eta_h/(\eta_h-z)$ and
$q_{\zeta_h}(z)=\zeta_h/(\zeta_h-z)$, and extracting the row factors
yields the following formula.

\begin{lemma}
	\label{l21} Let $r$ be a rational function with distinct poles $\{\zeta_{j}\}_{j=1}^{m}$ and distinct zeros $\{\eta_{j}\}_{j=1}^{m}$ in $D_{-}$.
    In each block of the determinant below, $h$ is the row index and $j$
    is the column index. The identity
	\begin{equation}
		\tau_{A}(r) = \Lambda_{1}\Lambda_{2}\det\begin{pmatrix}
			\dfrac{1 - m_{1}(\zeta_{j})m_{2}(\eta_{h})}{\eta_{h} - \zeta_{j}} & \dfrac{m_{2}(\eta_{j}) - m_{2}(\eta_{h})}{\eta_{h} - \eta_{j}} \\
			\dfrac{m_{1}(\zeta_{j}) - m_{1}(\zeta_{h})}{\zeta_{h} - \zeta_{j}} & \dfrac{1 - m_{1}(\zeta_{h})m_{2}(\eta_{j})}{\zeta_{h} - \eta_{j}}
		\end{pmatrix} \label{E2.25}
	\end{equation}
	holds with%
	\[
\left\{
	\begin{array}
		[c]{ll}%
		\Lambda_{1} = \left(\prod\limits_{h=1}^{m} a_{h}b_{h}\right)^{2}
		\left(\prod\limits_{h=1}^{m}\zeta_h\eta_h\right)
		\det(q_{\zeta_{k}}(\eta_{h}))\det(q_{\eta_{k}}(\zeta_{h})) \\
		\Lambda_{2} = \prod\limits_{h=1}^{m} \dfrac{(1 + \varphi_{11}(\zeta_{h}))^{2}(1 + \varphi_{22}(\eta_{h}))^{2}}{\Delta(\zeta_{h})\Delta(\eta_{h})} %
	\end{array}
	\right.
\]
	In particular, if $m=1$,
	\[
\tau_{A}\left(q_{\zeta}q_{\eta}^{-1}\right) = \dfrac{1 + \varphi_{11}(\zeta)}{1 + \varphi_{11}(\eta)}\dfrac{1 + \varphi_{22}(\eta)}{1 + \varphi_{22}(\zeta)}\dfrac{(1 - m_{1}(\zeta)m_{2}(\eta))^{2} + (\zeta - \eta)^{2}m_{1}^{\prime}(\zeta)m_{2}^{\prime}(\eta)}{(1 - m_{1}(\zeta)m_{2}(\zeta))(1 - m_{1}(\eta)m_{2}(\eta))}.
\]
The expressions involving $m_1,m_2$ are initially evaluated where
their coordinate denominators are nonzero. At all other choices
of the poles and zeros, the right-hand sides are interpreted by
their unique analytic continuation from this set. The continuation
exists because the finite-dimensional matrix computed above can
be written in the entries of $I+\Phi_A$ and their divided
differences, with nonzero denominators $\Delta(\zeta_h)$ and
$\Delta(\eta_h)$ by Lemma~\ref{l20}; differences between distinct
pole and zero parameters also stay nonzero.
This convention includes the one-factor formula.
\end{lemma}

Combining this formula with Lemma \ref{l19} gives contour independence.

\begin{lemma}
	\label{l22} Under the assumptions of Lemma \ref{l19}, let $r\in\Gamma_{0}(C)$ be rational, with distinct poles $\{\zeta_{j}\}_{j=1}^{m}$ and distinct zeros $\{\eta_{j}\}_{j=1}^{m}$ in the common exterior domain. Then $\tau_{A}(r)=\tau_{\widetilde{A}}(r)$.
\end{lemma}

We will also use the corresponding equality for the modified determinant. Write
\[
K_{A,C}(r)=G(r)^{-1}T_C(G(r)A)T_C(A)^{-1}-I.
\]
The finite-dimensional matrix computed above has entries depending only on
$r$ and the normalized exterior factor $I+\Phi_A$. By
Lemma~\ref{l19}, this matrix is identical on the two compatible contours.
In particular,
$\operatorname{tr}K_{A,C}(r)=
 \operatorname{tr}K_{\widetilde A,\widetilde C}(r)$.
Together with Lemma~\ref{l22} and
$\det_2(I+K)=\det(I+K)e^{-\operatorname{tr}K}$, this proves the equality
of the modified tau-functions for rational $r$.
For $g=re^{ih}\in\Gamma_N^{\mathrm{real}}(C)$, choose the same rational
approximants to $e^{ih}$ on both contours, as in Lemma~\ref{l23}.
Lemma~\ref{l13}, applied separately on each contour, then gives
\begin{equation}
 \tau_{A,C}^{(2)}(g)
 =\tau_{\widetilde A,\widetilde C}^{(2)}(g).
 \label{E2.26}
\end{equation}
Repeated rational zeros and poles are handled by continuity. This argument
compares the original invertible symbols; it does not assume invertibility
of $T_{\widetilde C}(G(g)\widetilde A)$ in advance.

\subsection{\texorpdfstring{One-factor transformations of the $m$-function}{One-factor transformations of the Weyl coordinates}}

In this section, we compute the complete one-factor transformation.
Our purpose is to show that the transformed $m$-function depends only on the
original $m$-function and the rational multiplier.

Fix $A\in\mathcal A_N^{inv}(C)$ and distinct $\zeta,\eta\in D_-$. Set
\[
r=q_\zeta q_\eta^{-1},\qquad R=G(r),\qquad
    a=r(\infty)=\frac{\zeta}{\eta},\qquad b=a^{-1}.
\]
Then
\[
r=a+a_1q_\zeta,\qquad r^{-1}=b+b_1q_\eta,\qquad
    a_1=1-a,\quad b_1=1-b.
\]
Assume $\tau_A^{(2)}(r)\ne0$, so $T(RA)$ is invertible by
Lemma~\ref{l11}. Write
\[
\begin{gathered}
    U=T(A)^{-1}I,\qquad F=AU=I+\Phi_A=(f_{jk}),\\
    \widetilde U=T(RA)^{-1}I,\qquad
    \widetilde F=RA\widetilde U=I+\widetilde\Phi,\qquad
    D_0=\operatorname{diag}(b,a).
    \end{gathered}
\]
Let $P_j=\boldsymbol e_j\boldsymbol e_j^{\mathsf T}$.
The function $R^{-1}$ is bounded and analytic in $D_+$, and
$R^{-1}\widetilde\Phi\in\boldsymbol L^2(C)$ columnwise.
Using the weighted Toeplitz definition and
$\mathfrak p_+(q_\alpha\widetilde\Phi)
=q_\alpha\widetilde\Phi(\alpha)$ for $\alpha\in D_-$, we have
\[
T(A)\widetilde U
    =R^{-1}+\mathfrak p_+(R^{-1}\widetilde\Phi)
    =R^{-1}
     +b_1q_\eta P_1\widetilde\Phi(\eta)
     +a_1q_\zeta P_2\widetilde\Phi(\zeta).
\]
Applying (\ref{E2.24}) and combining $I+\widetilde\Phi$ gives
\begin{equation}
    \begin{aligned}
    \widetilde U&=UH,\qquad \widetilde F=RFH,\\
    H(z)&=D_0
      +b_1q_\eta(z)F(\eta)^{-1}P_1\widetilde F(\eta)
      +a_1q_\zeta(z)F(\zeta)^{-1}P_2\widetilde F(\zeta).
    \end{aligned}
    \label{E2.27}
\end{equation}
Thus only two constant rows, rather than an unknown matrix function, remain
to be determined.

Initially suppose $f_{11}(\zeta)f_{22}(\eta)\ne0$, and put
\[
a_\zeta=m_1(\zeta),\qquad b_\eta=m_2(\eta),\qquad
    \Delta_{\zeta,\eta}=1-a_\zeta b_\eta,\qquad d=\zeta-\eta.
\]
Normalize the two unknown rows by
\[
\boldsymbol x=
      \frac{b f_{22}(\eta)}{\det F(\eta)}
      \boldsymbol e_1^{\mathsf T}\widetilde F(\eta),\qquad
    \boldsymbol y=
      \frac{a f_{11}(\zeta)}{\det F(\zeta)}
      \boldsymbol e_2^{\mathsf T}\widetilde F(\zeta).
\]
Using the inverse of $F$ and the identities
$(b_1/b)q_\eta=-d/(z-\eta)$ and
$(a_1/a)q_\zeta=d/(z-\zeta)$, (\ref{E2.27}) becomes
\[
H(z)=D_0
      -\frac{d}{z-\eta}\binom{1}{-b_\eta}\boldsymbol x
      +\frac{d}{z-\zeta}\binom{-a_\zeta}{1}\boldsymbol y.
\]
We determine $\boldsymbol x,\boldsymbol y$ by the absence of poles in
$\widetilde F=RFH$. Multiplication by the two rows of $F$ gives
\[
\begin{aligned}
    \frac{\boldsymbol e_1^{\mathsf T}F(z)H(z)}{f_{11}(z)}
      &=(1,m_1(z))D_0-\frac{d(1-b_\eta m_1(z))}{z-\eta}\boldsymbol x
       +d\frac{m_1(z)-a_\zeta}{z-\zeta}\boldsymbol y,\\
    \frac{\boldsymbol e_2^{\mathsf T}F(z)H(z)}{f_{22}(z)}
      &=(m_2(z),1)D_0-d\frac{m_2(z)-b_\eta}{z-\eta}\boldsymbol x
       +\frac{d(1-a_\zeta m_2(z))}{z-\zeta}\boldsymbol y.
    \end{aligned}
\]
The difference quotients have removable singularities at $\zeta$ and
$\eta$, with limits $m_1'(\zeta)$ and $m_2'(\eta)$, respectively.
Since $r$ has a simple pole at $\zeta$ and $r^{-1}$ at $\eta$, the
respective rows of $FH$ must vanish there due to $\widetilde F = RFH$. Taking these limits and
using $d=\zeta-\eta$, we obtain
\[
\begin{aligned}
    0&=(1,a_\zeta)D_0-\Delta_{\zeta,\eta}\boldsymbol x
         +d\,m_1'(\zeta)\boldsymbol y,\\
    0&=(b_\eta,1)D_0-d\,m_2'(\eta)\boldsymbol x
         -\Delta_{\zeta,\eta}\boldsymbol y.
    \end{aligned}
\]
Rearranging gives the single matrix system
\begin{equation}
	\mathcal K_{\zeta,\eta}\begin{pmatrix}\boldsymbol x\\ \boldsymbol y\end{pmatrix} :=
    \begin{pmatrix}
        \Delta_{\zeta,\eta}&-d\,m_1'(\zeta)\\
        d\,m_2'(\eta)&\Delta_{\zeta,\eta}
    \end{pmatrix}
    \begin{pmatrix}\boldsymbol x\\ \boldsymbol y\end{pmatrix}
    =
    \begin{pmatrix}1&a_\zeta\\ b_\eta&1\end{pmatrix}D_0.
    \label{E2.28}
\end{equation}

For rational multipliers, $\tau_A$ and $\tau_A^{(2)}$ have the same
zero set. The one-factor case of Lemma~\ref{l21} reads
\[
\begin{aligned}
    \tau_A(r)
    &=\frac{f_{11}(\zeta)^2f_{22}(\eta)^2}
      {\det F(\zeta)\det F(\eta)}
      \det\mathcal K_{\zeta,\eta},\\
    \det\mathcal K_{\zeta,\eta}
    &=\Delta_{\zeta,\eta}^{\,2}
      +d^2m_1'(\zeta)m_2'(\eta).
    \end{aligned}
\]
Since $\det F\ne0$ by Lemma~\ref{l20}, our assumptions imply
$\det\mathcal K_{\zeta,\eta}\ne0$. No separate condition
$\Delta_{\zeta,\eta}\ne0$ is needed.
Solving (\ref{E2.28}) and substituting in $H$ yields
\begin{equation}
    H(z)=D_0+
    d\begin{pmatrix}
        -\dfrac1{z-\eta}&-\dfrac{a_\zeta}{z-\zeta}\\[
2mm]
        \dfrac{b_\eta}{z-\eta}&\dfrac1{z-\zeta}
    \end{pmatrix}
    \mathcal K_{\zeta,\eta}^{-1}
    \begin{pmatrix}1&a_\zeta\\ b_\eta&1\end{pmatrix}D_0.
    \label{E2.29}
\end{equation}
Every entry of this rational matrix depends only on $m_A$, $\zeta$,
and $\eta$. Finally, left multiplication by the diagonal matrix $R$
does not change projective row coordinates. Thus, writing $H=(H_{jk})$,
\begin{equation}
    \begin{aligned}
    m_{1,RA}
    &=\frac{H_{12}+m_1H_{22}}{H_{11}+m_1H_{21}},\\
    m_{2,RA}
    &=\frac{m_2H_{11}+H_{21}}{m_2H_{12}+H_{22}}.
    \end{aligned}
    \label{E2.30}
\end{equation}
These are identities of meromorphic projective coordinates; common
factors in a row are cancelled before evaluating an apparent $0/0$.
If $f_{11}(\zeta)f_{22}(\eta)=0$, perturb the two parameters to avoid
these isolated zeros. The finite-rank Toeplitz corrections vary
continuously with the parameters, and the assumed invertibility of
$T(RA)$ persists for sufficiently small perturbations. The inverses converge in operator norm, and the Cauchy representation
then gives locally uniform convergence of the normalized exterior factors
on exterior compact sets. Passing to the limit in the cross-multiplied
identities extends the conclusion to the exceptional parameters.
For a real factor the perturbations may be chosen with
$\eta=\overline\zeta$. For two symbols, the pole sets of both
$m$-functions can be avoided simultaneously. Consequently, on one
contour or on two contours with $\zeta,\eta$ in their common exterior,
\[
    m_A=m_B,\qquad
    \tau_A^{(2)}(r)\tau_B^{(2)}(r)\ne0
    \quad\Longrightarrow\quad
    m_{G(r)A}=m_{G(r)B}.
\]
Each tau-function here is computed on its symbol's original contour.
No reality reduction was used in this calculation.

For a constant $c$ with $|c|=1$, the identity
$F_{G(c)A}=G(c)F_AG(c)^{-1}$ gives
\begin{equation}
    m_{G(c)A}=(c^2m_1,c^{-2}m_2).
    \label{E2.31}
\end{equation}
Taking $\eta=\overline\zeta$ in (\ref{E2.29})--(\ref{E2.30}) therefore
defines the elementary real transformation
\begin{equation}
    m_{G(r_\zeta)A}=\mathcal R_\zeta m_A,\qquad
    r_\zeta=q_\zeta q_{\overline\zeta}^{-1}
    =\frac{\zeta}{\overline\zeta}
      \frac{z-\overline\zeta}{z-\zeta}.
    \label{E2.32}
\end{equation}
For $r=c\prod_{\nu=1}^{\ell}r_{\zeta_\nu}$, this gives
\[
m_{G(r)A}
    =\mathcal C_c\mathcal R_{\zeta_\ell}\cdots
      \mathcal R_{\zeta_1}m_A,\qquad
    \mathcal C_c(u_1,u_2)=(c^2u_1,c^{-2}u_2),
\]
provided the transformed Toeplitz operator is invertible after every
partial product. Repeated factors are allowed. The non-vanishing
hypothesis in Lemma~\ref{l24} ensures these intermediate
invertibilities.

\begin{lemma}
	\label{l23} Let $0\leq m\leq N$ and
	$g=e^{ih}\in\Gamma_{m}^{\operatorname{real}}$. Then, on an admissible
	contour $C$, there is a neighborhood $U$ satisfying (\ref{E2.11}) and a
	sequence of rational functions $\{r_{k}\}_{k\geq1}\subset
	\Gamma_{0}^{\operatorname{real}}(C)$, with distinct zeros and poles,
	such that $r_{k}\to g$ in the sense of (\ref{E2.12}).
	The sequence may be chosen so that the zeros and poles escape to
	infinity and their distance from $\overline D_{+}$ tends to infinity.
\end{lemma}

\begin{proof}
	If $h$ is constant, the assertion is immediate. Otherwise, choose
	\[
U=\{z:|\operatorname{Im}z|<
	C_U\langle\operatorname{Re}z\rangle^{-(N-1)}\},
\]
	where $C_U$ is large enough that $U$ contains $\overline D_+$ and
	satisfies (\ref{E2.11}). Since $\deg h\leq N$,
	\[
c_{2}=\sup\limits_{z \in U} |\operatorname{Im}h(z)| < \infty.
\]
	For each $k$, choose distinct numbers
	\[
a_{j,k}\in(2k-\tfrac12,2k+\tfrac12),\qquad 1\leq j\leq k,
\]
	so that none of the values $\pm ia_{j,k}$ is a critical value of $h$, and set
	\[
r_{k}(z)=\prod_{j=1}^{k}
	\frac{1+ih(z)/a_{j,k}}{1-ih(z)/a_{j,k}}.
\]
	We work with sufficiently large $k$ that $a_{j,k}>2c_2$ for all $j$;
	deleting finitely many initial terms gives the required sequence.
	Each factor satisfies the required reality symmetry and has no zeros or
	poles in $U$, and hence $r_{k}\in\Gamma_{0}^{\operatorname{real}}(C)$. The zeros and poles solve $h(z)=\pm ia_{j,k}$. Our choice of the $a_{j,k}$ makes them simple and pairwise distinct. Since the leading coefficient of $h$ is real, these roots have modulus tending to infinity and their arguments remain bounded away from $0$ and $\pi$; in particular, they escape from $\overline D_{+}$.

	On every compact subset of $\mathbb C$,
	\[
\log r_{k}(z)
	=\sum_{j=1}^{k}\left(\frac{2ih(z)}{a_{j,k}}+O(k^{-2})\right)
	=ih(z)+O(k^{-1}),
\]
	so $r_{k}\to e^{ih}=g$ locally uniformly. Since $|\operatorname{Im}h|\leq c_{2}$ on $U$, the elementary identity
	\[
\left|\frac{1+w/a}{1-w/a}\right|^{2}
	=\frac{(a+\operatorname{Re}w)^{2}+(\operatorname{Im}w)^{2}}
	{(a-\operatorname{Re}w)^{2}+(\operatorname{Im}w)^{2}}
\]
	with $w=ih(z)$ gives a constant $c_{3}>1$, independent of $k$, such that
	\[
c_{3}^{-1}\leq |r_{k}(z)|,|r_{k}(z)^{-1}|\leq c_{3},
	\qquad z\in U.
\]
	Finally,
	\[
\frac{r_{k}'(z)}{r_{k}(z)}
	=\sum_{j=1}^{k}\frac{2ia_{j,k}h'(z)}{a_{j,k}^{2}+h(z)^{2}}.
\]
	The same bounds and $h'(z)=O((1+|z|)^{m-1})$ on $U$ show that, for some $c_{4}$ independent of $k$,
	\[
|r_{k}'(z)|+|(r_{k}^{-1})'(z)|\leq c_{4}(1+|z|)^{m-1},
	\qquad z\in U.
\]
	These estimates verify (\ref{E2.12}). \bigskip
\end{proof}

\begin{lemma}
    \label{l24}
    Let $C_A$ and $C_B$ be admissible contours, with corresponding
    domains $D_{A,\pm}$ and $D_{B,\pm}$, and let
    \[
A\in\mathcal A_N^{inv}(C_A),\qquad
        B\in\mathcal A_N^{inv}(C_B).
\]
    Suppose that $A$ and $B$ satisfy the same reality reduction:
    \[
\sigma_j^{-1}\overline A\sigma_j=A,\qquad
        \sigma_j^{-1}\overline B\sigma_j=B
        \quad\text{for some }j\in\{2,4\}.
\]
    Assume that $m_A=m_B$ on
    $\Omega=D_{A,-}\cap D_{B,-}$, where $m$ denotes the pair of
    projective coordinates in (\ref{E2.22}). Fix
    \[
g\in\Gamma_N^{\operatorname{real}}(C_A)
        \cap\Gamma_N^{\operatorname{real}}(C_B).
\]
    Suppose that
    \begin{equation}
        \tau_{A,C_A}^{(2)}(g)\ne0,\qquad
        \tau_{B,C_B}^{(2)}(g)\ne0,
        \label{E2.33}
    \end{equation}
    and that
    \begin{equation}
        \tau_{A,C_A}^{(2)}(r)\ne0,\qquad
        \tau_{B,C_B}^{(2)}(r)\ne0
        \quad\text{for every }
        r\in\Gamma_0^{\operatorname{real}}(C_A)
        \cap\Gamma_0^{\operatorname{real}}(C_B).
        \label{E2.34}
    \end{equation}
    Here each tau-function is computed on the indicated contour.
    Then
    \begin{equation}
        m_{G(g)A}=m_{G(g)B}
        \quad\text{on }\Omega.
        \label{E2.35}
    \end{equation}
\end{lemma}

\begin{proof}
    First let $r$ be real rational, with all its zeros and poles in
    $\Omega$, and write
    \[
r=c\prod_{\nu=1}^{L}r_{\zeta_\nu},\qquad |c|=1.
\]
    Every partial product
    \[
r^{(0)}=1,\qquad
        r^{(j)}=\prod_{\nu=1}^{j}r_{\zeta_\nu},
        \quad 1\le j\le L,
\]
    belongs to both rational multiplier groups. By
    (\ref{E2.34}) and Lemma~\ref{l11}, its transformed
    Toeplitz operator is invertible on each original contour.
    The one-factor calculation (\ref{E2.32}), applied separately on
    $C_A$ and $C_B$, depends only on the preceding $m$-function and
    the parameter $\zeta_\nu$, not on the contour. It therefore
    preserves equality of the two projective coordinates at each
    step. The constant-factor formula (\ref{E2.31}) gives
    \begin{equation}
        m_{G(r)A}=m_{G(r)B}\quad\text{on }\Omega.
        \label{E2.36}
    \end{equation}

    For the general multiplier, write $g=r_0e^{ih}$, where
    $h\in\mathbb R[z]$, $\deg h\le N$, and $r_0$ is real rational
    with all its zeros and poles in $\Omega$.
    Choose an auxiliary admissible contour whose interior
    contains those of $C_A$ and $C_B$. Applying
    Lemma~\ref{l23} to $e^{ih}$ on this auxiliary contour gives
    real rational functions $s_k$, with zeros and poles escaping
    to infinity in $\Omega$, such that
    \[
r_k=r_0s_k\longrightarrow g
\]
    satisfies (\ref{E2.12}) on both original contours.
    Indeed, the estimates for $s_k$ hold on suitable neighborhoods
    of both contours; these neighborhoods can be chosen to
    avoid the finitely many zeros and poles of $r_0$.
    Multiplication by $r_0$ preserves the bounds in (\ref{E2.12}).
    The auxiliary contour is used only to choose the approximants;
    neither symbol is transported to it.
    By (\ref{E2.36}),
    \[
m_{G(r_k)A}=m_{G(r_k)B}\quad\text{on }\Omega.
\]

    We pass to the limit separately on the two contours.
    For $D=A,B$, use its original contour $C_D$ and put
    \[
\begin{aligned}
        \mathcal S_{D,k}
        &=G(r_k)^{-1}T_{C_D}(G(r_k)D)T_{C_D}(D)^{-1},\\
        \mathcal S_D
        &=G(g)^{-1}T_{C_D}(G(g)D)T_{C_D}(D)^{-1}.
        \end{aligned}
\]
    On $\boldsymbol H_1(D_{D,+})$, the dominated-convergence
    estimate in Lemma~\ref{l13} gives
    \[
\|\mathcal S_{D,k}-\mathcal S_D\|_{HS}\longrightarrow0.
\]
    By (\ref{E2.33}), $\mathcal S_D$ is invertible. Hence
    $\mathcal S_{D,k}^{-1}\to\mathcal S_D^{-1}$ in operator norm.
    The bounds in (\ref{E2.12}) also give
    \[
G(r_k)^{-1}I\longrightarrow G(g)^{-1}I
        \quad\text{in }\boldsymbol H_1(D_{D,+})^2.
\]
    Consequently,
    \[
\begin{aligned}
        U_{G(r_k)D}^{(0)}
        &=T_{C_D}(D)^{-1}\mathcal S_{D,k}^{-1}G(r_k)^{-1}I\\
        &\longrightarrow
        T_{C_D}(D)^{-1}\mathcal S_D^{-1}G(g)^{-1}I
        =U_{G(g)D}^{(0)}
        \end{aligned}
\]
    in this space.

    Let $F$ be an analytic part of $D$ in
    (\ref{E2.2}). The analytic part of $G(r_k)D$ is
    $G(r_k)F$, so the exterior Cauchy formula is
    \[
\Phi_{G(r_k)D}(z)
        =\frac{1}{2\pi i}\int_{C_D}
        \frac{G(r_k)(\lambda)
        (D(\lambda)-F(\lambda))
        U_{G(r_k)D}^{(0)}(\lambda)}
        {z-\lambda}\,d\lambda.
\]
    Since $\lambda(D-F)$ is bounded on $C_D$,
    convergence in $\boldsymbol H_1^2$, together with the
    uniform multiplier bounds and dominated convergence,
    implies convergence of these densities in $L^2(C_D)^{2\times2}$.
    The Cauchy kernel has uniformly bounded $L^2(C_D)$ norm
    when $z$ ranges over a compact subset of $D_{D,-}$.
    It follows that
    \[
I+\Phi_{G(r_k)D}\longrightarrow I+\Phi_{G(g)D}
        \quad\text{locally uniformly on }D_{D,-}.
\]
    Passing to the limit in the cross-multiplied identities
    for the two row quotients proves (\ref{E2.35}).
    Lemma~\ref{l20} ensures that no row of either limiting
    exterior factor vanishes, and its normalization at infinity
    ensures that the diagonal entries are not identically zero.
    Thus the conclusion is an equality of meromorphic
    projective coordinates on the common exterior. \bigskip
\end{proof}

\subsection{Self-adjointness and skew-self-adjointness}
The system (\ref{E2.20}) reduces to defocusing NLS when $\alpha_{21}=\overline{\alpha}_{12}$ and to focusing NLS when $\alpha_{21}=-\overline{\alpha}_{12}$. We examine how $\Phi_{A}$ transforms under
\[
\sigma_{2}^{-1}\overline{A}\sigma_{2} \text{, } \sigma_{4}^{-1}\overline{A}\sigma_{4}
\]
with $\overline{A}(\lambda) = \overline{A(\overline{\lambda})}$. Recall the matrices
\[
\sigma_{2} = \begin{pmatrix}
	0 & 1 \\
	1 & 0
\end{pmatrix} \text{, } \sigma_{4} = \begin{pmatrix}
	0 & 1 \\
	-1 & 0
\end{pmatrix} \text{.}
\]

\begin{lemma}\label{l25}
For $A\in\mathcal A_N^{inv}(C)$, put
$A^{[j]}=\sigma_j^{-1}\overline A\sigma_j$, $j=2,4$.
Then
\[
\Phi_{A^{[j]}}=\sigma_j^{-1}\overline{\Phi_A}\sigma_j,
 \qquad
 \tau_{A^{[j]}}^{(2)}(g)=
 \overline{\tau_A^{(2)}(\overline g^{-1})}.
\]
For rational $g$, the same identity holds for the ordinary tau-function.
\end{lemma}
\begin{proof}
Let $\mathcal C_jf=\sigma_j^{-1}\overline{f(\bar z)}$.
Conjugating the Cauchy formula and reversing the reflected boundary
orientation gives $\mathfrak p_+\overline f=\overline{\mathfrak p_+f}$.
Consequently
\[
T(A^{[j]})=\mathcal C_jT(A)\mathcal C_j^{-1},\qquad
 T(A^{[j]})^{-1}=\mathcal C_jT(A)^{-1}\mathcal C_j^{-1}.
\]
Apply the second identity to the constant matrix $I$ to obtain the
formula for $\Phi$. The identity
\[
\sigma_j^{-1}\overline{G(\overline g^{-1})}\sigma_j=G(g)
\]
gives the same anti-linear conjugation relation for the normalized
Toeplitz perturbations. Their eigenvalues, with algebraic multiplicity,
are complex conjugated. The canonical product for $\det_2$, and the
finite-rank determinant in the rational case, give the assertions.
\end{proof}

The same conjugation argument shows that, under either reduction,
\[
E_{A}(g_{1},g_{2})\in\mathbb R
\]
whenever $g_{1},g_{2}$ belong to the real subgroup and all terms in the cocycle identity are defined. Thus the exponential correction in (\ref{E2.8}) is positive in every application below.

Thus, if $A=\sigma_{2}^{-1}\overline{A}\sigma_{2}$, then
\[
\Phi_{A} = \sigma_{2}^{-1}\overline{\Phi}_{A}\sigma_{2},
\]
hence the first coefficient $A_{1}$ of $\Phi_{A}(z)$ and the $m$-functions satisfy
\[
\alpha_{21} = \overline{\alpha}_{12}, \quad \alpha_{22} = \overline{\alpha}_{11}, \quad m_{1} = \overline{m}_{2}.
\]
If instead $A=\sigma_{4}^{-1}\overline{A}\sigma_{4}$, then $\Phi_{A}=\sigma_{4}^{-1}\overline{\Phi}_{A}\sigma_{4}$ and
\[
\alpha_{21} = -\overline{\alpha}_{12}, \quad \alpha_{22} = \overline{\alpha}_{11}, \quad m_{1} = -\overline{m}_{2}.
\]

We call $A\in\mathcal{A}_{N}^{inv}(C)$ \textbf{self-adjoint} (respectively, \textbf{skew-self-adjoint}) if $A=\sigma_{2}^{-1}\overline{A}\sigma_{2}$ (respectively, $A=\sigma_{4}^{-1}\overline{A}\sigma_{4}$). These symmetries are preserved by multiplication by $G(g)$ whenever
\[
g(z)\overline{g}(z) = 1,
\]
or equivalently $g \in \Gamma_{m}^{\text{real}}(C)$, since
\[
\sigma_{2}^{-1}\overline{G(g)A}\sigma_{2} = \sigma_{2}^{-1}\overline{G(g)}\overline{A}\sigma_{2} = \sigma_{2}^{-1}G(\overline{g})\sigma_{2}\sigma_{2}^{-1}\overline{A}\sigma_{2} = G(g)A,
\]
\[
\sigma_{4}^{-1}\overline{G(g)A}\sigma_{4} = \sigma_{4}^{-1}\overline{G(g)}\overline{A}\sigma_{4} = \sigma_{4}^{-1}G(\overline{g})\sigma_{4}\sigma_{4}^{-1}\overline{A}\sigma_{4} = G(g)A.
\]
Lemmas \ref{l17} and \ref{l25} now give the NLS reductions.

\begin{proposition}
	\label{p26} Let $N\geq2$ and $A\in\mathcal{A}_{N}^{inv}(C)$. Assume that $G(e^{ixz+itz^{2}})A\in\mathcal{A}_{N}^{inv}(C)$ for all $x,t\in\mathbb{R}$. Let $A_{1}=\begin{pmatrix}
		\alpha_{11} & \alpha_{12} \\
		\alpha_{21} & \alpha_{22}
	\end{pmatrix}$ be the first coefficient of $\Phi_{G(e^{ixz+itz^{2}})A}$. Then
	\[
\begin{array}
		[c]{ll}%
		\text{if } A = \sigma_{2}^{-1}\overline{A}\sigma_{2} \implies i\partial_{t}\alpha_{21} = -\dfrac{1}{2}\partial_{x}^{2}\alpha_{21} + 4|\alpha_{21}|^{2}\alpha_{21} \quad \text{(defocusing NLS)} \\
		\text{if } A = \sigma_{4}^{-1}\overline{A}\sigma_{4} \implies i\partial_{t}\alpha_{21} = -\dfrac{1}{2}\partial_{x}^{2}\alpha_{21} - 4|\alpha_{21}|^{2}\alpha_{21} \quad \text{(focusing NLS)} %
	\end{array} \text{.}
\]
\end{proposition}

\section{Construction of the defocusing NLS flow}\label{sec:defocusing-construction}
We first impose the reduction $A=\sigma_{2}^{-1}\overline{A}\sigma_{2}$, for which $\alpha_{21}$ satisfies the defocusing NLS equation. The underlying operator is then a self-adjoint Dirac operator, so Weyl--Titchmarsh theory is available. When it is useful to display the dependence on the symbol, we write $m_{A}=m_{1}$.

\subsection{\texorpdfstring{A subclass of $\mathcal A_{N,d}^{inv}$}{Canonical defocusing symbols}}
Recall that $\mathcal A_{N,d}^{inv}$ is defined by (\ref{E2.4}), including
the holomorphic continuation and the common analytic part specified in
(\ref{E2.5}). In particular, its symbols form
compatible families on all admissible contours inside a witnessing contour.
Define
\[
\mathcal{M}_{N,d} = \left\{ M = \begin{pmatrix}
	1 & m_{1} \\
	\overline{m}_{1} & 1
\end{pmatrix} \text{; } m_{1} \text{ satisfies (D.1) and (D.2)}
\right\}
\]
where the following conditions hold.

\smallskip
\noindent\textup{(D.1)} The function $m_{1}$ is analytic on $\mathbb{C}\setminus\mathbb{R}$ and satisfies
\begin{align*}
|m_{1}(z)|&<1 &&\text{on }\mathbb{C}\setminus\mathbb{R},\\
m_{1}(z)&=O(z^{-1}) &&\text{on every admissible exterior domain.}
\end{align*}
\smallskip
\noindent\textup{(D.2)} For every admissible contour $C$, the function $m_{1}$ has the asymptotic expansion
\[
m_{1}(z) = \sum_{1 \leq k \leq N+1} a_{k}z^{-k} + O(z^{-N-2}) \quad\text{on } D_{-}.
\]
The symmetry $G(g)M=\sigma_{2}^{-1}\overline{G(g)M}\sigma_{2}$ holds for $M\in\mathcal{M}_{N,d}$ and $g\in\Gamma_{N}^{\operatorname{real}}$, although invertibility of $T(G(g)M)$ has not yet been established.

If $A\in\mathcal{A}_{N}^{inv}(C)$ satisfies $A=\sigma_{2}^{-1}\overline{A}\sigma_{2}$, then
\[
\varphi_{11} = \overline{\varphi}_{22} \text{, } \varphi_{12} = \overline{\varphi}_{21} \text{, } m_{1} = \overline{m}_{2}
\]
and Lemma \ref{l20} gives $\Delta(\zeta)\neq0$ on $D_{-}$. Lemma \ref{l21} further yields
\begin{align}
	\tau_{A}(r_{\zeta}) & = \left| \dfrac{1+\varphi_{11}(\zeta)}{1+\overline{\varphi}_{11}(\zeta)} \right|^{2} \dfrac{\left( 1-|m_{1}(\zeta)|^{2} \right)^{2} - 4\left( \operatorname{Im} \zeta \right)^{2} |m_{1}'(\zeta)|^{2}}{|1-m_{1}(\zeta)\overline{m}_{1}(\zeta)|^{2}} \nonumber \\
	& = \dfrac{|1+\varphi_{11}(\zeta)|^{4}}{|\Delta(\zeta)|^{2}} \left( \left( 1-|m_{1}(\zeta)|^{2} \right)^{2} - 4\left( \operatorname{Im} \zeta \right)^{2} |m_{1}'(\zeta)|^{2} \right) \text{.} \label{E3.1}%
\end{align}
Here the $\varphi_{ij}$ are holomorphic on $D_{-}$ and satisfy
$\varphi_{ij}(z)=O(z^{-1})$ uniformly in sufficiently high horizontal
half-planes. Lemma \ref{l25} also shows that $\tau_{A}^{(2)}(g)\in\mathbb{R}$ for $g\in\Gamma_{N}^{\operatorname{real}}$.

\begin{lemma}
	\label{l27} Let $M\in\mathcal{M}_{N,d}$. For every contour $C$ satisfying (\ref{E2.1}) and every $1\leq n\leq N+1$, the operator $T(M)$ has a bounded inverse on $\boldsymbol{H}_{n}(D_{+})$. Moreover, $\Phi_{M}=M-I$, $m_{M}=m_{1}$, and $\mathcal{M}_{N,d}\subset\mathcal{A}_{N,d}^{inv}$.
\end{lemma}

\begin{proof}
	Condition (D.2) implies $M\in\mathcal{A}_{N+2}(C)$. Since $|m_{1}|<1$ in both half-planes,
	\[
M^{-1} = \dfrac{1}{1 - m_{1}\overline{m}_{1}} \begin{pmatrix}
		1 & -m_{1} \\
		-\overline{m}_{1} & 1
	\end{pmatrix}.
\]
	Both $M$ and $M^{-1}$ satisfy
	\[
M(z)\text{, } M(z)^{-1} = I + O(z^{-1}) \text{ on } D_{-} \text{,}
\]
	uniformly on admissible exterior contours. In fact, (D.2) gives $M^{-1}\in\mathcal{A}_{N+2}(C)$. Let $F$ and $\widetilde F$ be bounded analytic functions on $D_{+}$ associated with $M$ and $M^{-1}$, respectively, so that
	\[
\begin{aligned}
	\sup_{z\in\overline D_{+}}\bigl(\|F(z)\|+\|\widetilde F(z)\|\bigr)&<\infty,\\
	\sup_{z\in C}|z|^{N+2}\bigl(\|M(z)-F(z)\|+\|M(z)^{-1}-\widetilde F(z)\|\bigr)&<\infty.
	\end{aligned}
\]
	Multiplication by either $M$ or $M^{-1}$ preserves $\boldsymbol{H}(D_{-})$. Therefore, for $\boldsymbol{u}\in\boldsymbol{H}_{n}(D_{+})$, a direct calculation from (\ref{E2.2}) gives
	\[
\boldsymbol{u} - T(M)T(M^{-1})\boldsymbol{u} = \mathfrak{p}_{-} \left( (M-F)T(M^{-1})\boldsymbol{u} \right) + M\mathfrak{p}_{-} \left( (M^{-1}-\widetilde{F})\boldsymbol{u} \right) \text{.}
\]
	The left-hand side belongs to $\boldsymbol{H}_{n}(D_{+})$, whereas the right-hand side belongs to $\boldsymbol{H}(D_{-})$; their intersection is $\{0\}$. Hence $T(M)T(M^{-1})=I$. Interchanging $M$ and $M^{-1}$ gives $T(M^{-1})T(M)=I$, and therefore $T(M)^{-1}=T(M^{-1})$. Finally, $m_{1}\boldsymbol{e}_{1},\overline m_{1}\boldsymbol{e}_{2}\in\boldsymbol{H}(D_{-})$, so $T(M)I=I$. Consequently,
	\[
\Phi_{M} = MT(M)^{-1}I - I = M - I
\]
	and $m_{M} = m_{1}$. %

	We also verify the analytic-contour condition in the definition.
	The entries of $M$ are holomorphic off $\mathbb R$, and
	$1-m_1\overline m_1\ne0$ there. Fix a witnessing contour $C$ and a point
	$b\in D_-$. A matrix function of the form
	$F(z)=I+\sum_{k=1}^{N+1} B_k(z-b)^{-k}$ can be chosen to match the first
	$N+1$ coefficients of $M$ at infinity. It is bounded and analytic on
	$\overline D_+$. Condition (D.2), applied to each smaller contour,
	together with analyticity on compact subsets off $\mathbb R$, gives
	(\ref{E2.5}) with this same $F$.

	It remains to verify positivity. Fix $\zeta\in\mathbb{C}_{+}$. The holomorphic function
	\[
f(z) = \dfrac{m_{1}\left(\dfrac{\zeta + \overline{\zeta}z}{1 + z}\right) - m_{1}(\zeta)}{1 - m_{1}\left(\dfrac{\zeta + \overline{\zeta}z}{1 + z}\right)\overline{m_{1}(\zeta)}}
\]
	satisfies $|f(z)| < 1$ for $|z| < 1$, $f(0) = 0$, and
	\[
f'(0) = \dfrac{-2i(\operatorname{Im}\zeta)m_{1}^{\prime}(\zeta)}{1 - |m_{1}(\zeta)|^{2}}.
\]
	Schwarz's lemma gives $|f'(0)|\leq1$, which is precisely the non-negativity of the last factor in (\ref{E3.1}). Thus $\tau_{M}(r_{\zeta})\geq0$. The lower half-plane follows by conjugation, and hence $M\in\mathcal{A}_{N,d}^{inv}$. \bigskip
\end{proof}

\subsection{Non-vanishing of the defocusing tau-function}
We now prove that $\tau_{A}^{(2)}(g)>0$ for $A\in\mathcal{A}_{N,d}^{inv}$ and $g\in\Gamma_{N}^{\operatorname{real}}$. For a fixed admissible contour, set
\[
\mathcal{Z} = \{z \in D_{-}:1 + \varphi_{11}(z) = 0\}.
\]
The high-horizontal-half-plane normalization shows that
$1+\varphi_{11}$ is not identically zero on either component. Thus
$\mathcal Z$ is relatively closed and discrete in $D_-$. The function $m_{A}$ is holomorphic on $D_{-}\setminus\mathcal{Z}$ and may have poles only at points of $\mathcal{Z}$; by Lemma \ref{l20}, $1+\varphi_{11}$ and $\varphi_{12}$ cannot vanish simultaneously.

\begin{lemma}\label{l28}
Let $\mathcal Z$ be relatively closed and discrete in $D_-$, and let $f$
be holomorphic on $D_-\setminus\mathcal Z$. Suppose
\begin{equation}\label{E3.2}
 \bigl|1-|f(z)|^2\bigr|-2|\operatorname{Im}z|\,|f'(z)|\ge0.
\end{equation}
Assume that each component of $D_-$ contains a point outside $\mathcal Z$
where $|f|<1$. Then $f$ extends holomorphically to $D_-$ and $|f|<1$ there.
\end{lemma}
\begin{proof}
Deleting a relatively closed discrete set leaves each exterior component
connected. If one such component contained both $|f|<1$ and $|f|>1$,
the separating level set $|f|=1$ could not be discrete. On that set
(\ref{E3.2}) gives $f'=0$, and the identity theorem would make $f$
constant, a contradiction. Hence $|f|\le1$ off $\mathcal Z$.
Its isolated singularities are removable. The maximum-modulus principle
and the point where $|f|<1$ give strict inequality on each component.
\end{proof}

\begin{lemma}
	\label{l29} If $A\in\mathcal{A}_{N,d}^{inv}$, then $\Phi_{A}$ is holomorphic on $\mathbb{C}\setminus\mathbb{R}$, $m_{A}$ satisfies (D.1), and $\tau_{A}(r_{\zeta})>0$ for every $\zeta\in\mathbb{C}\setminus\mathbb{R}$.
\end{lemma}

\begin{proof}
	Condition (ii) in the definition of $\mathcal{A}_{N,d}^{inv}$ and Lemma \ref{l19} show that $\Phi_{A}$ is holomorphic on $\mathbb{C}\setminus\mathbb{R}$. Condition (iii) and (\ref{E3.1}) imply
	\begin{equation}
		\left|1 - |m_{A}(\zeta)|^{2} \right| - 2|\operatorname{Im}\zeta||m_{A}'(\zeta)| \geq 0 \label{E3.3}
	\end{equation}
	for $\zeta\in D_{-}\setminus\mathcal{Z}$, where $\mathcal{Z}=\{\zeta\in D_{-}:1+\varphi_{11}(\zeta)=0\}$. Lemma \ref{l28} shows that $m_{A}$ is holomorphic on $D_{-}$ and that
	\[
|m_{A}(z)| < 1 \quad \text{on } D_{-}, \quad \text{and } \mathcal{Z} = \emptyset.
\]
	Here $\mathcal Z=\emptyset$: if a zero of $1+\varphi_{11}$ were removable in the quotient defining $m_A$, then $\varphi_{12}$ would vanish at the same point, contrary to Lemma \ref{l20}. Because the contour may be chosen arbitrarily close to the real axis,
these conclusions hold on all of $\mathbb C\setminus\mathbb R$.
To verify the uniform exterior estimate in (D.1), fix a target contour
$C_0$ and take a compatible contour $C_1$ whose height is at most half
that of $C_0$. Apply the Cauchy expansion defining the coefficients
$A_k$ on $C_1$, and identify its exterior factor with $\Phi_A$ by
Lemma~\ref{l19}. This gives
\[
\Phi_A(z)=\sum_{k=1}^{N+1}A_kz^{-k}+z^{-N-1}\Psi(z),
 \qquad \Psi\in H(D_{1,-})^{2\times2}.
\]
For $z\in D_{0,-}$, Lipschitz geometry and Cauchy point evaluation give
\[
\operatorname{dist}(z,C_1)\ge c\langle z\rangle^{-(N-1)},\qquad
 |\Psi(z)|\le C\operatorname{dist}(z,C_1)^{-1/2}\|\Psi\|_2.
\]
The remainder is therefore $O(|z|^{-(N+3)/2})$, so
$\Phi_A=O(z^{-1})$ and $m_A=O(z^{-1})$ on the target exterior.
This proves (D.1). Finally, (\ref{E3.3}) is the Schwarz--Pick inequality for the map from a half-plane to the unit disk used in the proof of Lemma \ref{l27}. Equality at one point would force that map to be a conformal automorphism. Such an automorphism has a unimodular limit at infinity, contradicting $m_A(iy)\to0$ along the corresponding imaginary ray. Thus (\ref{E3.3}) is strict, and (\ref{E3.1}) gives $\tau_{A}(r_{\zeta})>0$. \bigskip
\end{proof}

\begin{lemma}
	\label{l30} Let $\zeta_{1},\zeta_{2}\in\mathbb{C}\setminus\mathbb{R}$ and $A\in\mathcal{A}_{N,d}^{inv}$. Then
	\[
\tau_{A}(r_{\zeta_{1}}r_{\zeta_{2}}) \geq 0.
\]
\end{lemma}

\begin{proof}
Write $m=m_A$, $F=I+\Phi_A=(f_{jk})$, and $\Delta=\det F$.
By Lemmas~\ref{l20} and~\ref{l29}, $m$ is holomorphic on $\mathbb C\setminus\mathbb R$,
$|m|<1$, and $f_{11}$ has no zeros there. Indeed, a zero of $f_{11}$
would also be a zero of $f_{12}=mf_{11}$, contrary to $\det F\ne0$.
The defocusing reduction gives
\[
f_{22}(\overline z)=\overline{f_{11}(z)},\qquad
 \Delta(\overline z)=\overline{\Delta(z)}.
\]
All conjugations of matrices and evaluated scalar quantities below are
ordinary complex conjugations.

If $\zeta_2=\overline{\zeta_1}$, then
$r_{\zeta_1}r_{\zeta_2}=1$, so the tau-function equals one.
For the moment, suppose also that $\zeta_1\ne\zeta_2$.
Thus the four points $\zeta_1,\zeta_2,\overline{\zeta_1},
\overline{\zeta_2}$ are distinct. Choose a compatible admissible contour
with all four points in its exterior, and put $\mu_j=m(\zeta_j)$.
Define
\[
C_{jh}=\frac{1-\mu_j\overline{\mu_h}}
 {2i(\overline{\zeta_h}-\zeta_j)},\qquad j,h\in\{1,2\},
\]
\[
D_{jh}=\begin{cases}
 \displaystyle\frac{\mu_j-\mu_h}{2i(\zeta_h-\zeta_j)},&j\ne h,\\[2mm]
 \displaystyle-\frac{m'(\zeta_j)}{2i},&j=h,
 \end{cases}
 \qquad j,h\in\{1,2\},
\]
and set
\[
\Xi=\begin{pmatrix}C&D\\ \overline D&\overline C\end{pmatrix}.
\]
Since $C=C^*$ and $D=D^{\mathsf T}$, the matrix $\Xi$ is Hermitian.

We first identify its determinant with the tau-function, keeping the
indices in (\ref{E2.25}) explicit. In that formula $h$ is the row index and $j$
is the column index in every block. Substituting
$\eta_h=\overline{\zeta_h}$ and
$m_2(\eta_h)=\overline{\mu_h}$ shows that the matrix in (\ref{E2.25}) is
\[
K=2i\begin{pmatrix}
 \overline C&-\overline D\\ D&-C
 \end{pmatrix}
 =2i\,\overline\Xi\begin{pmatrix}I_2&0\\0&-I_2\end{pmatrix}.
\]
Consequently,
\[
\det K=16\det\Xi,\qquad
 \tau_A(r_{\zeta_1}r_{\zeta_2})
 =16\Lambda_1\Lambda_2\det\Xi.
\]
Both scalar factors are strictly positive. Indeed, the reality condition
on $r=r_{\zeta_1}r_{\zeta_2}$ gives $b_h=\overline{a_h}$ in the
partial-fraction notation of Lemma~\ref{l21}. Thus, with
$Q=(q_{\zeta_k}(\overline{\zeta_h}))_{k,h=1}^2$,
\[
\begin{split}
 \Lambda_1&=
 \left(\prod_{h=1}^2|a_h|^2\right)^2
 \left(\prod_{h=1}^2|\zeta_h|^2\right)|\det Q|^2>0,\\
 \Lambda_2&=\prod_{h=1}^2
 \frac{|f_{11}(\zeta_h)|^4}{|\Delta(\zeta_h)|^2}>0.
 \end{split}
\]
Here the residues $a_h$ are nonzero, and $Q$ is a nonsingular Cauchy
matrix because the four points are distinct. It remains to prove
$\det\Xi\ge0$.

\medskip
\noindent\emph{Case 1: $(\operatorname{Im}\zeta_1)
(\operatorname{Im}\zeta_2)>0$.}
First suppose $\zeta_1,\zeta_2\in\mathbb C_+$.
The function
\[
\varphi(z)=i\,\frac{1-m(z)}{1+m(z)}
\]
is Herglotz on $\mathbb C_+$, so
\[
\varphi(z)=\alpha+\beta z+
 \int_{\mathbb R}\left(\frac1{\lambda-z}
 -\frac{\lambda}{1+\lambda^2}\right)\sigma(d\lambda),
 \qquad \alpha\in\mathbb R,\quad\beta\ge0,
\]
where $\sigma$ is positive and
$\int_{\mathbb R}(1+\lambda^2)^{-1}\sigma(d\lambda)<\infty$.
In $\mathcal H=\mathbb C\oplus L^2(\sigma)$, with inner product linear
in its first argument, put
\[
f_j=\left(
 \frac{\sqrt\beta}{i+\varphi(\zeta_j)},\quad
 \frac1{(\lambda-\zeta_j)(i+\varphi(\zeta_j))}
 \right),\qquad j=1,2.
\]
Subtracting the Herglotz representations and using
$m=(i-\varphi)/(i+\varphi)$ gives
\begin{align*}
 C_{jh}
 &=\frac{\displaystyle\beta+
 \int_{\mathbb R}\frac{\sigma(d\lambda)}
 {(\lambda-\zeta_j)(\lambda-\overline{\zeta_h})}}
 {(i+\varphi(\zeta_j))\overline{(i+\varphi(\zeta_h))}}
 =(f_j,f_h),\\
 D_{jh}
 &=\frac{\displaystyle\beta+
 \int_{\mathbb R}\frac{\sigma(d\lambda)}
 {(\lambda-\zeta_j)(\lambda-\zeta_h)}}
 {(i+\varphi(\zeta_j))(i+\varphi(\zeta_h))}
 =(f_j,\overline{f_h}).
\end{align*}
The diagonal formula follows by differentiation. Hence $\Xi$ is the
Gram matrix of
$f_1,f_2,\overline{f_1},\overline{f_2}$, and $\det\Xi\ge0$.
If both points lie in $\mathbb C_-$, apply the preceding argument to
$\widetilde m(z)=m(-z)$ at $-\zeta_1,-\zeta_2\in\mathbb C_+$.
Its matrices are $-C,-D$, so its Gram matrix is $-\Xi$.
Since $\Xi$ has size four, this again gives $\det\Xi\ge0$.

\medskip
\noindent\emph{Case 2: $(\operatorname{Im}\zeta_1)
(\operatorname{Im}\zeta_2)<0$.}
Relabel the points so that $\operatorname{Im}\zeta_1>0$ and
$\operatorname{Im}\zeta_2<0$, and write
\[
C=\begin{pmatrix}a&b\\\overline b&c\end{pmatrix},\qquad
 D=\begin{pmatrix}r&s\\s&t\end{pmatrix}.
\]
By Lemma~\ref{l29} and (\ref{E3.1}),
\[
a>0,\qquad c<0,\qquad |r|<a,\qquad |t|<-c.
\]
Indeed,
\[
a=\frac{1-|m(\zeta_1)|^2}{4\operatorname{Im}\zeta_1},\quad
 c=\frac{1-|m(\zeta_2)|^2}{4\operatorname{Im}\zeta_2},\quad
 |r|=\frac{|m'(\zeta_1)|}{2},\quad
 |t|=\frac{|m'(\zeta_2)|}{2},
\]
and (\ref{E3.1}), together with $\tau_A(r_{\zeta_j})>0$, gives the strict
inequalities above. Simultaneously permute the rows and columns of
$\Xi$ into the order $(1,3,2,4)$. The resulting matrix is
\[
\begin{pmatrix}P&B\\B^*&-Q_0\end{pmatrix},\qquad
 B=\begin{pmatrix}b&s\\\overline s&\overline b\end{pmatrix},
\]
where
\[
P=\begin{pmatrix}a&r\\\overline r&a\end{pmatrix}>0,\qquad
 Q_0=\begin{pmatrix}-c&-t\\-\overline t&-c\end{pmatrix}>0.
\]
Its Schur complement therefore gives
\[
\begin{aligned}
 \det\Xi
 &=\det P\,\det(-Q_0-B^*P^{-1}B)\\
 &=\det P\,\det(Q_0+B^*P^{-1}B)>0,
 \end{aligned}
\]
because the second block has size two and
$Q_0+B^*P^{-1}B$ is positive definite.

Finally, the case $\zeta_2=\zeta_1$ follows by continuity on a fixed
compatible contour. The finite-rank Toeplitz corrections converge in
Hilbert--Schmidt norm by the kernel estimates used in Lemma~\ref{l13}.
Their differences have uniformly bounded rank, so the convergence is
also in trace norm and the ordinary determinants converge.
This proves the stated non-negativity in every case. \bigskip
\end{proof}

We can now prove non-vanishing along the full defocusing action.

\begin{proposition}
	\label{p31} For every $A\in\mathcal{A}_{N,d}^{inv}$ and $g\in\Gamma_{N}^{\operatorname{real}}$, one has $\tau_{A}^{(2)}(g)>0$ and $G(g)A\in\mathcal{A}_{N,d}^{inv}$.
\end{proposition}

\begin{proof}
	We begin with the rational subgroup. By Lemma \ref{l29}, $\tau_{A}(r_{\zeta})>0$, so $G(r_{\zeta})A\in\mathcal{A}_{N}^{inv}$ on every compatible admissible
	contour having $\zeta,\overline\zeta$ in its exterior. Choose a new
	witnessing contour inside the original one and avoiding these two
	points; the same conclusion then holds on all smaller contours. The ordinary determinant is multiplicative for finite-rank perturbations, and hence
	\[
\tau_{G(r_{\zeta})A}(r_{\eta})
	=\frac{\tau_{A}(r_{\zeta}r_{\eta})}{\tau_{A}(r_{\zeta})}.
\]
	Lemma \ref{l30} gives $\tau_{G(r_{\zeta})A}(r_{\eta})\geq0$, and therefore $G(r_{\zeta})A\in\mathcal{A}_{N,d}^{inv}$. Lemma \ref{l29}, applied to the transformed symbol, upgrades the inequality to strict positivity. Under the reduction, the trace correction relating $\tau$ and $\tau^{(2)}$ is real, so the two determinants have the same sign. Induction therefore yields
	\begin{equation}
	G(r)A\in\mathcal{A}_{N,d}^{inv},\qquad \tau_{A}^{(2)}(r)>0 \label{E3.4}
	\end{equation}
	for every rational $r\in\Gamma_{0}^{\operatorname{real}}$; a unimodular constant factor has tau-function equal to one.

	Next let $g=e^{ih}\in\Gamma_{N}^{\operatorname{real}}$. Fix a witnessing
	contour $C$ for $A$, and take the approximants $r_k$ from
	Lemma~\ref{l23} on this contour. Lemma \ref{l13} and (\ref{E3.4}) give
	\[
\tau_{A}^{(2)}(g)=\lim_{k\to\infty}\tau_{A}^{(2)}(r_{k})\geq0.
\]
	If $g=re^{ih}$ is a general element of $\Gamma_{N}^{\operatorname{real}}$, apply this conclusion to $G(r)A$ and use the cocycle identity together with (\ref{E3.4}). Since its exponential correction is positive, this shows that
	\[
\tau_A^{(2)}(g)\geq0
	\qquad(g\in\Gamma_N^{\operatorname{real}}).
\]

	We now prove strict positivity, first for $g=e^{ih}$. Since $r_{k}^{-1}g\to1$ and $\tau_{A}^{(2)}(1)=1$, continuity gives $\tau_{A}^{(2)}(r_{k}^{-1}g)>0$ for all sufficiently large $k$. Fix such a $k$. Choose a witnessing contour $C$
	on which $r_k^{-1}g$ is admissible. Formula (\ref{E2.26}),
	which follows from Lemma~\ref{l22} and the finite-rank trace calculation,
	gives
	\[
\tau_{A|_{\widetilde C},\widetilde C}^{(2)}(r_k^{-1}g)
	 =\tau_{A,C}^{(2)}(r_k^{-1}g)>0
\]
	for every admissible $\widetilde C\subset\overline D_+$.
	Hence $T_{\widetilde C}(G(r_k^{-1}g)A)$ is invertible on all the weighted
	spaces in (\ref{E2.3}). Multiplication by $G(r_k^{-1}g)$ also preserves the
	analytic-contour condition, with analytic part $G(r_k^{-1}g)F$.
	The preceding non-negativity and the cocycle identity, computed on a
	sufficiently small such contour for each $\zeta$, imply
	\[
\tau_{G(r_k^{-1}g)A}^{(2)}(r_{\zeta})
	=\frac{\tau_{A}^{(2)}(r_k^{-1}gr_{\zeta})}{\tau_{A}^{(2)}(r_k^{-1}g)}
	\exp\!\left(E_{A}(r_k^{-1}g,r_{\zeta})\right)\geq0.
\]
	The ordinary determinant has the same sign, so $G(r_k^{-1}g)A\in\mathcal{A}_{N,d}^{inv}$. Applying (\ref{E3.4}) to this symbol and using the cocycle identity once more, we obtain
	\[
\tau_{A}^{(2)}(g)
	=\tau_{A}^{(2)}(r_k^{-1}g)\tau_{G(r_k^{-1}g)A}^{(2)}(r_k)
	\exp\!\left(-E_A(r_k^{-1}g,r_k)\right)>0.
\]
	Similarly,
	\[
\tau_{G(g)A}^{(2)}(r_{\zeta})
	=\frac{\tau_A^{(2)}(gr_{\zeta})}{\tau_A^{(2)}(g)}
	\exp\!\left(E_A(g,r_{\zeta})\right)\geq0,
\]
	and (\ref{E2.26}) transfers
	$\tau_A^{(2)}(g)>0$ to every smaller compatible contour.
	Thus $G(g)A\in\mathcal{A}_{N,d}^{inv}$.

	Finally, write an arbitrary $g$ as $g=re^{ih}$. By (\ref{E3.4}), $B=G(r)A$ belongs to $\mathcal{A}_{N,d}^{inv}$ and $\tau_A^{(2)}(r)>0$. Applying the exponential case to $B$ and then the cocycle identity gives $\tau_A^{(2)}(g)>0$ and $G(g)A=G(e^{ih})B\in\mathcal{A}_{N,d}^{inv}$. \bigskip
\end{proof}

\subsection{Proof of theorems}

\begin{proposition}
	\label{p32} Suppose $A \in \mathcal{A}_{N,d}^{inv}$. Then
	\[
m_{A}(z) = \begin{cases}
		m_{-}(z) & \text{for } z \in \mathbb{C}_{+}, \\
		m_{+}(z) & \text{for } z \in \mathbb{C}_{-},
	\end{cases}
\]
	where $m_{\pm}$ are the Weyl functions at $\pm\infty$.
\end{proposition}

\begin{proof}
	Recall $e_{x}(z)=e^{ixz}\in\Gamma_{1}^{\operatorname{real}}$. We identify $m_{A}$ with the Weyl--Titchmarsh function of the self-adjoint Dirac operator by an $L^{2}$ estimate for the associated Baker--Akhiezer solution.

	In the notation preceding (\ref{E2.21}), write
\[
\Psi(x,z)=G(e_x)^{-1}F(x,z)=AU_{G(e_x)A}^{(0)}=
\begin{pmatrix}
	f_{11} & f_{12} \\
	f_{21} & f_{22}
\end{pmatrix},
\]
and set $\boldsymbol f=\Psi(x,z)^T\boldsymbol e_1$. Since multiplication by $G(e_x)^{-1}$ multiplies the first row by the scalar $e^{-ixz}$, it does not change the quotient of its two entries. Hence
\[
m_{G(e_x)A}(z)=\frac{f_{12}(x,z)}{f_{11}(x,z)}.
\]
Equation (\ref{E2.21}) gives
\[
\left\{
\begin{array}
	[c]{ll}%
	i\partial_{x}f_{11} = zf_{11}-2\alpha_{21}f_{12} \\
	i\partial_{x}f_{12} = -zf_{12} + 2\alpha_{12}f_{11} %
\end{array}
\right. \text{.}
\]
For every $x$, Proposition~\ref{p31} and Lemma~\ref{l29} give
$|m_{G(e_x)A}(z)|<1$. Lemma~\ref{l20} ensures that the row
$(f_{11},f_{12})$ does not vanish. Thus, writing
$m_1(x,z)=m_{G(e_x)A}(z)$, we have
\[
J(x)=|f_{11}(x,z)|^2-|f_{12}(x,z)|^2
     =|f_{11}(x,z)|^2(1-|m_1(x,z)|^2)>0.
\]
Since $\alpha_{12}=\overline{\alpha}_{21}$, the displayed Dirac system yields
\[
J'(x)=2\operatorname{Im}z\,
       \bigl(|f_{11}(x,z)|^2+|f_{12}(x,z)|^2\bigr).
\]
For $z\in D_-\cap\mathbb C_-$ and $R>0$, therefore,
\[
\int_0^R\!\bigl(|f_{11}|^2+|f_{12}|^2\bigr)\,dx
 =\frac{J(0)-J(R)}{-2\operatorname{Im}z}
 \le\frac{J(0)}{-2\operatorname{Im}z}.
\]
Hence $\boldsymbol f\in L^2(\mathbb R_+;\mathbb C^2)$.
For $z\in D_-\cap\mathbb C_+$, integration over $[-R,0]$ gives
\[
\int_{-R}^0\!\bigl(|f_{11}|^2+|f_{12}|^2\bigr)\,dx
 =\frac{J(0)-J(-R)}{2\operatorname{Im}z}
 \le\frac{J(0)}{2\operatorname{Im}z},
\]
so $\boldsymbol f\in L^2(\mathbb R_-;\mathbb C^2)$.
If $m_{\pm}$ denote the Weyl functions at $\pm\infty$, respectively, then
\[
m_{A}(z)=m_{1}(0,z) = \frac{f_{12}(0,z)}{f_{11}(0,z)} = m_{+}(z) \quad \text{on } D_{-} \cap \mathbb{C}_{-}.
\]
	The left-half-line estimate similarly gives
\[
m_{A}(z)=m_{1}(0,z) = \frac{f_{12}(0,z)}{f_{11}(0,z)} = m_{-}(z) \quad \text{on } D_{-} \cap \mathbb{C}_{+}.
\]
	By Lemma \ref{l29}, $m_A$ is holomorphic on $\mathbb C\setminus\mathbb R$, while $m_{-}$ and $m_{+}$ are holomorphic in the corresponding half-planes. The identity theorem therefore extends the two identities from $D_{-}\cap\mathbb C_{+}$ and $D_{-}\cap\mathbb C_{-}$ to all of $\mathbb C_{+}$ and $\mathbb C_{-}$, respectively. \bigskip
\end{proof}

\begin{proof}[Proof of Theorem \ref{t1} (defocusing case)]
	Let $q\in\mathcal Q_{N,d}$ and choose $A\in\mathcal A_{N,d}^{inv}$ such that
	\[
q(x)=2\alpha_{21}(G(e_x)A).
\]
	Proposition \ref{p31} shows that
	$G(g)A\in\mathcal A_{N,d}^{inv}$ for every
	$g\in\Gamma_N^{\operatorname{real}}$.  Consequently, the formula defining
	$\operatorname{dNLS}(g)q$ is meaningful for every $g$, and its value again
	belongs to $\mathcal Q_{N,d}$.

	We next show that this value is independent of the representing symbol.
	Suppose that $B\in\mathcal A_{N,d}^{inv}$ is another symbol satisfying
	\[
q(x)=2\alpha_{21}(G(e_x)B).
\]
	By Proposition \ref{p32},
	\[
m_A=m_B.
\]
	Fix $g\in\Gamma_N^{\operatorname{real}}$. Choose a contour inside
	witnessing contours for $A$ and $B$, with all rational zeros and poles of
	$g$ in its exterior. The defocusing definition and Lemma~\ref{l19}
	allow restriction to this contour without changing the exterior data.
	Proposition \ref{p31}
	guarantees the invertibility required in Lemma \ref{l24}; the rational
	one-factor iteration and its limiting argument therefore give
	\[
m_{G(g)A}=m_{G(g)B}.
\]
	Applying Proposition \ref{p32} once more, these are the Weyl functions of
	the two transformed potentials
	\[
q_g^A(x)=2\alpha_{21}(G(e_x)G(g)A),
		\qquad
		q_g^B(x)=2\alpha_{21}(G(e_x)G(g)B).
\]
	The uniqueness of the correspondence between the Weyl functions and the potentials yields
	$q_g^A=q_g^B$.  Thus $\operatorname{dNLS}(g)q$ does not depend on the
	choice of $A$.

	If $g_1,g_2\in\Gamma_N^{\operatorname{real}}$, then $G(g_2)A$ represents
	$\operatorname{dNLS}(g_2)q$, and hence
	\begin{align*}
		\bigl(\operatorname{dNLS}(g_1)
		\operatorname{dNLS}(g_2)q\bigr)(x)
		&=2\alpha_{21}\bigl(G(e_x)G(g_1)G(g_2)A\bigr)\\
		&=2\alpha_{21}\bigl(G(g_1g_2e_x)A\bigr)\\
		&=\bigl(\operatorname{dNLS}(g_1g_2)q\bigr)(x).
	\end{align*}
	Since $G(1)A=A$, the identity element acts trivially.  This proves the
	flow assertion.

	It remains to identify the distinguished time flow.  For $N\geq2$, put
	\[
a(t,x)=\alpha_{21}\bigl(G(e^{ixz+itz^{2}})A\bigr),
		\qquad q(t,x)=2a(t,x).
\]
	The functions $e^{ixz}$ and $e^{itz^{2}}$ belong to the required groups.
	Proposition~\ref{p31} supplies the invertibility hypothesis of
	Proposition~\ref{p26} for every $(t,x)\in\mathbb R^2$.
	Its defocusing conclusion gives
	\[
i\partial_ta=-\dfrac{1}{2}\partial_x^2a+4|a|^2a.
\]
	Multiplication by $2$ therefore gives
	\[
i\partial_tq=-\dfrac{1}{2}\partial_x^2q+|q|^2q.
\]
	At $t=0$, the definition of $\mathcal Q_{N,d}$ gives $q(0,x)=q(x)$.
	Finally, apply the finite-parameter assertion of Lemma~\ref{l15}
    with $(h_1,h_2)=(z,z^2)$, $n=0$, and $\ell=1$.
    It gives jointly continuous derivatives
    $\partial_x^j\partial_t^kq$ whenever $j+2k\leq N$.
    In particular, $q(t,\cdot)\in C^N(\mathbb R)$ and
    $q(\cdot,x)\in C^{\lfloor N/2\rfloor}(\mathbb R)$.
    Since $N\geq2$, both $q_{xx}$ and $q_t$ are jointly continuous,
    so the equation just obtained holds classically, including at $t=0$. \bigskip
\end{proof}

\begin{proof}[Proof of Theorem \ref{t2} (defocusing inclusion)]
	Let $q\in W^{2N+1,\infty}(\mathbb R)$ and consider the self-adjoint Dirac operator
	\[
D_q^{d}=i\sigma_3\partial_x+
	\begin{pmatrix}0&q\\ \overline q&0\end{pmatrix}.
\]
	Write $m_{\pm}^{q}(x,z)$ for its Weyl coordinates at $\pm\infty$, with the convention used in Proposition \ref{p32}, and set
	\[
m^{d}(x,z)=
	\begin{cases}
		m_-^{q}(x,z),&z\in\mathbb C_+,\\
		m_+^{q}(x,z),&z\in\mathbb C_-.
	\end{cases}
\]
	The Weyl-disk property gives $|m^{d}(x,z)|<1$. Apply \cite[Theorem~3]{SZ} with its regularity index equal to $2N+1$ and its boundary index equal to $N-1$. The expansion has $N+1$ terms and remainder of order $N+2$, since
    \[
(2N+1)-(N-1)-1=N+1,\qquad
        (N-1)-(2N+1)=-N-2.
\]
    For every admissible exterior domain,
	\[
m^{d}(x,z)=\sum_{k=1}^{N+1}c_k(x)z^{-k}+O(z^{-N-2}),
	\qquad c_1(x)=\frac{\overline{q(x)}}{2}.
\]
	Consequently,
	\[
M_d(z)=\begin{pmatrix}1&m^{d}(0,z)\\
	\overline{m^{d}}(0,z)&1\end{pmatrix}
\]
	belongs to $\mathcal M_{N,d}$. Lemma \ref{l27} therefore gives $M_d\in\mathcal A_{N,d}^{inv}$ and $m_{M_d}=m^{d}(0,\cdot)$.

	Let $q_{M_d}(x)=2\alpha_{21}(G(e_x)M_d)$. By Proposition \ref{p32}, the two half-line Weyl functions of $q_{M_d}$ are precisely the two components of $m_{M_d}$. They coincide with those of $q$ by construction. The uniqueness of the correspondence between the Weyl functions and the potentials now yields $q_{M_d}=q$ almost everywhere, and hence everywhere after choosing the continuous representatives. Thus $q\in\mathcal Q_{N,d}$. \bigskip
\end{proof}

\section{Construction of the focusing NLS flow}\label{sec:focusing-construction}
In this section, we consider the reduction $A = \sigma_{4}^{-1}\overline{A}\sigma_{4}$, under which $\alpha_{21}$ satisfies the focusing NLS equation. In contrast to the defocusing case, the associated Dirac operator is non-self-adjoint, and its spectral theory is correspondingly more delicate; see, for example, \cite{FKRS}.

\subsection{Non-vanishing of the tau-function}
Unlike in the defocusing case, the argument here is based primarily on intrinsic properties of the tau-function rather than on positivity properties of the $m$-function.

\begin{lemma}\label{l33}
Let $A\in\mathcal A_N^{inv}(C)$ satisfy
$A=\sigma_4^{-1}\overline A\sigma_4$. Write
\[
I+\Phi_A=\begin{pmatrix}u&v\\-\overline v&\overline u\end{pmatrix},
 \qquad \Delta=u\overline u+v\overline v.
\]
For every $\zeta\in D_-$,
\begin{equation}\label{E4.1}
 \tau_A(r_\zeta)=
 \frac{\bigl(|u(\zeta)|^2+|v(\zeta)|^2\bigr)^2
 +4(\operatorname{Im}\zeta)^2|uv'-u'v|^2(\zeta)}
 {|\Delta(\zeta)|^2}>0.
\end{equation}
Consequently $\tau_A(r)>0$ for every $r\in\Gamma_0^{\mathrm{real}}(C)$.
\end{lemma}
\begin{proof}
Where $u(\zeta)u(\bar\zeta)\ne0$, insert $m_1=v/u$,
$m_2=-\overline{m_1}$ and $\eta=\bar\zeta$ in the one-factor case of
Lemma~\ref{l21}. The scalar prefactor cancels to
$|u(\zeta)|^4/|\Delta(\zeta)|^2$ and the remaining factor is
\[
(1+|m_1(\zeta)|^2)^2
 +4(\operatorname{Im}\zeta)^2|m_1'(\zeta)|^2.
\]
Using $m_1'=(uv'-u'v)/u^2$ gives
(\ref{E4.1}). The high-half-plane normalization
shows that $u$ is not identically zero on either exterior component.
On compact subsets of $D_-$, the rational finite-rank corrections vary
continuously in trace norm with $\zeta$. Since $\Delta\ne0$ by
Lemma~\ref{l20}, continuity extends the identity across all exceptional
coordinate points. Its numerator is strictly positive because the row
$(u,v)$ cannot vanish.

The determinant criterion makes $B=G(r_\zeta)A$ invertible on every
required weight, with the same reduction. Apply the same formula to $B$.
The ordinary determinant cocycle then gives
\[
\tau_A(r_\zeta r_\eta)=\tau_A(r_\zeta)\tau_B(r_\eta)>0.
\]
Every real rational multiplier is a finite product of these factors and
a unimodular constant, whose tau-function is one. Induction proves the
claim, including repeated and cancelling factors.
\end{proof}

\begin{proposition}\label{p34}
Let $A\in\mathcal A_N^{inv}(C)$ satisfy the focusing reduction.
For every $g\in\Gamma_N^{\mathrm{real}}(C)$,
\[
\tau_A^{(2)}(g)>0,
 \qquad G(g)A\in\mathcal A_{N,f}^{inv}.
\]
\end{proposition}
\begin{proof}
First let $g=e^{ih}$. Choose the rational approximants $r_k$ of
Lemma~\ref{l23} and put $s_k=g/r_k$. The functions $s_k,s_k^{-1}$
are uniformly bounded on a neighborhood of the contour, their
derivatives have the bounds in (\ref{E2.12}), and $s_k\to1$.
Lemma~\ref{l13} therefore gives $\tau_A^{(2)}(s_k)\to1$.
Choose $k$ for which this real number is positive, and put $B=G(s_k)A$.
Then $B$ is invertible and focusing-reduced, so Lemma~\ref{l33}
gives $\tau_B(r_k)>0$.

All trace corrections are real. Indeed, the antiunitary map
$Jf(z)=\sigma_4\overline{f(\bar z)}$ commutes with the normalized
Toeplitz operators under this reduction. For a trace-class combination
$K$ of their perturbations, $JKJ^{-1}=K$ implies
$\operatorname{tr}K=\overline{\operatorname{tr}K}$. This applies both to
the ordinary-to-modified determinant correction and to the product of
the two Hilbert--Schmidt perturbations in (\ref{E2.6}). Thus
\[
\tau_A^{(2)}(g)
 =\tau_A^{(2)}(s_k)\tau_B^{(2)}(r_k)
   e^{-E_A(s_k,r_k)}>0.
\]
For $g=re^{ih}$, first apply Lemma~\ref{l33} to $r$, then the
exponential argument to $G(r)A$. The transformed symbol is invertible
and focusing-reduced; Lemma~\ref{l33} gives its defining single-factor
positivity. Hence it belongs to $\mathcal A_{N,f}^{inv}$.
\end{proof}

Define
\[
\mathcal{M}_{N,f} = \left\{ M = \begin{pmatrix}
	1 & m_{1} \\
	-\overline{m}_{1} & 1
\end{pmatrix} \text{; } m_{1} \text{ satisfies (F.1) and (F.2)}
\right\}
\]
where the following conditions hold.

\smallskip
\noindent\textup{(F.1)} The function $m_1$ is analytic on an open neighborhood of
$\overline D_-$ for some admissible contour $C$, and
$1+m_1\overline m_1$ has no zeros on $\overline D_-$.

\smallskip
\noindent\textup{(F.2)} The function $m_{1}$ has the asymptotic expansion
\[
m_{1}(z) = \sum_{1 \leq k \leq N+1} a_{k}z^{-k} + O(z^{-N-2}) \quad\text{on } D_{-}.
\]
For $M \in \mathcal{M}_{N,f}$ and $g \in \Gamma_{N}^{\operatorname{real}}(C)$, the symmetry
\[
G(g)M = \sigma_{4}^{-1}\overline{G(g)M}\sigma_{4}
\]
is preserved. The analogue of Lemma \ref{l27} is as follows.

\begin{lemma}
	\label{l35} Let $M \in \mathcal{M}_{N,f}$. Then $T(M)$ has a bounded inverse on $\boldsymbol{H}_{n}(D_{+})$ for $1\leq n \leq N+1$. Moreover, $\Phi_{M}=M-I$, $m_{M}=m_{1}$, and $\mathcal{M}_{N,f} \subset \mathcal{A}_{N,f}^{inv}$.
\end{lemma}

\begin{proof}
	Condition (F.2) gives $M\in\mathcal{A}_{N+2}(C)$, while (F.1) gives
	\[
M^{-1}=\frac{1}{1+m_{1}\overline m_{1}}
	\begin{pmatrix}
		1 & -m_{1}\\
		\overline m_{1} & 1
	\end{pmatrix}
	\in\mathcal{A}_{N+2}(C).
\]
	Both $M$ and $M^{-1}$ preserve $\boldsymbol H(D_{-})$. Repeating the Toeplitz-factorization argument in the proof of Lemma \ref{l27} therefore yields
	\[
T(M)T(M^{-1})=T(M^{-1})T(M)=I.
\]
	Since $T(M)I=I$, we obtain
	\[
\Phi_{M}=MT(M)^{-1}I-I=M-I,
	\qquad m_{M}=m_{1}.
\]
The defining symmetry of $M$ is the focusing reduction.
Lemma~\ref{l33} now gives $\tau_M(r_\zeta)>0$ for every
$\zeta\in D_-$, verifying condition (iii) in (\ref{E2.4}).
Thus $M\in\mathcal A_{N,f}^{inv}$.
\end{proof}

\subsection{Proof of theorems}

Throughout this subsection, each focusing symbol is kept on its own
fixed admissible contour. Symbols on different contours will be compared
through their $m$-functions on the common exterior domain, without
restricting either symbol to a smaller contour.

For $A\in\mathcal A_{N,f}^{inv}$, put
\[
q_A(x)=2\alpha_{21}(G(e_x)A),
\]
and, for every $x\in\mathbb R$, set
\[
F_A(x,z)=G(e_x)A U_{G(e_x)A}^{(0)}=I+\Phi_{G(e_x)A}(z),
	\qquad
	\Psi_A(x,z)=G(e_x)^{-1}F_A(x,z).
\]
Proposition~\ref{p34} makes these functions well defined for all $x$.
The coefficient assertion of Lemma~\ref{l15}, with $h(z)=z$ and
$n=0$, $\ell=1$, gives $q_A\in C^N(\mathbb R)$, and the calculation leading
to (\ref{E2.21}) gives
\begin{equation}
	i\partial_x\Psi_A-\Psi_A B_{q_A}=z\Psi_A\sigma_3,
	\qquad
	B_{q_A}=\begin{pmatrix}0&-\overline{q_A}\\-q_A&0\end{pmatrix}.
	\label{E4.2}
\end{equation}
The focusing reduction gives
\[
F_A(x,z)=
 \begin{pmatrix}
 F_{A,11}(x,z)&F_{A,12}(x,z)\\
 -\overline{F_{A,12}}(x,z)&\overline{F_{A,11}}(x,z)
 \end{pmatrix}.
\]
We therefore retain only the first row quotient and write
\[
m_A(x,z)=\frac{F_{A,12}(x,z)}{F_{A,11}(x,z)},
 \qquad m_A(z)=m_A(0,z).
\]
Here, as throughout the paper, the bar includes reflection of the spectral
parameter:
\[
\overline m_A(x,z)=\overline{m_A(x,\overline z)},\qquad
 \frac{F_{A,21}(x,z)}{F_{A,22}(x,z)}=-\overline m_A(x,z).
\]
For each fixed $x$, restriction of the exterior Hardy functions to a
sufficiently high horizontal half-plane, followed by the usual Hardy-space
point-evaluation estimate, gives
\begin{equation}
	F_A(x,z)=I+O(z^{-1}),
	\qquad m_A(x,z)=O(z^{-1}),
	\label{E4.3}
\end{equation}
uniformly in both sufficiently high horizontal half-planes. To justify
the uniformity, use the expansion of $\Phi_{G(e_x)A}$ from Section~\ref{sec:equation-derivation}.
Its remainder is $z^{-N-1}$ times an exterior Hardy function; above a
horizontal line at positive distance from the contour, the Cauchy
point-evaluation estimate bounds that Hardy function uniformly.
In particular, the diagonal entries of $F_A$ are nonzero when
$|\operatorname{Im}z|$ is sufficiently large.

We use the finite-interval Weyl disks of
\cite[Definition~2.1, Remarks~2.3--2.4, and Definition~3.1]{FKRS}.
More explicitly, in the scalar case, for
\[
u'(s,\zeta)=\bigl(i\zeta\sigma_3+\sigma_3V(s)\bigr)u(s,\zeta),
	\qquad u(0,\zeta)=I,
\]
the disk $\mathcal N_V(\ell,\zeta)$ consists of the values
\[
\frac{(0\ \ 1)u(\ell,\zeta)^{-1}P(\zeta)}
	{(1\ \ 0)u(\ell,\zeta)^{-1}P(\zeta)},
	\qquad
	P(\zeta)^*P(\zeta)>0,
	\quad P(\zeta)^*\sigma_3P(\zeta)\geq0.
\]
An analytic selection of these disks on a half-plane
$\mathbb C_M=\{\operatorname{Im}\zeta>M\}$ is a Weyl function of the system
on $[0,\ell]$; see also \cite[Remark~2.4]{FKRS}.

\begin{proposition}
	\label{p36}
	Let $A\in\mathcal A_{N,f}^{inv}$.  For every $a\in\mathbb R$ and
	$\ell>0$, the functions
	\[
m_A(a,-\zeta),
		\qquad -\overline m_A(a,-\zeta)
\]
	are, on a sufficiently high half-plane $\mathbb C_M$, finite-interval
	Weyl functions for the restrictions of $q_A$ to $[a,a+\ell]$ and
	$[a-\ell,a]$, respectively, the latter after reflection.  Moreover, for
	$A,B\in\mathcal A_{N,f}^{inv}$,
	\[
q_A=q_B\quad\Longleftrightarrow\quad m_A=m_B,
\]
	where equality of the $m$-functions is understood on their common exterior
	domain (and hence by meromorphic continuation).
\end{proposition}

\begin{proof}
	Fix $a\in\mathbb R$ and $\ell>0$, and set
	\[
Q_{q_A}^f(x)=\begin{pmatrix}0&q_A(x)\\-\overline{q_A(x)}&0\end{pmatrix},
		\qquad
		V_{A,a}(s)=iQ_{q_A}^f(a+s).
\]
	The transpose of the first row of $\Psi_A(a+s,-\zeta)$, denoted by
	$y_1(s,\zeta)$, satisfies
	\[
y_1'=\bigl(i\zeta\sigma_3+\sigma_3V_{A,a}(s)\bigr)y_1.
\]
	Let $u_{A,a}(s,\zeta)$ be the normalized fundamental matrix of this
	system and put $P_1(\zeta)=y_1(\ell,\zeta)$.  Since row quotients are not
	affected by the scalar exponential in $G(e_x)^{-1}$,
	\[
u_{A,a}(\ell,\zeta)^{-1}P_1(\zeta)
		=c_1(\zeta)\binom{1}{m_A(a,-\zeta)},
		\qquad c_1(\zeta)\ne0.
\]
	At the other endpoint,
	\[
P_1(\zeta)^*\sigma_3P_1(\zeta)
		=|(P_1)_1(\zeta)|^2
		\bigl(1-|m_A(a+\ell,-\zeta)|^2\bigr).
\]
	By (\ref{E4.3}), this expression is positive throughout
	$\mathbb C_M$ once $M$ is sufficiently large.  We also choose
	$M>\|q_A\|_{L^\infty(a,a+\ell)}$, as required in the construction of the
	finite Weyl disk.  Thus $P_1$ has property-$\sigma_3$, and the displayed
	linear-fractional formula gives
	\[
m_A(a,-\zeta)\in
		\mathcal N_{V_{A,a}}(\ell,\zeta),
		\qquad \zeta\in\mathbb C_M.
\]
	The dependence on $\zeta$ is analytic there, because the diagonal entries
	of $F_A$ are $1+O(\zeta^{-1})$.

	For the left interval, let
	\[
S=\begin{pmatrix}0&1\\1&0\end{pmatrix},
		\qquad
		q_A^{a,\ell}(s)=-\overline{q_A(a-s)},
		\qquad 0\leq s\leq\ell.
\]
	If
	\[
y_2(s,\zeta)=S\Psi_A(a-s,-\zeta)^T\binom{0}{1},
\]
	then direct substitution in (\ref{E4.2}) gives
	\[
y_2'=\bigl(i\zeta\sigma_3+
		\sigma_3 iQ_{q_A^{a,\ell}}^f(s)\bigr)y_2.
\]
	Its projective coordinate at $s=0$ is
	$-\overline m_A(a,-\zeta)$ and at $s=\ell$ is
	$-\overline m_A(a-\ell,-\zeta)$. In particular, the latter has modulus
	$|m_A(a-\ell,-\overline\zeta)|$, which is less than one on a sufficiently
	high $\mathbb C_M$ by the upper-half-plane part of
	(\ref{E4.3}). Choose also
	$M>\|q_A\|_{L^\infty(a-\ell,a)}$.
	The same endpoint calculation therefore proves the asserted
	finite-disk property for the reflected potential.

	We next prove direct uniqueness.  Suppose $q_A=q_B=q$.  The matrices
	$\Psi_A$ and $\Psi_B$ solve the same equation
	(\ref{E4.2}), and hence
	\[
C(z)=\Psi_A(x,z)\Psi_B(x,z)^{-1}
\]
	is independent of $x$.  On a common sufficiently high half-plane,
	\[
F_A(x,z)F_B(x,z)^{-1}=G(e_x)(z)C(z)G(e_x)(z)^{-1}.
\]
	Writing $C=(c_{jk})$, (\ref{E4.3}) gives, for every fixed
	$x\in\mathbb R$,
	\[
e^{2ixz}c_{12}(z)=O(z^{-1}),
		\qquad e^{-2ixz}c_{21}(z)=O(z^{-1}).
\]
	On a fixed sufficiently high half-plane $\mathbb C_{M_0}$,
	$F_A(0,\cdot)-I$ and $F_B(0,\cdot)-I$ belong to $H^2$,
	and $F_B(0,\cdot)^{-1}$ is bounded. Hence
	$c_{12},c_{21}\in H^2(\mathbb C_{M_0})$.
	For every $r>0$, take $x=-r$ in the first estimate and $x=r$ in the
	second. On a possibly higher half-plane $\mathbb C_{M_r}$ this gives
	\[
e^{-2irz}c_{12}(z),\quad e^{-2irz}c_{21}(z)
	 \in H^2(\mathbb C_{M_r}).
\]
	The Hardy-space Fourier representation writes each $c=c_{12},c_{21}$ as
	\[
c(z)=\frac1{\sqrt{2\pi}}\int_0^\infty
	 e^{i(z-iM_0)t}a(t)\,dt,\qquad a\in L^2(0,\infty).
\]
	Thus the representing functions of $c_{12}$ and $c_{21}$ are supported in
	$[0,\infty)$. Raising the boundary line multiplies these functions by
	$e^{-(M_r-M_0)t}$ and does not change their supports; multiplication
	by $e^{-2irz}$ translates their supports by $-2r$.
	The last display therefore forces the original supports to be contained
	in $[2r,\infty)$. Since $r>0$ is arbitrary, both representing functions
	vanish. Thus $c_{12}=c_{21}=0$ in the upper component of the common
	exterior. The focusing symmetry gives the same conclusion in the lower
	component. Consequently $C$ is diagonal, so
	$\Psi_A=C\Psi_B$ implies equality of the first row quotients:
	$m_A=m_B$.

	Conversely, suppose $m_A=m_B$.  Fix $\ell>0$ and choose a common
	$\mathbb C_M$ in the first part of the proof.  The common function
	$m_A(0,-\zeta)=m_B(0,-\zeta)$ is a finite-interval Weyl function
	for both $q_A$ and $q_B$ on $[0,\ell]$.  Its difference from itself is
	zero, so the hypothesis of the finite-interval inverse uniqueness theorem
	\cite[Theorem~3.15]{FKRS} is satisfied for every
	$0<r<\ell$.  It follows that $q_A=q_B$ almost everywhere on
	$[0,\ell]$.  Applying the same argument to the second coordinate and the
	reflected potentials gives equality on $[-\ell,0]$.  Since $\ell$ is
	arbitrary and the two Sato potentials are continuous, $q_A=q_B$ on
	$\mathbb R$. \bigskip
\end{proof}

\begin{proof}[Proof of Theorem \ref{t1} (focusing case)]
    Let $q\in\mathcal Q_{N,f}$ and choose a representing symbol
    $A\in\mathcal A_{N,f}^{inv}$ on an admissible contour $C_A$.
    For $g=e^{ih}\in\Gamma_N^{\operatorname{sub}}$, both $g$ and $g^{-1}$
    are bounded and analytic on the interior of $C_A$.
    Proposition~\ref{p34}, applied on this fixed contour, gives
    \[
G(g)A\in
        \mathcal A_N^{inv}(C_A)\cap\mathcal A_{N,f}^{inv}.
\]
    Thus $q_{G(g)A}\in\mathcal Q_{N,f}$.

    Let $B\in\mathcal A_{N,f}^{inv}$, on a possibly different
    contour $C_B$, represent the same potential.
    Proposition~\ref{p36} gives $m_A=m_B$ on their common
    exterior. The focusing reduction determines the second
    projective coordinate from the first by
    $-\overline m_A$ and $-\overline m_B$, respectively.
    Thus equality of the scalar $m$-functions used in this
    subsection gives equality of the pairs required in
    Lemma~\ref{l24}.

The polynomial multiplier $g$ is admissible on both original contours.
Proposition~\ref{p34} gives positivity of the modified tau-functions
of $A$ and $B$ for $g$ and for every real rational multiplier admissible
on both contours. Both non-vanishing hypotheses of Lemma~\ref{l24}
therefore hold, and that lemma gives $m_{G(g)A}=m_{G(g)B}$ on the
common exterior. Proposition~\ref{p36} then yields
$q_{G(g)A}=q_{G(g)B}$. Thus $\operatorname{fNLS}(g)q$ is independent
of the representing symbol and contour.

	If $g_1,g_2\in\Gamma_N^{sub}$, all multipliers can be applied on
	the original contour of $A$. Their commutativity gives
	\begin{align*}
	 \bigl(\operatorname{fNLS}(g_1)
	 \operatorname{fNLS}(g_2)q\bigr)(x)
	 &=2\alpha_{21}\bigl(G(e_x)G(g_1)G(g_2)A\bigr)\\
	 &=\bigl(\operatorname{fNLS}(g_1g_2)q\bigr)(x).
	\end{align*}
	The identity acts trivially, and $e^{-ih}$ is the inverse of $e^{ih}$.
	This proves the group-action assertion.

	For $N\geq2$, put
	\[
a(t,x)=\alpha_{21}\bigl(G(e^{ixz+itz^2})A\bigr),
		\qquad q(t,x)=2a(t,x).
\]
	Proposition~\ref{p34} supplies the invertibility hypothesis in
	Proposition~\ref{p26} for every $(t,x)\in\mathbb R^2$.
	Its focusing conclusion gives
	\[
i\partial_ta=-\frac12\partial_x^2a-4|a|^2a,
\]
	and hence
	\[
i\partial_tq=-\frac12\partial_x^2q-|q|^2q.
\]
	At $t=0$ we recover the prescribed initial value. Finally, the
    finite-parameter assertion of Lemma~\ref{l15}, with
    $(h_1,h_2)=(z,z^2)$, $n=0$, and $\ell=1$,
    gives jointly continuous derivatives $\partial_x^j\partial_t^kq$
    for $j+2k\leq N$. Thus
    $q(t,\cdot)\in C^N(\mathbb R)$ and
    $q(\cdot,x)\in C^{\lfloor N/2\rfloor}(\mathbb R)$.
    In particular, for $N\geq2$, the equation holds classically on
    $\mathbb R^2$. \bigskip
\end{proof}

\begin{proof}[Proof of Theorem \ref{t2} (focusing inclusion)]
	Let $q\in W^{2N+1,\infty}(\mathbb R)$ and consider
	\[
D_q^f=i\sigma_3\partial_x+
		Q_q^f,
		\qquad
		Q_q^f=\begin{pmatrix}0&q\\-\overline q&0\end{pmatrix}.
\]
	Because $q$ is globally bounded, the right and left half-line Weyl
	coordinates $m_+^q$ and $m_-^q$ exist in sufficiently low and high
	half-planes, respectively, by
	\cite[Proposition~2.2 and Definition~2.5]{FKRS}, after the change of
	spectral variable $\zeta=-z$.  Define
	\[
m^f(x,z)=
		\begin{cases}
			m_-^q(x,z),&z\in D_-\cap\mathbb C_+,\\
			m_+^q(x,z),&z\in D_-\cap\mathbb C_-.
		\end{cases}
\]
	By \cite[Theorem~4]{SZ}, used with regularity index $2N+1$ and boundary
	index $N-1$, the number of terms is $(2N+1)-(N-1)-1=N+1$
    and the remainder is $O(z^{-N-2})$. An admissible contour can be chosen,
    strictly inside the
	analyticity domains furnished by that theorem, such that
	\[
m^f(x,z)=\sum_{k=1}^{N+1}c_k(x)z^{-k}+O(z^{-N-2}),
		\qquad c_1(x)=-\frac{\overline{q(x)}}2,
\]
	uniformly on $D_-$.  In particular,
	\[
d(z)=1+m^f(0,z)\overline{m^f}(0,z)=1+O(z^{-2}).
\]
	Thus $d$ has no zeros for sufficiently large $|z|$. Its zeros in the
	initial closed exterior lie in a compact subset of the common
	analyticity domain, so there are only finitely many. A compact
	enlargement of $D_+$ excludes these zeros while leaving the
	high-energy tails unchanged. Choosing the new boundary away from
	the zeros, we obtain $d\ne0$ on the new $\overline D_-$.
	Consequently,
	\[
A(z)=M_f(0,z):=
		\begin{pmatrix}
			1&m^f(0,z)\\
			-\overline{m^f}(0,z)&1
		\end{pmatrix}
\]
	belongs to $\mathcal M_{N,f}$.  Lemma~\ref{l35} gives
	$A\in\mathcal A_{N,f}^{inv}$ and
	\[
m_A=m^f(0,\cdot).
\]
	Let $q_A(x)=2\alpha_{21}(G(e_x)A)$.

	Fix $\ell>0$ and choose $M$ larger than the bounds required for $q_A$ on
	$[-\ell,\ell]$ and for $q$ on the whole line.  For
	$\zeta\in\mathbb C_M$ we have $-\zeta\in\mathbb C_-$ and hence
	\[
m_A(0,-\zeta)=m^f(0,-\zeta)=m_+^q(0,-\zeta).
\]
	By Proposition~\ref{p36}, the left-hand side is a finite-interval
	Weyl function for $q_A$ on $[0,\ell]$.  The right-hand side is the genuine
	right half-line Weyl function for $q$; Proposition~2.2 of \cite{FKRS}
	places it in every finite Weyl disk, in particular in the disk for
	$[0,\ell]$.  Thus the same analytic Weyl function belongs to the
	finite-interval disks of $q_A$ and $q$.  The zero-difference case of
	\cite[Theorem~3.15]{FKRS} gives
	\[
q_A=q\quad\text{almost everywhere on }[0,\ell].
\]

	For the other half-line, set $q^{\rm ref}(s)=-\overline{q(-s)}$.
Its right Weyl coordinate obeys
\[
m_+^{q^{\rm ref}}(0,z)
 =-\overline{m_-^q(0,\bar z)}.
\]
Indeed, if $y_-$ is a left Weyl solution, then
$y_{\rm ref}(s,z)=\sigma_3\overline{y_-(-s,\bar z)}$ is the right
Weyl solution for the reflected potential. The corresponding Sato coordinate is
$-\overline m_A(0,z)$ by Proposition~\ref{p36}. Canonical equality
of these coordinates and the same finite-interval inverse theorem give
equality of the reflected potentials on $[0,\ell]$, hence
$q_A=q$ almost everywhere on $[-\ell,0]$.  Since $\ell$ is arbitrary,
	$q_A=q$ almost everywhere on $\mathbb R$, and hence everywhere after
	choosing the continuous representative of $q$.  Thus $A$ represents $q$
	and $q\in\mathcal Q_{N,f}$. \bigskip
\end{proof}

\section{Global dynamics in a non-decaying phase space}\label{sec:phase}

\subsection{Spectral notation and established identities}

In this section we set $N=2$. We recall the notation needed
from the preceding spectral construction.
An admissible contour is
\[
C=\{a\pm i\omega(a):a\in\R\},
 \quad \omega>0,\quad \omega'\in L^\infty,
 \quad \omega(a)\asymp\br{a}^{-1}.
\]
Its interior strip and two-component exterior are denoted by $D_+$ and
$D_-$. Compact enlargements of the strip are allowed. Write
$H_j=\boldsymbol H_j(D_+)$, $1\le j\le3$, for the weighted vector Hardy
spaces defined above, with norms equivalent to
$\norm{\abs z^{-j}u}_{L^2(C)}$. The componentwise Cauchy projection
$\mathfrak p_+$ is bounded on $L^2(C)$.

For a symbol $A=F+E$, where $F$ is bounded and analytic in $D_+$ and
$E(z)=O(z^{-4})$ on $C$, the Toeplitz operator is
\[
T(A)u=Fu+\mathfrak p_+(Eu).
\]
For $g(z)$ bounded and boundedly invertible in $D_+$ put
$G(g)=\diag(g,g^{-1})$. If $T(A)$ is invertible, put
\[
U_A=T(A)^{-1}I,
 \quad F_A=AU_A=I+\Phi_A,
 \quad \Phi_A(z)=A_1z^{-1}+\cdots.
\]
Matrix columns are treated separately in these formulas.

We use the following consequences of the preceding results.
\begin{enumerate}[label=\textup{(S\arabic*)},leftmargin=*]
\item
The canonical Weyl symbol of $v\in W^{5,\infty}(\mathbb R)$
is
\[
M_\kappa(v,z)=
 \begin{pmatrix}1&m_v(z)\\\kappa \overline{m_v}(z)&1\end{pmatrix},
 \qquad \overline{m}(z)=\overline{m(\bar z)}.
\]
Here $m_v=m_-^v$ in the upper component and $m_v=m_+^v$ in the lower
component, using the preceding half-line conventions.
Lemmas~\ref{l27} and~\ref{l35}, with the Weyl
expansions in the proof of Theorem~\ref{t2}, give invertibility
on the admissible contours used below. It represents $v$:
$2(A_1(G(e^{ixz})M_\kappa(v)))_{21}=v(x)$.
\item\label{input:nonzero}
By Propositions~\ref{p31} and~\ref{p34},
for canonical symbols and all real $s,t$, the operators
\[
T(G(e^{i(sz+tz^2)})M_\kappa(v))
\]
are invertible on $H_1,H_2,H_3$.
\item
The function
\[
Q_\kappa[v](t,x)
 =2\bigl(A_1(G(e^{i(xz+tz^2)})M_\kappa(v))\bigr)_{21}
\]
is a global classical solution of (\ref{EE2.1}) by
Theorem~\ref{t1} and Lemma~\ref{l15}. It is independent of
the representing symbol and contour. Its mixed derivatives of total
spectral degree at most $2$ are given by differentiating the coefficient
integrals in the weighted spaces. The group construction gives
\begin{equation}\label{E5.1}
 Q_\kappa[\T_yv](t,x)=Q_\kappa[v](t,x+y),
 \qquad (\T_yv)(x)=v(x+y).
\end{equation}
\end{enumerate}
The representation and contour independence in Theorem~\ref{t1}
justify the translation identity before any global spatial bound is used.

\subsection{Uniform spectral estimates on bounded initial-data sets}

\begin{proposition}\label{p37}
For every $R,T<\infty$ there is a finite $B_T(R)$ such that, for both
signs and every $v\in W^{5,\infty}(\R)$ with $p_5(v)\le R$,
\begin{equation}\label{E5.2}
 \sup_{|t|\le T}p_2(Q_\kappa[v](t,\cdot))\le B_T(R).
\end{equation}
The function $B_T(R)$ may be chosen nondecreasing in $T$ and $R$.
For each such $v$, the reconstructed solution belongs to
$C(\R;W^{2,\infty})\cap C^1(\R;L^\infty)$.
\end{proposition}

\begin{proof}
We prove the common-contour continuity needed for a compact maximum.
Set
\[
\mathcal B_R=\{v\in W^{5,\infty}(\R):p_5(v)\le R\},
 \qquad
 d(v,w)=\sum_{a=1}^\infty2^{-a}
 \min\{1,\norm{v-w}_{C^4([-a,a])}\}.
\]
Representatives through the fourth derivative are continuous.
Arzel\`a--Ascoli gives sequential precompactness in this metric.
The fifth distributional derivative of a limit remains bounded by $R$;
thus $\mathcal B_R$ is closed and compact. It is translation invariant.

\smallskip\noindent\textit{A common Weyl domain.}
The near-real-axis Dirac estimates of \cite[Theorems 3--4]{SZ}, with
regularity index $5$ and boundary index $1$, give
\begin{equation}\label{E5.3}
 m_v(x,z)=\sum_{a=1}^3 c_a[v](x)z^{-a}+O_R(|z|^{-4}).
\end{equation}
The constants and high-energy threshold can be chosen uniformly on
$\mathcal B_R$: the finite diagonalization uses bounded derivatives
through order five, and its Volterra remainder is controlled by
$C_R/(|\operatorname{Im} z||z|^5)$ with a contraction constant at most $1/4$ after a
common enlargement of the threshold. More explicitly, with
$\eta=|\operatorname{Im}z|$, the diagonalized system satisfies
\[
|\operatorname{Re}(D_5)_{jj}|\le C_R\eta|z|^{-2},\qquad
 \|R_5\|_\infty\le C_R|z|^{-5},\qquad
 \eta^{-1}\|R_5\|_\infty\le C'_R|z|^{-4}.
\]
The last inequality uses $\eta>a_R\langle\operatorname{Re}z\rangle^{-1}$.
The diagonal gap and the contraction threshold are thus uniform on the
initial-data ball. The coefficients $c_a$ are
differential polynomials through order $a-1$. Consequently there are
common constants $a_R,L_R$ such that the two Weyl coordinates are
analytic on
\[
\mathcal U_\pm
 =\{z:\ \pm\operatorname{Im} z>a_R\br{\operatorname{Re} z}^{-1},\ |z|>L_R\},
\]
and \eqref{E5.3} is uniform there, for all basepoints $x$.
Increase $L_R$ so that $|m_v(x,z)|\le1/4$ on these domains.

Choose one smooth height $\omega_R$ strictly above
$a_R\br{a}^{-1}$, with $\omega_R(a)\ge2L_R$ for $|a|\le L_R$, and
$\omega_R(a)=2a_R\br{a}^{-1}$ for sufficiently large $|a|$.
Its closed exterior lies in $\mathcal U_+\cup\mathcal U_-$.
For the canonical symbols on this contour,
\[
|\det M_\kappa(v,z)|=|1-\kappa m_v(z)\overline{m_v}(z)|\ge15/16.
\]
Thus the same contour is valid for every $v\in\mathcal B_R$, for either
sign. No individual selection of spectral zeros is involved.

\smallskip\noindent\textit{Continuity of the Weyl data.}
Suppose $v_n\to v$ in $\mathcal B_R$. In a sufficiently low half-plane,
the right Weyl quotient is the bounded small solution of
\[
m'=2izm-i\kappa\bar v-i v m^2.
\]
It is the fixed point of
\begin{equation}\label{E5.4}
 (\mathcal V_v m)(x)
 =-\int_x^\infty e^{2iz(x-y)}
 \bigl(-i\kappa\overline{v(y)}-iv(y)m(y)^2\bigr)\dd y.
\end{equation}
For $-\operatorname{Im} z>4R+1$, this is a common strict contraction on
$\norm m_\infty\le1/2$. Indeed its image norm is at most
$5R/(8|\operatorname{Im} z|)$ and its Lipschitz constant at most
$R/(2|\operatorname{Im} z|)$. Picard iterates starting at zero converge geometrically,
uniformly in $v$. At each finite iteration, local convergence of $v_n$
and the common exponential tail in \eqref{E5.4} give local uniform
convergence of the iterates. Hence $m_{v_n}\to m_v$ locally uniformly in
that half-plane. Reversing the integration direction proves the left
half-line assertion in a sufficiently high half-plane.

On $\mathcal U_\pm$ the family is bounded by $1/4$. Montel's theorem and
the identity theorem therefore extend the convergence to both common
domains: every subsequential analytic limit agrees with $m_v$ on the
distant half-plane, and each $\mathcal U_\pm$ is connected. In particular
there is uniform convergence on every compact portion of the fixed
contour. The coefficients in \eqref{E5.3} at $x=0$ converge
as well, since they depend on finitely many local derivatives.

\smallskip\noindent\textit{Continuity in a weighted symbol norm.}
Fix $b\in D_-$ at positive distance from $\overline D_+$. Choose matrices
$B_a(v)$ so that
\[
F_v(z)=I+\sum_{a=1}^3 B_a(v)(z-b)^{-a}
\]
matches the three Laurent coefficients of $M_\kappa(v,z)$. This is a fixed
triangular linear system. Thus $F_v$ depends continuously on $v$ in
$L^\infty(D_+)$, and, with $E_v=M_\kappa(v)-F_v$,
\begin{equation}\label{E5.5}
 \sup_{v\in\mathcal B_R,z\in C}|z|^4\norm{E_v(z)}<\infty.
\end{equation}
Define
\[
\norm{(F,E)}_{\mathrm{sym}}
 =\norm F_{L^\infty(D_+)}
  +\norm{z^3E}_{L^\infty(C)}+\norm{z^3E}_{L^2(C)}.
\]
Compact-contour convergence and \eqref{E5.5} imply
\begin{equation}\label{E5.6}
 (F_{v_n},E_{v_n})\longrightarrow(F_v,E_v)
 \quad\hbox{in this norm}.
\end{equation}
Indeed, on the tail of the fixed contour,
\[
\sup_{|z|>A}|z|^3\norm{E_{v_n}(z)-E_v(z)}\le C_R A^{-1},
 \qquad
 \norm{\mathbf1_{\{|z|>A\}}z^3(E_{v_n}-E_v)}_{L^2(C)}
 \le C_R A^{-1/2}.
\]
The second bound uses the bounded slope of the contour.
On its compact part, uniform convergence gives convergence in both
norms. First let $n\to\infty$ and then $A\to\infty$.
Each symbol still has the $O(z^{-4})$ remainder required by the
$N=2$ spectral construction. This weaker continuity topology controls
both the operator norm and the differentiated coefficient integrals.

\smallskip\noindent\textit{Normalized Toeplitz operators.}
Put $G_{s,t}=\diag(e^{i(sz+tz^2)},e^{-i(sz+tz^2)})$ and
\[
S(v,s,t)=G_{s,t}^{-1}T(G_{s,t}M_\kappa(v)).
\]
Both multipliers are uniformly bounded in $D_+$ when $(s,t)$ stays in
a compact rectangle. Since $\inf_C|z|>0$, for $1\le j\le3$,
\begin{equation}\label{E5.7}
 \norm{\mathfrak p_+(Eu)}_{H_j}
 \le C_j\norm{Eu}_{L^2(C)}
 \le C_j\sup_C|z|^j\norm E\,\norm u_{H_j}.
\end{equation}
It follows that $S$ is norm-continuous in $v$, locally uniformly in
$(s,t)$.

To check parameter continuity without assuming norm continuity of an
oscillatory multiplier, write
\[
S(v,s,t)=T(M_\kappa(v))+
 G_{s,t}^{-1}[\mathfrak p_+,G_{s,t}]E_v.
\]
Truncate $E_v(\lambda)$ to $|\lambda|\le A$. By \eqref{E5.7}, the
discarded commutator has norm tending to zero uniformly on compact
parameter rectangles. The truncated kernel is
\[
\frac{G_{s,t}(z)^{-1}G_{s,t}(\lambda)-I}{\lambda-z}
 E_v(\lambda)\mathbf1_{|\lambda|\le A}.
\]
On compact sets its diagonal singularity is removable. On the $z$ tail,
its weighted magnitude is $O(|z|^{-j-1})$ for bounded $\lambda$.
Dominated convergence therefore gives Hilbert--Schmidt, hence operator
norm, continuity for the truncation. Passing to $A\to\infty$ proves
joint norm continuity of $S$ on each $H_j$.

Every $S$ is invertible by input~\ref{input:nonzero}; focusing positivity
is also given by Lemma~\ref{l33}. Continuity of inversion and
compactness give
\[
\max_{\substack{v\in\mathcal B_R,\ |s|\le1,\ |t|\le T\\1\le j\le3}}
 \norm{S(v,s,t)^{-1}}_{H_j\to H_j}<\infty.
\]
The image is a norm-compact family of invertible bounded operators.
Multiplication by $G_{s,t}^{-1}$ is strongly continuous on $H_j$ and
locally uniformly bounded. Thus
\[
T(G_{s,t}M_\kappa(v))^{-1}=S(v,s,t)^{-1}G_{s,t}^{-1}
\]
is jointly strongly continuous, with the same local uniform bounds.

\smallskip\noindent\textit{Coefficient evaluation and spatial derivatives.}
Set $U=T(G_{s,t}M_\kappa(v))^{-1}I$. The reconstruction coefficient is
\begin{equation}\label{E5.8}
 A_1(v,s,t)=\frac1{2\pi i}\int_C G_{s,t}(\lambda)E_v(\lambda)
 U(v,s,t)(\lambda)\dd\lambda.
\end{equation}
It is jointly continuous. Differentiation in $s$ uses multiplication by
$z$ and raises a Hardy weight by one; differentiation in $t$ uses $z^2$
and raises it by two. More explicitly, with $T=T(G_{s,t}M_\kappa(v))$,
\[
U_s=-T^{-1}T_sU,\qquad
 U_{ss}=2T^{-1}T_sT^{-1}T_sU-T^{-1}T_{ss}U,
\]
where $T_s=i\sigma_3T(zG_{s,t}M_\kappa(v))$ and
$T_{ss}=-T(z^2G_{s,t}M_\kappa(v))$.
These formulas hold successively in $H_2,H_3$. The multiplier terms and
their derivatives are strongly continuous between the indicated
weighted spaces, with uniform operator bounds; the remainders are
controlled by \eqref{E5.6}. Hence $U,U_s,U_{ss}$ depend
jointly continuously on $(v,s,t)$ in $H_1,H_2,H_3$, respectively.

Each term obtained by differentiating \eqref{E5.8} to total
spectral degree $d\le2$ has the form
\[
\int_C \lambda^aG_{s,t}E_v V\dd\lambda,
 \qquad V\in H_{1+d-a},\quad0\le a\le d.
\]
Its absolute value is at most
\[
C\norm{\lambda^{1+d}E_v}_{L^2(C)}\norm V_{H_{1+d-a}}.
\]
The first norm and its tails are controlled uniformly by
\eqref{E5.5}. These estimates also justify differentiation
under the integral. It follows that
\[
a_j(v,t):=\partial_x^jQ_\kappa[v](t,0),\qquad0\le j\le2,
\]
are jointly continuous on $\mathcal B_R\times[-T,T]$.

Finally, translation covariance \eqref{E5.1} gives
\[
\partial_x^jQ_\kappa[v](t,y)=a_j(\T_yv,t).
\]
All translates $\T_yv$ remain in $\mathcal B_R$. The maximum of the
finitely many $|a_j|$ on this compact set is finite, proving
\eqref{E5.2}. Taking maxima over the signs makes the bound
common. Taking integer upper envelopes in $R,T$ makes it nondecreasing.
Uniform continuity of $a_j$ on $\mathcal B_R\times[-T,T]$ and the same
translation identity also give continuity in time in $W^{2,\infty}$.
The equation makes $Q_t$ continuous in $L^\infty$. Integrating the
pointwise equation with this uniform bound proves the asserted
$C^1(\R;L^\infty)$ regularity.
\end{proof}

\subsection{Uniqueness without a decay assumption}

\begin{lemma}\label{l38}
Let $u,v$ be classical solutions of (\ref{EE2.1}) on $[-T,T]\times\R$.
Assume $u,v,u_x,v_x$ are bounded there. If $u(0)=v(0)$, then $u=v$.
\end{lemma}
\begin{proof}
Put $w=u-v$, $N=|u|^2u-|v|^2v$,
$M=\max(\norm u_\infty,\norm v_\infty)$ and $D=\norm{w_x}_\infty$,
where these norms are on the time strip. Then $|N|\le3M^2|w|$.
For $L\ge1$ set
\[
\rho_L(x)=e^{-\sqrt{1+x^2}/L},\qquad
 E_L(t)=\int_\R\rho_L|w(t,x)|^2\dd x.
\]
We have $|\rho_L'|\le\rho_L/L$ and $\int\rho_L\le2L$.
Integration by parts in $w_t=iw_{xx}/2-i\kappa N$ gives
\begin{align*}
 |E_L'(t)|
 &\le\int|\rho_L'||w_x||w|+6M^2E_L(t)\\
 &\le D\sqrt{2/L}\,E_L(t)^{1/2}+6M^2E_L(t)\\
 &\le(6M^2+1)E_L(t)+C D^2/L.
\end{align*}
To justify this with only local classical regularity, first insert a
compact cutoff. After integration by parts, all discarded terms involve
bounded $w,w_x$ and the exponential tail, so they vanish as the cutoff
radius tends to infinity. The inequality therefore holds in integrated
form. Gronwall forward and backward from $E_L(0)=0$ yields
$E_L(t)\le C_{M,T}D^2/L$. On every bounded interval $I$,
$\inf_I\rho_L\to1$. Letting $L\to\infty$ gives
$\int_I|w(t,x)|^2\dd x=0$. Continuity proves the claim.
\end{proof}

\subsection{Modulation-space estimates and local solutions}
\label{sec:modulation}

\begin{lemma}\label{l39}
For every integer $j\ge0$,
\[
X_j\hookrightarrow C_b^j(\R),\qquad
 W^{j+2,\infty}(\R)\hookrightarrow X_j.
\]
In particular $p_5(f)\le C\norm f_5$ and
$\norm f_0\le C p_2(f)$.
\end{lemma}
\begin{proof}
Choose $\widetilde\chi\in C_c^\infty((-2,2))$ equal to one on
$\supp\chi$. Bernstein's convolution estimate gives
\[
\norm{\partial_x^a\Box_k f}_\infty
 \le C_a\br{k}^a\norm{\Box_k f}_\infty.
\]
Summing proves the first embedding, including uniform convergence of
all derivative series through order $j$.

For $|k|\ge3$ and integer $m\ge1$,
\[
\Box_k f=\mathcal F^{-1}
 \left(\frac{\chi(\xi-k)}{(i\xi)^m}\right)*\partial_x^m f.
\]
The kernel has $L^1$ norm at most $C_m\br{k}^{-m}$: after shifting the
frequency by $k$, its symbol and its first two derivatives have this
bound on a fixed compact support, and two integrations by parts give
an integrable kernel bound. The finitely many remaining blocks are
bounded by $C\norm f_\infty$. Taking $m=j+2$ and summing
$\br{k}^{j-m}$ proves the second embedding.
\end{proof}

\begin{lemma}\label{l40}
For every $s\ge0$,
\begin{align}
 \norm{fgh}_s
 &\le C_s\bigl(\norm f_s\norm g_0\norm h_0+
 \norm f_0\norm g_s\norm h_0+
 \norm f_0\norm g_0\norm h_s\bigr),\label{E5.9}\\
 \norm{|f|^2f-|g|^2g}_s
 &\le C_s(\norm f_s+\norm g_s)^2\norm{f-g}_s.
 \label{E5.10}
\end{align}
Complex conjugation preserves each $X_s$.
\end{lemma}
\begin{proof}
Expand $f,g,h$ into their uniformly absolutely convergent block series.
The Fourier support of $(\Box_k f)(\Box_l g)(\Box_m h)$ is contained
in the interval with center $k+l+m$ and radius $3$. Thus its $j$th block
vanishes unless $|j-k-l-m|\le4$, and each nonzero block has norm at most
$C\norm{\Box_kf}_\infty\norm{\Box_lg}_\infty\norm{\Box_mh}_\infty$.
For these indices,
$\br j^s\le C_s(\br k^s+\br l^s+\br m^s)$.
Summation proves \eqref{E5.9}. Evenness of $\chi$ gives
$\Box_k\bar f=\overline{\Box_{-k}f}$. Expanding the difference of the
cubic terms proves \eqref{E5.10}.
\end{proof}

\begin{lemma}\label{l41}
The Schr\"odinger group $U(t)=e^{it\partial_x^2/2}$ is strongly
continuous on every $X_s$, and
\begin{equation}\label{E5.11}
 \norm{U(t)f}_s\le C(1+|t|)^2\norm f_s.
\end{equation}
\end{lemma}
\begin{proof}
On a block centered at $k$, factor the multiplier as
\[
e^{-it(k+\eta)^2/2}
 =e^{-itk^2/2}e^{-itk\eta}e^{-it\eta^2/2}.
\]
The first two factors give a scalar of modulus one and a spatial
translation. The last factor times $\widetilde\chi(\eta)$ has an
$L^1$ convolution kernel bounded by $C(1+|t|)^2$, by the same
two-derivative estimate as above. This is uniform in $k$. Commutation
with $\Box_k$ proves \eqref{E5.11}.

For a finite block sum, strong continuity follows by differentiating
the compact-frequency multiplier, or by integrating its bounded second
spatial derivative. For general $f\in X_s$, finite block sums converge
in $X_s$: finite overlap gives
\[
\norm{\sum_{|k|>K}\Box_k f}_s
 \le C_s\sum_{|k|>K}\br{k}^s\norm{\Box_k f}_\infty\longrightarrow0.
\]
Uniform boundedness on compact time intervals then gives strong
continuity for every $f$.
\end{proof}

\begin{proposition}\label{p42}
For $s\ge0$, the integral equation
\begin{equation}\label{E5.12}
 u(t)=U(t)q_0-i\kappa\int_0^tU(t-a)(|u(a)|^2u(a))\dd a
\end{equation}
is locally well posed in $X_s$, with an existence time depending only
on $\norm{q_0}_s$. The maximal solution can end at finite time only
if its $X_s$ norm becomes unbounded. For $s\ge2$, the solution is
classical and belongs to $C^1(X_{s-2})$.
\end{proposition}
\begin{proof}
On $C([-\delta,\delta];X_s)$, Lemmas~\ref{l40}--\ref{l41}
make the right-hand side of \eqref{E5.12} a contraction on a ball
of radius $2C\norm{q_0}_s$ whenever
$\delta C_s(1+\norm{q_0}_s)^2$ is sufficiently small. This gives
existence, uniqueness and local Lipschitz dependence. The argument
restarts at every time with the same norm bound; if a finite maximal
endpoint had bounded norm, restarting sufficiently close to it would
extend the solution past the endpoint. Finally
$\partial_x^2:X_s\to X_{s-2}$ is bounded, so differentiation of the
integral equation, first for finite block sums and then by strong
convergence in $X_{s-2}$, gives the classical equation and the asserted
time regularity.
\end{proof}

\subsection{Global continuation and norm estimates}

\begin{proof}[Proof of Theorem~\ref{t3}]
Let $r=\norm{q_0}_5$. Lemma~\ref{l39} gives
$p_5(q_0)\le C_Er$. Proposition~\ref{p37} supplies a global
spectral solution $Q$ satisfying
\begin{equation}\label{E5.13}
 \sup_{|t|\le T}\norm{Q(t)}_0
 \le C_E B_T(C_Er)=:b_T(r).
\end{equation}
Let $u$ be the maximal local $X_5$ solution. On every compact interval
strictly inside its lifespan, $u,u_x$ are bounded; the same is true of
$Q$ by Proposition~\ref{p37}. Lemma~\ref{l38} therefore
gives $u=Q$ there. In particular \eqref{E5.13} applies to the
local solution without presupposing its global $X_5$ existence.

Write $A_T=C(1+2T)^2$, enlarged so that \eqref{E5.11} holds whenever
$|t|\le2T$. For $0\le t\le T$ within the lifespan,
\[
\norm{u(t)}_5
 \le A_Tr+C_5A_T b_T(r)^2\int_0^t\norm{u(a)}_5\dd a.
\]
Hence, for positive and negative times,
\[
\norm{u(t)}_5
 \le A_Tr\exp\bigl(C_5A_T b_T(r)^2T\bigr)=:F_T(r).
\]
This excludes a finite maximal endpoint in either direction by
Proposition~\ref{p42}. Thus $u=Q$ globally and
\eqref{E1.3} holds.

For two initial values in the $X_5$ ball of radius $R$, apply
\eqref{E5.10} to the difference of their integral equations.
The preceding bound and Gronwall yield
\[
\sup_{|t|\le T}\norm{u(t)-v(t)}_5
 \le A_T\exp\bigl(C_5A_TF_T(R)^2T\bigr)\norm{u(0)-v(0)}_5.
\]
This proves \eqref{EE2.2}. Time translation, spatial
translation and uniqueness give the group law, its inverse $\Phi_{-t}$,
and covariance. Strong time continuity and the locally uniform
Lipschitz estimate give joint continuity of the group action.

If $q_0\in X_s$, $s\ge5$, the local $X_s$ solution agrees with the
global $X_5$ solution. The tame estimate gives, on its lifespan,
\[
\sup_{|t|\le T}\norm{q(t)}_s
 \le A_T\norm{q_0}_s
 \exp\bigl(C_sA_T b_T(\norm{q_0}_5)^2T\bigr).
\]
The same continuation and difference arguments prove the asserted
$X_s$ results.

Finally, the embeddings in Lemma~\ref{l39} give
\[
C_b^\infty(\R)=\bigcap_{j\in\N}X_j
\]
with equivalent Fr\'echet topologies: $p_j\le C_j\norm\cdot_j$ and
$\norm\cdot_j\le C_jp_{j+2}$. The estimates just obtained prove
continuous dependence in every seminorm and preservation of this
intersection. The equation then gives all time derivatives in the same
Fr\'echet space by repeated differentiation. Lemma~\ref{l38}
establishes the stated broader PDE uniqueness class.
\end{proof}

\subsection{Approximation, almost periodicity and frequency modules}

\begin{lemma}\label{l43}
If $f\in X_s\cap\AP(\R)$, then its spatial translation orbit has
compact closure in $X_s$.
\end{lemma}
\begin{proof}
Every block $\Box_kf$ is almost periodic, since convolution by an
$L^1$ kernel preserves uniform limits of trigonometric polynomials.
For a fixed finite set of blocks, their translation orbits are jointly
precompact in the uniform norm. On a fixed compact frequency support
the $X_s$ norm is bounded by a constant times the uniform norm.
Thus each finite block sum has precompact translation orbit in $X_s$.
The estimate for block tails in Lemma~\ref{l41} is invariant
under translation. Approximating the entire orbit by these finite
block-sum orbits proves total boundedness. Completeness of $X_s$
proves the assertion.
\end{proof}

\begin{proof}[Proof of Theorem~\ref{t8}]
Let $\mathcal H_X(q_0)$ be the compact hull given by
Lemma~\ref{l43}. Theorem~\ref{t3} and covariance give
\[
\{\T_yq(t):y\in\R\}
 =\{\Phi_t(\T_yq_0):y\in\R\}
 \subseteq\Phi_t(\mathcal H_X(q_0)).
\]
The right-hand side is compact in $X$, hence in the uniform norm.
The compact-hull characterization of Bohr almost periodicity shows
that $q(t)$ is almost periodic. Continuity and approximation by orbit
points show
\[
\Phi_t(\mathcal H_X(q_0))=\mathcal H_X(q(t)).
\]
The inverse is the restriction of $\Phi_{-t}$.

We include a proof of exact preservation of the frequency module.
For $a\in\R$, define the continuous evaluation
$E_a:\mathcal H_X(q_0)\to\C$ by $E_a(p)=p(a)$.
The unital algebra generated by all $E_a$ and their conjugates separates
points of the compact hull. By the complex Stone--Weierstrass theorem,
it is dense in the continuous functions on that hull. In particular,
the continuous function $p\mapsto\Phi_t(p)(0)$ is uniformly approximable
by polynomials in finitely many $E_a,\overline{E_a}$.

Restricting these polynomials to $p=\T_xq_0$ gives polynomials in
$q_0(x+a)$ and their conjugates. By the Bohr approximation theorem,
$q_0$ is uniformly approximable by trigonometric polynomials with
frequencies in its Bohr spectrum. Products and conjugates therefore
have frequencies in $\Gamma(q_0)$. Each Fourier coefficient is a
continuous functional of the uniform norm; hence uniform limits still
have zero coefficients at frequencies outside $\Gamma(q_0)$.
It follows that $\Gamma(q(t))\subseteq\Gamma(q_0)$. Apply this same
inclusion to $q(t)$ and the inverse flow $\Phi_{-t}$ to obtain equality.
\end{proof}

\begin{corollary}\label{c44}
Let $q_0\in X$, and let $q_{0,n}\to q_0$ in $X$. Then
\[
\sup_{|t|\le T}\norm{\Phi_t(q_{0,n})-\Phi_t(q_0)}_5\to0
 \qquad(T<\infty).
\]
In particular this applies to finite block regularizations
$q_{0,n}=\sum_{|k|\le n}\Box_kq_0$, which belong to $C_b^\infty$.
For almost periodic $q_0$, it also applies to trigonometric polynomial
approximations converging in $X$.
\end{corollary}
\begin{proof}
The first assertion is \eqref{EE2.2}. Finite block
regularizations converge by the summable-tail estimate and are smooth
with all derivatives bounded. For the last assertion, first truncate
the block series. On its fixed compact frequency support, apply a
slightly larger smooth frequency cutoff to uniformly approximating
Bohr polynomials. This gives trigonometric polynomials converging to
the truncation in $X$. A diagonal choice gives convergence to $q_0$.
\end{proof}

\subsection{Local stability and approximation on expanding intervals}

\begin{proposition}\label{p45}
Let $u,v$ be classical solutions of \eqref{EE2.1} on
$[-T,T]\times\R$, with jointly continuous spatial derivatives through
order two, and suppose
\[
\sup_{|t|\le T}\max\{p_2(u(t)),p_2(v(t))\}\le B.
\]
For $L\ge1$,
\[
\sup_{|t|\le T}
 \norm{u(t)-v(t)}_{H^1([-L/2,L/2])}
 \le C_{B,T}\left(
 \norm{u(0)-v(0)}_{H^1([-L,L])}+L^{-1/2}\right).
\]
For $L\ge2A+2$, the same right-hand side bounds
$\sup_{|t|\le T}\norm{u(t)-v(t)}_{L^\infty([-A,A])}$.
\end{proposition}

\begin{proof}
Put $w=u-v$ and $N=|u|^2u-|v|^2v$. Choose
$\chi\in C_c^\infty((-1,1);\R)$ such that $0\le\chi\le1$ and
$\chi=1$ on $[-1/2,1/2]$, and write
\[
\rho_L(x)=\chi(x/L)^2,\qquad
 F_L(t)=\int_\R\rho_L(|w|^2+|w_x|^2).
\]
The identity
\[
N_x=2|u|^2w_x+u^2\overline{w_x}
 +2(|u|^2-|v|^2)v_x+(u^2-v^2)\overline{v_x}
\]
gives
$|N|\le3B^2|w|$ and
$|N_x|\le3B^2|w_x|+6B^2|w|$. Moreover,
\[
\int_{\{\rho_L>0\}}\frac{|\rho_L'|^2}{\rho_L}
 \le4L^{-1}\norm{\chi'}_{L^2}^2.
\]
In the localized mass identities for $w$ and $w_x$, the term
$2|u|^2w_x\overline{w_x}$ is real. The nonlinear contribution is
therefore bounded pointwise by
\[
B^2\bigl(6|w|^2+2|w_x|^2+12|w||w_x|\bigr)
 \le12B^2(|w|^2+|w_x|^2).
\]
The two identities yield
\begin{align*}
 |F_L'|
 &\le\int|\rho_L'|(|w_x||w|+|w_{xx}||w_x|)
       +12B^2F_L\\
 &\le 2L^{-1/2}\norm{\chi'}_{L^2}
       \bigl(\norm{w_x}_\infty^2+\norm{w_{xx}}_\infty^2\bigr)^{1/2}
       F_L^{1/2}+12B^2F_L\\
 &\le(12B^2+1)F_L+8B^2\norm{\chi'}_{L^2}^2L^{-1}.
\end{align*}
For the differentiated identity, convolve
$(w_x)_t=iw_{xxx}/2-i\kappa N_x$ in space.
Integration by parts leaves only the mollified
$w_x,w_{xx},N_x$. Local uniform convergence on the compact support
of $\rho_L$ gives the integrated identity as the mollifier is removed.
Thus $F_L$ is absolutely continuous and the displayed estimate holds
almost everywhere.

Gronwall in both time directions gives
\[
F_L(t)\le e^{a|t|}F_L(0)
       +\frac{8B^2\norm{\chi'}_{L^2}^2}{aL}(e^{a|t|}-1),
 \qquad a=12B^2+1.
\]
Use $\rho_L=1$ on $[-L/2,L/2]$ and
$F_L(0)\le\norm{w(0)}_{H^1([-L,L])}^2$.
For the pointwise estimate, apply the fixed-interval Sobolev
inequality on $[x-1,x+1]\subset[-L/2,L/2]$.
\end{proof}

\begin{corollary}
\label{c46}
Let $f\in W^{5,\infty}(\R)$ with $p_5(f)\le R$, and choose
$\theta\in C_c^\infty((-2,2))$ equal to one on $[-1,1]$.
Define
\[
f_L^{\rm cut}(x)=\theta(x/L)f(x),
\]
and let $f_L^{\rm per}$ be its $4L$-periodic extension from
$[-2L,2L]$. Let $q_L^{\rm cut}$ and $q_L^{\rm per}$ be the
corresponding finite-mass and periodic cubic NLS solutions.
For either sign and $L\ge\max\{1,2A+2\}$,
\[
\begin{aligned}
 &\sup_{|t|\le T}\left(
 \norm{q_L^\alpha(t)-Q_\kappa[f](t)}_{H^1([-A,A])}
 +\norm{q_L^\alpha(t)-Q_\kappa[f](t)}_{L^\infty([-A,A])}
 \right)
 \le C_{R,T,\theta}L^{-1/2},\\
 &\hfill \alpha\in\{\mathrm{cut},\mathrm{per}\}.
 \end{aligned}
\]
\end{corollary}

\begin{proof}
Leibniz's rule gives
$p_5(f_L^\alpha)\le C_\theta R$ uniformly in $L\ge1$.
The periodic seams are smooth because $\theta$ vanishes in
neighborhoods of $\pm2$. Both data agree with $f$ on $[-L,L]$.
The finite-mass and periodic data belong to $H^5$ on their respective
domains. Their usual global one-dimensional cubic NLS solutions,
for both signs, have bounded first derivatives on every finite
time strip; see \cite{Tsutsumi,Bourgain93} and persistence of regularity.
Lemma~\ref{l38} identifies them with the canonical solutions.
Proposition~\ref{p37}, in its $W^{5,\infty}$ form, supplies
the same $p_2$ bound for these solutions and $Q_\kappa[f]$,
independently of $L$. Apply Proposition~\ref{p45}.
\end{proof}

\subsection{Local topology and translation hulls}

\begin{proposition}
\label{p47}
If $f_n,f\in X_5$, $\sup_n\norm{f_n}_5+\norm f_5<\infty$, and
$f_n\to f$ in $C^4_{\rm loc}(\R)$, then, for every $A,T<\infty$,
\[
\sup_{|t|\le T}
 \norm{\Phi_t(f_n)-\Phi_t(f)}_{C^4([-A,A])}\longrightarrow0.
\]
\end{proposition}

\begin{proof}
Theorem~\ref{t3} and Lemma~\ref{l39} give a common
bound for $p_5(\Phi_t(f_n))$ and $p_5(\Phi_t(f))$ on $[-T,T]$.
For a fixed large $L$, Proposition~\ref{p45} applies,
and its initial $H^1([-L,L])$ difference tends to zero.
First let $n\to\infty$, and then $L\to\infty$.
The resulting convergence is uniform on $[-T,T]$ with values
in $L^\infty([-A-1,A+1])$.

For clarity, choose a smooth spatial cutoff equal to one on
$[-A,A]$ and supported in $[-A-1,A+1]$.
Apply the one-dimensional interpolation inequalities to its
product with $\Phi_t(f_n)-\Phi_t(f)$:
\[
\norm{\partial_x^j h}_\infty
 \le C_j\norm h_\infty^{1-j/5}
             \norm{\partial_x^5h}_\infty^{j/5},
 \qquad 1\le j\le4.
\]
The fifth derivatives of these products are uniformly bounded.
This proves the asserted convergence of every derivative through
order four, uniformly in time.
\end{proof}

\begin{proposition}
\label{p48}
For $f\in X_5$, let
\[
\mathcal H_{\rm loc}(f)
 =\overline{\{\T_yf:y\in\R\}}^{\,C^4_{\rm loc}(\R)},
 \qquad \T_yf(x)=f(x+y).
\]
This is a compact subset of
$\{g\in X_5:\norm g_5\le\norm f_5\}$.
For every $t\in\R$, the flow restricts to a homeomorphism
\[
\Phi_t:\mathcal H_{\rm loc}(f)
       \longrightarrow\mathcal H_{\rm loc}(\Phi_t(f))
\]
for the local $C^4$ topologies and commutes with all translations.
\end{proposition}

\begin{proof}
The $p_5$ bound makes every bounded $X_5$ ball precompact in
$C^4_{\rm loc}$ by Arzel\`a--Ascoli.
Such a ball is closed for this topology. Indeed, if
$g_n\to g$ locally with $\norm{g_n}_5\le R$, the uniform
$L^\infty$ bound and the Schwartz tails of each block kernel give
$\Box_kg_n\to\Box_kg$ locally uniformly.
Hence
\[
\norm{\Box_kg}_\infty
 \le\liminf_n\norm{\Box_kg_n}_\infty,\qquad
 \norm g_5\le\liminf_n\norm{g_n}_5\le R
\]
by Fatou's lemma.
The translation orbit stays in this ball, proving compactness of
its local hull.

Proposition~\ref{p47} and translation
covariance show that the compact image of this hull under $\Phi_t$
is exactly the local hull of $\Phi_t(f)$: approximate each point
by translates in both the inclusion and closure arguments.
The inverse is the restriction of $\Phi_{-t}$, which is continuous
on the bounded image by the same proposition.
\end{proof}

\subsection{The defocusing sum space}

We prove Theorem~\ref{t7}, using the relative-energy
argument of Klaus--Kunstmann \cite[Section~4]{KlausKunstmann} with the
global background bounds supplied by Theorem~\ref{t3}. First,
\begin{equation}\label{E5.14}
 \norm b_0\le C\norm b_{H^1}.
\end{equation}
Indeed, the unit-band Bernstein inequality and Cauchy--Schwarz give
\[
\sum_k\norm{\Box_kb}_\infty
 \le C\sum_k\norm{\Box_kb}_{L^2}
 \le C\left(\sum_k\br{k}^{-2}\right)^{1/2}
       \left(\sum_k\br{k}^{2}\norm{\Box_kb}_{L^2}^2\right)^{1/2}
 \le C\norm b_{H^1}.
\]
Thus $Z_s$ embeds continuously in $X_0$. It is a Banach space:
the kernel of the addition map
$X_s\times H^1\to X_0$ is closed, and
\eqref{EE2.3} is the corresponding quotient norm.

Fix $a\in X_s$ and put $q(t)=\Phi_t(a)$. Theorem~\ref{t3},
the embeddings of Lemma~\ref{l39}, and the equation imply
\begin{equation}\label{E5.15}
 q\in C(\R;W^{2,\infty}(\R))
       \cap C^1(\R;L^\infty(\R)).
\end{equation}
These norms are bounded on every finite time interval, uniformly
when $a$ ranges over a bounded subset of $X_s$.
Writing $u=q+w$, the equation for the perturbation is
\begin{equation}\label{E5.16}
 iw_t=-\tfrac12w_{xx}+\mathcal N(q,w),\qquad
 \mathcal N(q,w)=|q+w|^2(q+w)-|q|^2q,\qquad w(0)=b.
\end{equation}
Its polynomial expansion is
\[
\mathcal N(q,w)
 =|w|^2w+2|q|^2w+q^2\overline w
                  +2q|w|^2+\overline q\,w^2.
\]
Since $H^1(\R)$ is an algebra and multiplication by a
$W^{1,\infty}$ function is bounded on $H^1$, the Duhamel
map for \eqref{E5.16} gives a unique local
$C_tH^1$ solution, with the $H^1$ norm blowup alternative.
More precisely, for bounded $q,p$ in $W^{1,\infty}$ and
bounded $w,v$ in $H^1$,
\begin{equation}\label{E5.17}
 \norm{\mathcal N(q,w)-\mathcal N(p,v)}_{H^1}
 \le C_B\bigl(\norm{w-v}_{H^1}
                    +\norm{q-p}_{W^{1,\infty}}\bigr),
\end{equation}
where $B$ bounds the four indicated norms. This follows by
subtracting the five monomials and using the same two
multiplication estimates.

We next obtain a bound that holds for arbitrary $\norm b_{H^1}$.
Define
\begin{align}
 M(w)&=\norm w_{L^2}^2,\nonumber\\
 E(q,w)&=\tfrac12\norm{w_x}_{L^2}^2
                     +\tfrac12\int_\R R(q,w)\,\dd x,
 \label{E5.18}\\
 R(q,w)&=|w|^4+4\operatorname{Re}(q\overline w)|w|^2
          +2|q|^2|w|^2
          +4\bigl(\operatorname{Re}(q\overline w)\bigr)^2.
 \nonumber
\end{align}
Every term in the integral is integrable for $w\in H^1$.
Pointwise, $R(q,w)$ equals
$|q+w|^4-|q|^4-4\operatorname{Re}(|q|^2q\overline w)$.
If $w\ne0$, write $q/w=\alpha+i\beta$. Then
\[
\frac{R(q,w)}{|w|^4}
 =1+4\alpha+6\alpha^2+2\beta^2
 =\tfrac13+6(\alpha+\tfrac13)^2+2\beta^2.
\]
Consequently,
\begin{equation}\label{E5.19}
 E(q,w)\ge\tfrac12\norm{w_x}_{L^2}^2
                         +\tfrac16\norm w_{L^4}^4 .
\end{equation}

Fix $T<\infty$ and set
\[
Q_T=\sup_{|t|\le T}\norm{q(t)}_\infty,\qquad
 P_T=\sup_{|t|\le T}\norm{q_t(t)}_\infty.
\]
Initially take $b\in H^2$. The mass identity for
\eqref{E5.16} gives
\begin{equation}\label{E5.20}
 |M'|\le2Q_T^2M+2Q_T\int_\R|w|^3\,\dd x.
\end{equation}
The real variational derivative of $E$ with respect to $w$ is
$-w_{xx}+2\mathcal N(q,w)=2iw_t$. Its contribution to the
time derivative of $E$ has zero real part. Differentiating
the explicit $q$ dependence in \eqref{E5.18}
therefore yields
\begin{align*}
 E'
 &=2\int_\R\operatorname{Re}(q_t\overline w)|w|^2\,\dd x
   +2\int_\R\operatorname{Re}(\overline q\,q_t)|w|^2\,\dd x\\
 &\qquad
   +4\int_\R\operatorname{Re}(q\overline w)
                 \operatorname{Re}(q_t\overline w)\,\dd x,
\end{align*}
and hence
\begin{equation}\label{E5.21}
 |E'|\le2P_T\int_\R|w|^3\,\dd x+6Q_TP_TM.
\end{equation}
These identities are valid for
$w\in C_tH^2\cap C_t^1L^2$ by integration by parts and the
Hilbert-space chain rule; the background regularity in
\eqref{E5.15} suffices.

Let $\mathcal H=M+E$. By Cauchy--Schwarz and
\eqref{E5.19},
\[
\int_\R|w|^3\,\dd x
 \le M^{1/2}\norm w_{L^4}^2
 \le\tfrac12M+3E.
\]
Equations \eqref{E5.20} and
\eqref{E5.21} give
\[
|\mathcal H'|
 \le C(1+Q_T+P_T)^2\mathcal H.
\]
Thus, in both time directions,
\begin{equation}\label{E5.22}
 \sup_{|t|\le T}\norm{w(t)}_{H^1}^2
 \le2\mathcal H(0)
            \exp\!\bigl(CT(1+Q_T+P_T)^2\bigr).
\end{equation}
The initial value satisfies
\[
\mathcal H(0)
 \le C\left[(1+\norm a_\infty^2)\norm b_{H^1}^2
                                      +\norm b_{H^1}^4\right].
\]
This follows from the polynomial in
\eqref{E5.18}, the embedding $H^1\hookrightarrow L^\infty$,
and Young's inequality for $\norm a_\infty\norm b_{H^1}^3$.

To pass to all $b\in H^1$, approximate it in $H^1$ by
$b_m\in H^2$. On bounded $H^1$ sets, the usual product
estimate gives
\[
\norm{\mathcal N(q,w)}_{H^2}
 \le C\bigl(\norm q_{W^{2,\infty}}+\norm w_{H^1}\bigr)^2
                                              \norm w_{H^2}.
\]
Therefore \eqref{E5.22} and Gronwall
extend each $H^2$ solution through $[-T,T]$.
These solutions have a uniform $H^1$ bound. Applying
\eqref{E5.17} with the same background
shows that these solutions are Cauchy in $C([-T,T];H^1)$.
Their limit solves \eqref{E5.16}; the energy
and mass converge, so \eqref{E5.22}
passes to the limit. Arbitrary $T$ and local uniqueness give
a single global $H^1$ perturbation.

The bounds just proved are uniform for $(a,b)$ in a bounded
subset of $X_s\times H^1$. For two such pairs $(a,b)$ and
$(\widetilde a,\widetilde b)$, let
$q=\Phi_t(a)$, $p=\Phi_t(\widetilde a)$ and denote their
perturbations by $w,v$. Duhamel, the $H^1$ isometry of the
free Schr\"odinger group, and
\eqref{E5.17} imply
\[
\norm{w-v}_{C([-T,T];H^1)}
 \le e^{C_BT}\left(
       \norm{b-\widetilde b}_{H^1}
       +C_BT\norm{q-p}_{C([-T,T];W^{1,\infty})}\right).
\]
The stability assertion of Theorem~\ref{t3} therefore gives
\begin{equation}\label{E5.23}
 \norm{q-p}_{C([-T,T];X_s)}
 +\norm{w-v}_{C([-T,T];H^1)}
 \le C_{s,T,R}\bigl(
      \norm{a-\widetilde a}_s+\norm{b-\widetilde b}_{H^1}\bigr)
\end{equation}
when the two pairs have norm at most $R$.

It remains to check that this construction defines the asserted
flow on the sum space. By \eqref{E5.14},
$u=q+w$ belongs to $C(\R;X_0)$. Adding the two Duhamel
equations shows that it is an $X_0$ mild solution of
\eqref{EE2.1}. Proposition~\ref{p42}, at $s=0$, implies
that two decompositions of the same initial value give the same
solution. It also gives uniqueness in the class stated in the
theorem. The equation holds distributionally, since every term
in its Duhamel formula is well defined in $X_0$.

For $\norm{u_0}_{Z_s}\le R$, choose $u_0=a+b$ with
$\norm a_s+\norm b_{H^1}\le R+1$.
Theorem~\ref{t3} and
\eqref{E5.22} give
\eqref{E1.6}. To prove
\eqref{E1.7}, put
$d=\norm{u_0-v_0}_{Z_s}$ and, for $0<\epsilon<1$, choose
\[
u_0=a+b,\qquad
 v_0-u_0=c+e,\qquad
 \norm a_s+\norm b_{H^1}\le R+1,\qquad
 \norm c_s+\norm e_{H^1}\le d+\epsilon.
\]
Since $d\le2R$, the pairs $(a,b)$ and $(a+c,b+e)$ lie
in the product ball of radius $3R+2$. Apply
\eqref{E5.23}, then take the sum-space
norm and let $\epsilon$ decrease to zero. This proves the
claimed Lipschitz estimate in the quotient norm.

Finally, $q(t)\in X_s$ and $w(t)\in H^1$, so the trajectory
is continuous in $Z_s$. Uniqueness after restarting at any
time gives the group law and the inverse $\Phi_{-t}$.
Translation covariance follows from the same uniqueness and the
translation invariance of the equation.
For initial data in $X_s$ the choice $b=0$ recovers the
original flow. Formula \eqref{EE2.4}
and its joint stability follow from the construction and
\eqref{E5.23}. Applying
\eqref{EE2.4} to the inverse flow
gives equality of the two affine spaces in the theorem.

\subsection{Full spectrum and spatial transfer growth}

The uniform spatial bounds on finite time strips also control the
time evolution of the associated spectral equation. We use the
Euclidean operator norm for matrices. For a potential $q(t)$, write
\[
\mathsf Q(t,x)=
 \begin{pmatrix}0&q(t,x)\\
 \kappa\overline{q(t,x)}&0\end{pmatrix},
 \qquad
 D(t)=i\sigma_3\partial_x+\mathsf Q(t,\cdot),
 \qquad
 \operatorname{Dom}D(t)=H^1(\R;\C^2).
\]
These are closed operators on $L^2(\R;\C^2)$, since $\mathsf Q(t)$
is bounded. Let $\mathsf Y(t;x,y,z)$ be the transfer matrix determined by
\[
\partial_x\mathsf Y(t;x,y,z)
 =\mathsf U(t,x,z)\mathsf Y(t;x,y,z),\qquad
 \mathsf Y(t;y,y,z)=I,\qquad
 \mathsf U(t,x,z)=-iz\sigma_3+i\sigma_3\mathsf Q(t,x).
\]
Thus $D(t)f=zf$ is equivalent to $f_x=\mathsf U(t,x,z)f$.
For $\epsilon\in\{-1,1\}$ and a fixed base point $y$, define
\begin{align}
 \overline\lambda_\epsilon(t,z;y)
 &=\limsup_{L\to\infty}\frac1L
       \log\norm{\mathsf Y(t;y+\epsilon L,y,z)},\notag\\
 \underline\lambda_\epsilon(t,z;y)
 &=\liminf_{L\to\infty}\frac1L
       \log\norm{\mathsf Y(t;y+\epsilon L,y,z)}.
 \label{E5.24}
\end{align}
These definitions require no spatial recurrence or ergodic hypothesis.

\begin{theorem}
\label{t49}
Let $q$ be a global solution supplied by Theorem~\ref{t3}, for
either $\kappa=1$ or $\kappa=-1$. Then
\[
\sigma(D(t))=\sigma(D(0)),\qquad t\in\R,
\]
where $\sigma$ denotes the entire complex operator spectrum.
For every $z\in\C$, there is an invertible matrix
$\mathsf B(t,x,z)$ such that
\begin{equation}\label{E5.25}
 \mathsf Y(t;x,y,z)
 =\mathsf B(t,x,z)\mathsf Y(0;x,y,z)\mathsf B(t,y,z)^{-1}.
\end{equation}
For every $T,K<\infty$,
\begin{equation}\label{E5.26}
 \sup_{\substack{|t|\le T,\ x\in\R\\|z|\le K}}
 \bigl(\norm{\mathsf B(t,x,z)}
       +\norm{\mathsf B(t,x,z)^{-1}}
       +\norm{\partial_x\mathsf B(t,x,z)}\bigr)<\infty.
\end{equation}
Consequently, both growth rates in
\eqref{E5.24} are independent of $t$.
In particular, every spatial Lyapunov exponent defined by an existing
limit of these transfer norms is preserved.
Multiplication by $\mathsf B(t,\cdot,z)$ gives bijections between
the corresponding spaces of spectral solutions square integrable
at either end of the line, and between
$\ker(D(0)-z)$ and $\ker(D(t)-z)$.
For $z$ in the common resolvent set, the resolvent norms at times
$0$ and $t$ are comparable, with comparison constants uniform for
$|t|\le T$ and $|z|\le K$.
The same conclusions hold for every global classical solution whose
spatial derivatives through order two are jointly continuous and
bounded on each finite time strip.
\end{theorem}

\begin{proof}
Define the time matrix
\[
\mathsf V(t,x,z)
 =-iz^2\sigma_3+iz\sigma_3\mathsf Q(t,x)
   -\frac12\partial_x\mathsf Q(t,x)
   -\frac i2\sigma_3\mathsf Q(t,x)^2.
\]
The identities
$\sigma_3\mathsf Q=-\mathsf Q\sigma_3$ and
$\mathsf Q^2=\kappa|q|^2I$ give
\[
\mathsf U_t-\mathsf V_x+[\mathsf U,\mathsf V]
 =i\sigma_3\mathsf Q_t
      +\frac12\mathsf Q_{xx}-\mathsf Q^3=0.
\]
The last equality is precisely equation~\eqref{EE2.1}, together with
its conjugate. This fixes the sign and normalization of the time
matrix for both reductions.

For fixed $x,z$, solve the finite-dimensional equation
\begin{equation}\label{E5.27}
 \partial_t\mathsf B(t,x,z)
   =\mathsf V(t,x,z)\mathsf B(t,x,z),\qquad
 \mathsf B(0,x,z)=I.
\end{equation}
Theorem~\ref{t3} and Lemma~\ref{l39} bound
$q,q_x,q_{xx}$ uniformly in $x$ on every finite time strip.
The coefficients of \eqref{E5.27} and their first
$x$ derivatives are therefore uniformly bounded when $t,z$ stay
in compact sets. The integral equation, differentiation in $x$,
and Gronwall's inequality show that $\mathsf B$ is continuously
differentiable in $x$. Gronwall's inequality for $\mathsf B$ and
for $(\mathsf B^{-1})_t=-\mathsf B^{-1}\mathsf V$ gives uniform
bounds for these two matrices.

Compatibility gives the more precise identity
\begin{equation}\label{E5.28}
 \mathsf B_x
 =\mathsf U(t,x,z)\mathsf B
       -\mathsf B\mathsf U(0,x,z).
\end{equation}
Indeed, the difference of the two sides, written as
$\mathsf C=\mathsf B_x-\mathsf U(t)\mathsf B+
\mathsf B\mathsf U(0)$, satisfies
\[
\mathsf C_t=\mathsf V\mathsf C+
   \bigl(\mathsf V_x-\mathsf U_t-[\mathsf U,\mathsf V]\bigr)
       \mathsf B=\mathsf V\mathsf C,\qquad
 \mathsf C(0,x,z)=0.
\]
Uniqueness for this matrix equation proves
\eqref{E5.28} and the remaining bound in
\eqref{E5.26}.

For a bounded matrix function $C(x)$, let $M_C$ denote multiplication
by $C$ on $L^2(\R;\C^2)$. The bounds just proved, and
$(\mathsf B^{-1})_x=-\mathsf B^{-1}\mathsf B_x\mathsf B^{-1}$,
show that $M_{\mathsf B}$ maps $H^1$ bijectively onto itself.
Using \eqref{E5.28} on this common domain gives
\begin{equation}\label{E5.29}
 (D(t)-z)M_{\mathsf B}
 =M_{\sigma_3\mathsf B\sigma_3}(D(0)-z).
\end{equation}
For example, the terms containing no derivative of the test function
cancel after inserting
$i\sigma_3\mathsf U=zI-\mathsf Q$; the derivative term is
$i\sigma_3\mathsf B\partial_x$ on either side.
All multiplication operators in
\eqref{E5.29} are bounded and invertible
on $L^2$. It follows that $D(t)-z$ has a bounded inverse if and only
if $D(0)-z$ does. More explicitly,
\[
(D(t)-z)^{-1}
 =M_{\mathsf B}(D(0)-z)^{-1}
       M_{\sigma_3\mathsf B^{-1}\sigma_3}.
\]
This proves equality of the full spectra and the upper resolvent
bound. The inverse rearrangement gives the lower bound.

For the transfer matrices, differentiation of the right-hand side
of \eqref{E5.25} with respect to $x$,
using \eqref{E5.28}, gives the spectral equation at
time $t$. Its value at $x=y$ is $I$, so spatial ODE uniqueness
proves \eqref{E5.25}.
If $C\ge1$ bounds $\mathsf B$ and $\mathsf B^{-1}$ for the fixed
$t,z$, then, for all $x,y\in\R$,
\[
C^{-2}\norm{\mathsf Y(0;x,y,z)}
 \le\norm{\mathsf Y(t;x,y,z)}
 \le C^2\norm{\mathsf Y(0;x,y,z)}.
\]
The two logarithms differ by at most $2\log C$, uniformly in both
endpoints. Division by $L$ proves the assertions about
\eqref{E5.24}. The same comparison also applies
if a supremum or infimum over the base point is taken before the
upper or lower limit.

Finally, \eqref{E5.28} maps each spatial solution at
time zero to one at time $t$. The uniform bounds for $\mathsf B$
and its inverse preserve square integrability on either half-line
and on the whole line. A whole-line $L^2$ solution belongs to
$H^1$, because its derivative is $\mathsf U f$ and $\mathsf U$
is bounded. This proves the asserted kernel bijection.
\end{proof}

\begin{corollary}
\label{c50}
Let $q_0\in X_5$, and put
\[
\mathcal H_0=\mathcal H_{\rm loc}(q_0),\qquad
 \mathcal H_t=\mathcal H_{\rm loc}(\Phi_t(q_0)).
\]
Let $F_t=\Phi_t:\mathcal H_0\to\mathcal H_t$ be the homeomorphism
intertwining translations supplied by Proposition~\ref{p48}.
For a translation-invariant probability measure $\mu$ on
$\mathcal H_0$, put $\mu_t=(F_t)_*\mu$. For each $z\in\C$, define
\[
\mathcal L_\mu(z)
 =\lim_{L\to\infty}\frac1L
      \int_{\mathcal H_0}\log\norm{\mathsf Y_p(L,0,z)}
      \dd\mu(p),
\]
where $\mathsf Y_p$ is the transfer matrix of the potential $p$.
This limit exists, and
\[
\mathcal L_{\mu_t}(z)=\mathcal L_\mu(z).
\]
\end{corollary}

\begin{proof}
The intertwining identity implies that $\mu_t$ is translation invariant. For fixed $L,z$,
the matrix $\mathsf Y_p(L,0,z)$ depends continuously on $p$ in
$C^0_{\mathrm{loc}}$, by continuous dependence for its matrix
ODE; hence its logarithmic norm is a continuous hull observable.

For fixed $z$, let
$a_\mu(L)=\int\log\norm{\mathsf Y_p(L,0,z)}\dd\mu(p)$.
The transfer composition law and translation invariance give
$a_\mu(L+M)\le a_\mu(L)+a_\mu(M)$ for $L,M\ge0$.
The coefficient matrices are uniformly bounded on the compact
hull, so $a_\mu$ is bounded on every bounded interval.
Also $\det\mathsf Y_p=1$, whence $a_\mu(L)\ge0$.
Writing $L=n\ell+r$, $0\le r<\ell$, for each fixed $\ell>0$,
proves
\[
\lim_{L\to\infty}\frac{a_\mu(L)}L
 =\inf_{\ell>0}\frac{a_\mu(\ell)}\ell.
\]
Thus the displayed limit defining $\mathcal L_\mu$ exists.

The initial hull is contained in a bounded $X_5$ ball by
Proposition~\ref{p48}.
The finite-time estimates of Theorem~\ref{t3} therefore
make \eqref{E5.26} uniform over $p\in\mathcal H_0$.
The transfer comparison in the proof of
Theorem~\ref{t49} consequently bounds
\[
\left|\log\norm{\mathsf Y_{\Phi_t(p)}(L,0,z)}
       -\log\norm{\mathsf Y_p(L,0,z)}\right|
\]
by a constant independent of $p$ and $L$.
Integrate with respect to $\mu$, divide by $L$, and pass to the
limit to obtain the claimed equality.
\end{proof}

\end{document}